\documentclass[12pt,letterpaper,twoside,english]{amsart}
\usepackage[T1]{fontenc}
\usepackage[latin9]{inputenc}
\usepackage{mathrsfs}
\usepackage[headings]{fullpage}
\usepackage{comment}
\usepackage{amsthm}
\usepackage{amstext}
\usepackage{amssymb}
\usepackage{tikz}
\usepackage[bbgreekl]{mathbbol}
\usepackage{amsfonts}
\DeclareSymbolFontAlphabet{\mathbb}{AMSb}
\DeclareSymbolFontAlphabet{\mathbbl}{bbold}
\DeclareFontFamily{U}{wncyss}{}
\DeclareFontShape{U}{wncyss}{m}{n}{<->wncyss8}{}
\DeclareSymbolFont{mcy}{U}{wncyss}{m}{n}
\DeclareMathSymbol{\ra}{\mathord}{mcy}{"61}
\DeclareMathSymbol{\rb}{\mathord}{mcy}{"62}
\DeclareMathSymbol{\rd}{\mathord}{mcy}{"64}
\DeclareMathSymbol{\rg}{\mathord}{mcy}{"67}
\usepackage{graphicx}
\usepackage{setspace}
\usepackage{esint}
\usepackage{enumerate}
\usepackage{amscd,color}

\makeatletter

\pdfpageheight\paperheight
\pdfpagewidth\paperwidth

\numberwithin{equation}{section}
\numberwithin{figure}{section}
\theoremstyle{plain}
\newtheorem{thm}{\protect\theoremname}[section]
  \theoremstyle{definition}
  \newtheorem{example}[thm]{\protect\examplename}
  \theoremstyle{remark}
  \newtheorem{rem}[thm]{\protect\remarkname}
  \theoremstyle{plain}
  \newtheorem{prop}[thm]{\protect\propositionname}
  \theoremstyle{plain}
  \newtheorem{cor}[thm]{\protect\corollaryname}
  \theoremstyle{definition}
  \newtheorem{defn}[thm]{\protect\definitionname}
  \theoremstyle{plain}
  \newtheorem{lem}[thm]{\protect\lemmaname}
 \theoremstyle{remark}
  \newtheorem{conv}[thm]{\protect\conventionname}

\theoremstyle{plain}
\newtheorem{thmx}{Theorem}

\theoremstyle{plain}
\newtheorem{corx}{Corollary}

\usepackage{enumitem}
\usepackage{amsmath}
\usepackage{amsfonts}
\usepackage{amssymb}
\usepackage[all]{xy}
\usepackage{array}

\newcommand{\longhookrightarrow}{\lhook\joinrel\longrightarrow}

\def\RR{\mathbb{R}}
\def\CC{\mathbb{C}}
\def\QQ{\mathbb{Q}}
\def\PP{\mathbb{P}}
\def\ZZ{\mathbb{Z}}
\def\HH{\mathbb{H}}
\def\VV{\mathbb {V}}

\def\fx{\mathfrak{X}}
\def\cE{\mathcal{E}}
\def\cf{\mathcal{F}}
\def\cx{\mathcal{X}}

\def\ch{\mathcal{H}}

\def\cy{\mathcal{Y}}

\def\cn{\mathcal{N}}
\def\cb{\mathcal{B}}
\def\tcb{\tilde{\mathcal{B}}}
\def\co{\mathcal{O}}
\def\cv{\mathcal{V}}
\def\cu{\mathcal{U}}

\def\ppsi{{}^p \psi}
\def\pphi{{}^p \phi}

\def\gr{\mathrm{Gr}}

\def\can{\mathtt{can}}
\def\gy{\mathtt{gy}}
\def\sp{\mathtt{sp}}

\def\pha{\mathrm{ph}}
\def\van{\mathrm{van}}
\def\lm{\mathrm{lim}}

\def\IH{\mathrm{IH}}
\def\IC{\mathrm{IC}^{\bullet}}

\def\uw{\underline{w}}
\def\um{\underline{m}}
\def\utw{\tilde{\uw}}
\def\uf{\underline{f}}
\def\uo{\underline{0}}
\def\uz{\underline{z}}

\def\ub{\underline{\beta}}
\def\fb{\mathfrak{B}}
\def\fF{\mathfrak{F}}
\def\fy{\mathfrak{Y}}
\def\fz{\mathfrak{Z}}
\def\sfm{\sigma_f^{\mathrm{min}}}
\def\disc{\Delta}
\def\NP{\mathbbl{\Delta}}
\def\h{\mathsf{h}}
\def\NNP{\hat{\NP}}

\let\oldtocsection=\tocsection
 
\let\oldtocsubsection=\tocsubsection
 
\let\oldtocsubsubsection=\tocsubsubsection
 
\renewcommand{\tocsection}[2]{\hspace{0em}\oldtocsection{#1}{#2}}
\renewcommand{\tocsubsection}[2]{\hspace{1em}\oldtocsubsection{#1}{#2}}
\renewcommand{\tocsubsubsection}[2]{\hspace{2em}\oldtocsubsubsection{#1}{#2}}

\makeatother

  \providecommand{\corollaryname}{Corollary}
  \providecommand{\definitionname}{Definition}
  \providecommand{\examplename}{Example}
  \providecommand{\lemmaname}{Lemma}
  \providecommand{\propositionname}{Proposition}
  \providecommand{\remarkname}{Remark}
\providecommand{\theoremname}{Theorem}

\providecommand{\conventionname}{Convention}

\begin{document}

\title[Hodge theory of degenerations, II]{Hodge theory of degenerations, (II): \\ Vanishing cohomology and geometric applications}

\author{Matt Kerr}
\address{Washington University in St. Louis, Department of Mathematics and Statistics, St. Louis, MO 63130-4899}
\email{matkerr@math.wustl.edu}

 \author{Radu Laza}
 \address{Stony Brook University, Department of Mathematics, Stony Brook, NY 11794-3651}
\email{rlaza@math.stonybrook.edu}

\thanks{Research of the first author is supported in part by NSF grant DMS-2101482. Research of the second author is supported in part by NSF grant DMS-2101640.}

\bibliographystyle{amsalpha}

\maketitle

\begin{abstract}
We study the weighted spectrum and vanishing cohomology for several classes of isolated hypersurface singularities, and how they contribute to the limiting mixed Hodge structure of a smoothing.  Applications are given to several types of singularities arising in KSBA and GIT compactifications and mirror symmetry, including nodes on odd-dimensional hypersurfaces, $k$-log-canonical and $k$-rational singularities, and singularities with Calabi-Yau tail.
\end{abstract}
\setcounter{tocdepth}{1}
\tableofcontents

\section*{Introduction}
When the period map is used to construct or interpret compactifications of a moduli space of some geometric objects, it is the interplay between asymptotic and specialized Hodge-theoretic invariants which allows us to map singular geometry to boundary components.  In this article, we continue the study of this interplay begun in \cite{KL1}, turning our focus from generalizations of the Clemens--Schmid sequence in [op. cit.] to the \emph{weighted spectra of singularities}, with a focus on singularity types appearing in GIT and MMP.  Our goals here are threefold.  First, we want to review and develop the calculus of weighted spectra and their relationship to birational and toric geometry; this is the subject of an absolutely vast literature, and we hope the resulting ``cheat sheet'' is useful to other researchers.  Second, it is through these spectra that the singularities contribute to the cohomology of the vanishing cycles; here we discuss several geometrically motivated examples in the isolated singularities case. The more intricate case of non-isolated singularities will be discussed in \cite{KL3}.  Finally, the influence of the vanishing cohomology on the limiting MHS of a degeneration is also not always straightforward, and we offer several results of a general nature (here and in [op.~cit.]) that resolve ambiguities in the Clemens--Schmid and vanishing-cycle exact sequences. 

\begin{conv}
References to \cite{KL1}, henceforth referred to as \textbf{Part I}, will be written (I.*.*) resp. ``Theorem I.*.*'' etc.
\end{conv}

\subsection*{Set-up and Overview of results} Throughout, we shall consider
\begin{multline}\label{f}
 f\colon \cx\to \Delta\; \textit{ a projective morphism from an irreducible complex analytic space of} \\ \textit{dim. } n+1 \textit{ to the disk, which extends to a projective morphism of quasi-projective varieties.}
\end{multline}

\noindent We write $X_t(=f^{-1}(t))$ and $X_0$ for the general and special fibers respectively. If $\cx$ and $f|_{\cx\setminus X_0}$ are smooth, and $\text{sing}(X_0)=\{x\}$, then we are in the setting of an \emph{isolated hypersurface singularity}, and the Clemens--Schmid and vanishing-cycle exact sequences (cf. \eqref{I1}-\eqref{I2} below) reduce to isomorphisms $H^k(X_0)\cong H^k_{\lm}(X_t)$ for $k\neq n,n+1$ and the exact sequence
\begin{equation}\label{vc}
0\to H^n(X_0)\overset{\sp}{\to}H^n_{\lm}(X_t)\overset{\can}{\to} V_f\to H^{n+1}_{\pha}(X_0)\to 0
\end{equation}
of MHS, where the image of $\sp$ is the monodromy invariants $H^n(X_t)^T$ and the phantom cohomology $H^{n+1}_{\pha}(X_0):=\ker\{H^{n+1}(X_0)\twoheadrightarrow H^{n+1}_{\lm}(X_t)\}$ is pure of weight $n+1$.  Writing $T=T^{\text{ss}}e^N$ for the Jordan decomposition, the mixed Hodge structure and $T^{\text{ss}}$-action (with eigenvalues $e^{2\pi i\lambda}$) decompose the vanishing cohomology $V_f\otimes\CC\cong \oplus_{\lambda\in[0,1)\cap \QQ}\oplus_{p,q\in\ZZ^2}V_{f,\lambda}^{p,q},$ and we define the \emph{weighted spectrum} 
\begin{equation}\label{msp}
\tilde{\sigma}_f\; :=\;\sum_{\lambda,p,q} \dim(V_{f,\lambda}^{p,q})[(p+\lambda,p+q)]\;\in\; \ZZ[\QQ\times\ZZ]
\end{equation}
and \emph{spectral minimum} $\sigma_f^{\text{min}}:=\text{min}\{p+\lambda\mid V_{f,\lambda}^{p,q}\neq \{0\} \text{ for some }q\}$ of $f$ (at $x$) accordingly.  

\medskip

Writing $f$ as a polynomial in suitable local coordinates at $x$, we may compute these invariants by means of the Brieskorn--Steenbrink residue theory (Theorems \ref{T2.2}, \ref{t5.2a}, \ref{t5.2b}), the combinatorics of the Newton polytope $\NP$ of $f$ (equations \eqref{eq2.5.2}-\eqref{eq2.5.6}), and the Sebastiani--Thom formula expressing the weighted spectrum of a join $f\oplus g$ as the convolution $\tilde{\sigma}_f * \tilde{\sigma}_g$ of weighted spectra (Theorem \ref{t6}).  This calculus owes its existence to work of several authors, especially Danilov, M. Saito, Scherk, Steenbrink, and Varchenko \cite{Da,Sa6,Sa3,St1,steenbrinkvan,ScSt,Va4}.  We reformulate and streamline it here in the course of treating several examples, two of which may be summarized as follows.  The first pertains to certain singularities in dimension $n=2$ that are worse than log-canonical (essentially higher-dimensional analogues of the cusp singularity for curves), and the exceptional component or ``tail'' acquired by the singular fiber under a variant of semistable reduction which occurs in the passage from GIT to KSBA compactifications of moduli (e.g.~see \cite{log2} and \cite{GPSZ} for some concrete applications in this direction):  

\begin{thmx}[cf. Theorem \ref{TCY}] \label{t-a}
The isolated quasi-homogeneous surface singularities with pure $K3$ tail (Def. \ref{def2.2}) are precisely the 14 Dolgachev singularities, the 6 quadrilateral singularities, and 2 trimodal singularities \textup{(}$V_{15}$ and $N_{16}$\textup{)}.
\end{thmx}

\noindent A second example concerns certain singularities (in arbitrary dimension), {\it $k$-log-canonical} singularities, arising in the work of Mustata and Popa on Hodge ideals \cite{MP}, which are better than log-canonical. For such singularities, we establish a tighter connection between the cohomology of the central fiber and the limit cohomology (compare \cite{KL1} for $k=0$). In fact, the analysis here singles out also the stronger notion of {\it $k$-rational} singularities (Definition \ref{def-krat}) for which even better results hold (compare \cite{KLS}  for $k=0$). 

\begin{thmx}[cf. Corollary \ref{cor2.4a}, Theorem \ref{thm2.5A}, $\S$\ref{sec-ex-klog}] \label{t-b}
\begin{enumerate}[label=\textup{({\roman*})},leftmargin=0.8cm]
\item[ ]
\item \textup{[isolated hypersurface singularity case]} $X_0$ has a $k$-log-canonical singularity at $x$ $\iff$ $\sigma_f^{\text{min}}\geq 1+k$.  In particular, $2k$-fold suspensions of log-canonical singularities are $k$-lc, as are singularities satisfying $\S$\textup{\ref{S2.5}(i)-(iii)} for which $(\tfrac{1}{k+1},\ldots,\tfrac{1}{k+1})$ lies in the Newton polyhedron of $f$.
\item \textup{[arbitrary (not necessarily isolated) case, but with $\cx$ still smooth]} If $X_0$ has $k$-log-canonical singularities, then the specialization maps $H^*(X_0)\overset{\sp}{\to} H^*_{\lim}(X_t)$ induce isomorphisms on the $\gr_F^i$ for $0\leq i\leq k$ \textup{(}in all degrees $*$\textup{)}.
\item \textup{[isolated $k$-rational case]} If $X_0$ has isolated hypersurface singularity with $\sigma_f^{\text{min}}> 1+k$ (in particular, $X_0$ is $k$-log canonical), it holds additionally that $\gr_F^p W_{n-1}H^n(X_0)=0$ for $0\le p\le k$.  \end{enumerate}
\end{thmx}

\begin{rem}\label{rem-non-iso}
Strictly speaking, we establish here Theorem \ref{t-b}(ii) only in the isolated singularity case. An earlier version of our manuscript included a proof of the non-isolated case as well; this material is now part of an expanded treatment \cite{KL3} of the non-isolated case. 
\end{rem}

\begin{rem}[{\it Recent developments on higher Du Bois singularities}]\label{Rem-kDB}
Since the posting of our first version of the manuscript, the theory of higher log-canonical singularities dramatically expanded (see esp. \cite{MOPW} and \cite{JKSY-duBois}). Namely, there is a general notion of {\it $k$-Du Bois singularities} extending naturally the definition of Du Bois singularities (the case $k=0$). In the case of hypersurface singularities, (1) $k$-Du Bois is equivalent to $k$-log canonical and (2) this can be characterized  numerically by a condition extending that of Theorem \ref{t-b}(i) to the possibly non-isolated case (see \cite[Thm. 1]{JKSY-duBois}). Furthermore, the second part of the Theorem \ref{t-b}, regarding the effect of $k$-log canonical (or equivalently $k$-Du Bois) singularities on the Hodge numbers in families  was generalized by second author and Friedman \cite{FL-DuBois} (see esp. \cite[Cor. 1.4]{FL-DuBois}). Finally, a concrete geometric application of this circle of ideas (beyond what is included here) is given in \cite{FL-Def}. 
\end{rem}
\begin{rem}[{\it Higher rational singularities}]\label{Rem-krat} Our analysis of the behavior of Hodge structure in families (esp. Corollary \ref{cor2.4a}) led us to the introduction of $k$-rational singularities (Definition \ref{def-krat}). Our definition here is purely numerical and applies a priori only to the case of isolated hypersurface singularities. Nonetheless, Theorem \ref{t-b}(iii) captures an important phenomenon occurring for this class of singularities, which subsequently led to a much more general theory \cite{FL-isolated,FL-DuBois}.  As in the case of higher Du Bois singularities, the study of higher rational knew rapid growth recently (see esp. \cite{FL-DuBois}, \cite{MP-Rat}, and Saito's appendix to \cite{FL-DuBois}). In particular, we note that the numerical definition given here agrees with the more general definition given subsequently (under the isolated hypersurface assumptions). 
\end{rem}

More generally, in order to deduce from spectral data results about specialization maps, LMHS types, or monodromy (see for example Corollary \ref{cor2.4b}), we need control over the phantom cohomology $H^*_{\pha}(X_0)$, especially in degree $*=n+1$.  In the isolated hypersurface singularity setting, one result in this direction is:

\begin{thmx}[cf. Proposition \ref{prop2.4b}]\label{t-c}
Assume that $\cx$ is smooth, and $X_0$ has a single, log-canonical singularity at $x$, not contained in the base-locus of $|K_{\cx}|$.  Then $h^{1,n}(H^{n+1}_{\pha}(X_0))=0$.
\end{thmx}

\noindent Via \eqref{vc} and $\S$\ref{S5.1}, this has the following consequence for degenerations of Calabi--Yau varieties:

\begin{corx}[cf. Corollaries \ref{cor2.5B1}-\ref{cor2.5B2}]\label{cor-d}
For $\cx$ smooth with Calabi--Yau $n$-fold fibers $\{X_t\}_{t \neq 0}$, and $X_0$ possessing a single \textup{(}convenient, nondegenerate\textup{)} singularity with Newton polytope $\NP\supset \ZZ_{>0}^{n+1}$, let $j$ be the least integer for which the $(j+1)$-skeleton of $\NP$ contains $(1,\ldots,1)$ \textup{(}except we put $j=0$ instead of $-1$ if this is a vertex of $\NP$\textup{)}. Then $h^{0,k}(H_{\lm}^n(X_t))=\delta_{jk}$.
\end{corx}

\noindent Some further results controlling $H_{\pha}^*$ pertain to the presence of multiple nodes:

\begin{thmx}[cf. Theorem \ref{t:schoen}]\label{t-d}
\begin{enumerate}[label=\textup{({\roman*})},leftmargin=0.8cm]
\item[ ]
\item If $n$ is odd, $\cx$ is a family \textup{(}with smooth total space\textup{)} of ample hypersurfaces in a smooth projective variety $\mathbf{P}$ satisfying Bott vanishing, and $X_0$ has only nodal singularities, then $H^{n+1}_{\pha}(X_0)$ has type $(\tfrac{n+1}{2},\tfrac{n+1}{2})$ and rank equal to the dimension of the cokernel of evaluation $H^0(\mathbf{P},K_{\mathbf{P}}(\tfrac{n+1}{2}X_0))\to\CC^{|\text{sing}(X_0)|}$ at the nodes.
\item If $n$ is even, $\cx$ has only nodal singularities, and $\text{sing}(\cx)\subset X_0$ \textup{(}while $X_0$ has possibly non-isolated singularities\textup{)}, then $\text{coker}\{H^n(X_0)\overset{\sp}{\to} H^n_{\lm}(X_t)^T\}$ and $W_n H^{n+1}_{\pha}(X_0)$ measure the failure of Clemens-Schmid; and the sum of their dimensions is bounded by $|\text{sing}(\cx)|$ \textup{(}with equality if $n=2$\textup{)}.
\end{enumerate}
\end{thmx}
\begin{rem}
As in Remark \ref{rem-non-iso}, only Theorem \ref{t-d}(i) is proved here.  We include (ii), which will be proved in \cite{KL3}, to give a more global view of the subject.
\end{rem}

The Hodge theory of singularities has numerous applications to bounding the number and type of singularities on hypersurfaces in a fixed ambient space (see for instance Example \ref{ex2.2d}). Recent accounts and results in this direction may be found in \cite{vS}, \cite{Ca}, and \cite{CDKP}.

\subsection*{Structure of the paper} We now turn to a summary of the contents.  After a brief review of the most relevant results of \textbf{Part I}, we commence with $\S$\ref{S2.1} in the setting of an isolated complete intersection singularity (ICIS), whose weighted spectrum captures the MHS-type and $T^{\text{ss}}$-action on the reduced middle cohomology of the Milnor fiber (or vanishing cohomology) in a smoothing, refining the usual spectrum and Milnor number.  We describe the generic symmetries and bounds on this spectrum in Prop. \ref{prop1.2}, the stricter bounds that accrue to du Bois and rational singularities in Prop. \ref{prop2.1b}, and the relation to the tail (and semistable reduction) in Prop. \ref{prop2.1c}.   A common thread of the paper is the computation of the weighted spectrum, and one case where this is easily done is that of quasi-homogeneous hypersurface singularities.  We review and amplify the residue-theory approach due to Steenbrink \cite{St1} in $\S$\ref{S2.2}, with a view to applications in $\S$\ref{S-wf} (incl. bounds on surface singularities) and $\S$\ref{S2.3}, as well as to eigenspectra in \cite{GKS} (where the smoothing admits a cyclic automorphism).  In $\S$\ref{s-schoen}, we explain how a result in Schoen's thesis \cite{Schoen} allows us to disambiguate the contribution of the spectrum to the LMHS for a degeneration of odd-dimensional varieties with nodal singular fiber.

\medskip

Our series of papers is motivated by the study of the interplay between geometric and Hodge-theoretic compactifications of moduli.  It is well-known, as discussed in {\bf Part I} (cf. Thm. \ref{th-I1} below), that one has relatively good Hodge-theoretic control over this interplay for degenerations where the central fiber has semi-log-canonical (slc) singularities; these are the singularities allowed in KSBA compactifications \cite{Kollar-Book}. 
However, in many geometric situations (such as GIT compactifications), the central fibers can have worse than log-canonical singularities. In $\S$\ref{S2.3}, we focus on essentially the simplest case for which the central fiber has non-log-canonical singularities. This study is inspired by the work of Hassett \cite{hassett} in dimension $1$, and some work of the authors and their collaborators in dimension $2$ (see esp. \cite[\S6]{log2}, \cite{gallardo}, and \cite{GPSZ}). Specifically, we classify singularities with Calabi--Yau tail (Thm. \ref{t-a}); these should be understood as higher-dimensional analogues of the cusp singularity in dimension $1$ (which is replaced by an elliptic curve in a semi-stable replacement of the degeneration).

\medskip

In $\S$\ref{S2.4}, we return to the other extreme: log-canonical singularities, and the even milder {\it $k$-log-canonical} and {\it $k$-rational} singularities with $k\geq 1$ ($k=0$ being the log-canonical and rational case respectively).  This is inspired by the recent work of Mustata--Popa on Hodge ideals \cite{MP,MP18i,MP18ii,MP-Inv,MOP,Po}, and some recent work of M. Saito and his collaborators \cite{Sai16, JKSY} (see also Remarks \ref{Rem-kDB} and \ref{Rem-krat} for subsequent developments). After delineating the relationships between the various birational and Hodge-theoretic invariants of singularities (log-canonical threshold, jumping numbers, generation level, period exponent, etc.) in Prop. \ref{prop2.4a}, we characterize $k$-log-canonical singularities in terms of the spectrum and describe the influence they have on the limiting MHS (Cor. \ref{cor2.4a}, Thm. \ref{t-b}).  Corollary \ref{cor2.4b} and Proposition \ref{prop2.4b} apply Prop. \ref{prop2.4a} to an algebro-geometric question about the equivalence of \emph{canonicity} of an isolated log-canonical singularity and \emph{Picard--Lefschetz monodromy} for the degeneration to it.  An additional application is to the period map for cubic fourfolds (cf. \cite{La10}); see Remark \ref{rem-cubic}. 

\medskip

After establishing the relevance of the spectrum for problem of degeneration, we return to the issue of computing the spectrum for a given singularity. While the quasi-homogeneous case of $\S$\ref{S2.2} is fairly easy to understand, the case of a general isolated hypersurface singularity is more subtle. The beginning of $\S$\ref{S5.1} uses work of Steenbrink and Danilov to derive formulas \eqref{eq2.5.2}-\eqref{eq2.5.6} for the weighted spectrum, which are then used to relate the Newton polytope $\NP$ to log-canonicity and rationality (Cor. \ref{cor2.5B2}), $k$-log-canonicity (Thm. \ref{thm2.5A}), extremal $N$-strings in $H^n_{\lm}(X_t)$ (Thm. \ref{thm2.5B}), and the Calabi--Yau case to the \emph{Kulikov type} of the degeneration (Cor. \ref{cor2.5B1}).  In $\S$\ref{S5.2}, we recall how the interplay between the $\cv$-filtration (arising from the Gauss-Manin system on the Milnor fiber) and the Brieskorn lattice (arising from local asymptotics of periods) leads to a combinatorial, Jacobian-ring type description of $(V_f,F^{\bullet})$ and $T^{\text{ss}},N\in\mathrm{End}(V_f)$ (Thms. \ref{t5.2a} and \ref{t5.2b}). We use this to work out weighted spectra of a generalization of $T_{p,q,r}$-singularities to higher dimension (Ex. \ref{e5.11}). In $\S$\ref{S2.6}, we review and reprove the well-known Sebastiani--Thom formula for the join of singularities, using the occasion to correct an inaccuracy in Scherk--Steenbrink and Kulikov \cite{ScSt,Kul}.  An easy consequence in $\S$\ref{sec-ex-klog} is that the suspension (join with $z^2$) of a singularity [resp. double suspension] amounts to a half-twist on the level of MHSs [resp. Tate-twist], which leads to natural examples of $k$-log-canonical singularities.

\subsection*{Synopsis of Part I (\cite{KL1})}
For ease of reference, we summarize here the results, notations and terminologies that we shall most frequently invoke below.

Let $f$ be as in \eqref{f}. Shrinking $\Delta$ if needed, $X_0$ is a deformation retract of $\cx$, so that $H^*(\cx)\cong H^*(X_0)$.  Assuming that $f^{-1}(\Delta^*)$ is smooth, we define additional MHSs $H^k_{\lm}(X_t):=\HH^k(X_0,\psi_f\QQ_{\cx})$, $H^k_{\van}(X_t):=\HH^k(X_0,\phi_f\QQ_{\cx})$, $H^k_{\pha}(X_0):=\ker\{\sp\colon H^k(X_0)\to H^k_{\lm}(X_t)\}$, $\IH^k(\cx):=\HH^{k-n-1}(X_0,\imath^*_{X_0}\IC_{\cx})$, $\IH_c^k(\cx):=\HH^{k-n-1}(X_0,\imath^!_{X_0}\IC_{\cx})$, and write $T=T^{\text{ss}}T^{\text{un}}=T^{\text{ss}}e^N$ for the action of monodromy on $H^*_{\lm}$ and $H^*_{\van}$.  (Note that $\IH^k(\cx)$ carries a \emph{mixed} Hodge structure, not a pure one:  e.g., if $\cx$ is smooth, then $\IC_{\cx}\simeq \QQ_{\cx}[n+1]$ and this is $H^k(X_0)$.)  Then we have the \emph{vanishing cycle sequences} (of MHS, compatible with $T$)
\begin{equation}\label{I1}
0\to H^k_{\pha}(X_0)\to H^k(X_0)\overset{\sp}{\to}H^k_{\lm}(X_t)\overset{\can}{\to} H^k_{\van}(X_t)\overset{\delta}{\to} H^{k+1}_{\pha}(X_0)\to 0
\end{equation}
and the \emph{IH Clemens-Schmid sequences} (of MHS)
\begin{equation}\label{I2}
0\to H^{k-2}_{\lm}(X_t)_T (-1)\overset{\sp^{\vee}_{\IH}}{\to} \IH^k_c(\cx) \overset{\gy}{\to} \IH^k(\cx)\overset{\sp_{\IH}}{\to} H^k_{\lm}(X_t)^T\to 0
\end{equation}
for each $k$, cf. (I.5.7),\footnote{We have used the assumption that $f^{-1}(\Delta^*)$ is smooth to simplify the end terms of (I.5.7).} (I.5.12), and Remark I.5.6.

Consider a semistable reduction which factors as $\cy\overset{\pi}{\to}\fx$ and $\fx\overset{\rho}{\to}\cx$, where $\rho$ is the base-change by $t\mapsto t^{\kappa}$ (with $(T^{\text{ss}})^{\kappa}=I$), and $\pi$ is a log-resolution of $(\fx,X_0)$; and let $\cy'\overset{\pi'}{\to}\cx$ be a log-resolution of $(\cx,X_0)$. We write $Y_0=\pi^{-1}(X_0)=\tilde{X}_0\cup_E \cE$ and $Y_0'=(\pi')^{-1}(X_0)=\tilde{X}_0\cup_{E'}\cE'$, and refer to $(\cE,E)$ or $\cE\setminus E$ as the \emph{tail}.  This leads to sequences (cf. (I.8.4))
\begin{equation}\label{I3}
\left\{
\begin{split}
0\to \bar{H}^k(X_0)\to H^k_{\lm}(X_t)^T\to \bar{H}^k(\cE')\to 0 \\
0\to \bar{H}^k(X_0)\to H^k_{\lm}(X_t)^{T^{\kappa}}\to \bar{H}^k(\cE)\to 0
\end{split}
\right.
\end{equation}
termed \emph{generalized Clemens-Schmid} in {\bf Part I}.  (The bars denote certain subquotients, for which we refer to $\S$I.8.1.)

Assuming that $\cx$ has reduced special fiber (i.e. $(f)=X_0$), we have the following (Thms. I.9.3, I.9.9, I.9.11)
\begin{thm}\label{th-I1}
\textup{(i)} If $X_0$ is du Bois, then
\begin{equation*}
\gr_F^0 H^k(X_0)\cong \gr_F^0 H_{\lm}^k(X_t)^{T^{\text{ss}}}\cong \gr_F^0 H^k_{\lm}(X_t)\;\;\;\;(\forall k). 
\end{equation*}
\textup{(ii)}  If $X_0$ has rational singularities and $\cx$ is smooth, then
\begin{equation*}
W_{k-1}\gr_F^0 H^k_{\lm}(X_t)=\{0\}\;\;\;\;(\forall k).
\end{equation*}
\textup{(iii)}  If $X_0$ has rational singularities and $\cx$ is smooth, then 
\begin{equation*}
\gr_F^1 H^k(X_0)\cong \gr_F^1 (H^k_{\lm}(X_t))^{T^{\text{ss}}}.
\end{equation*}
\end{thm}
For instance, (i) [resp. (ii)] holds if $\cx$ is normal and $\QQ$-Gorenstein and $X_0$ is semi-log-canonical [resp. log-terminal].

Finally, assume that $\cx$ is smooth:   then \eqref{I2} reduces to the \emph{Clemens-Schmid sequence} (cf. Thm. I.5.3)
\begin{equation}\label{I4}
0\to H^{k-2}_{\lm}(X_t)_T (-1)\overset{\sp^{\vee}}{\to} H_{2n-k+2}(X_0)(-n-1)\overset{\gy}{\to}H^k(X_0)\overset{\sp}{\to}H^k_{\lm}(X_t)^T\to 0,
\end{equation}
for each $k$, and we have
\begin{thm}\label{th-I2}
\textup{(i)} The phantom cohomology $H^k_{\pha}(X_0)=\ker(\sp)=\delta(H^{k-1}_{\van})=\text{im}(\gy)$ is pure of weight $k$ and level $\leq k-2$.\\
\textup{(ii)} The vanishing cohomology $H^k_{\van}(X_t)$ \textup{(}and hence $H^k_{\pha}$\textup{)} is zero outside the range $|k-n|\leq \dim(\mathrm{sing}(X_0))$. \textup{(}More generally, this holds if $\cx$ has local complete intersection singularities.\textup{)}
\end{thm}

When $\cx$ is smooth with isolated singularities, \eqref{I4} reduces to (cf. (I.6.2))
\begin{equation}\label{I5}
\left\{
\begin{split}
0&\to H^{n-1}_{\lm}(X_t)(-1)\to H_{n+1}(X_0)(-n-1)\to H^{n+1}(X_0)\to H^{n+1}_{\lm}(X_t)\to 0\\
&\text{and}\;\;\;\;\; H^n(X_0)\cong H^n_{\lm}(X_t)^T \,,
\end{split}
\right.
\end{equation}
while $H^k(X_0)\cong H^k_{\lm}(X_t)$ in all other degrees.  The only nonzero degree of vanishing cohomology is $k=n$, for which (cf. Prop. I.6.3)
\begin{equation}\label{I6}
H^k_{\van}(X_t)\cong \tilde{H}^k(\cE\setminus E)
\end{equation}
as MHS if $\cE,E$ are irreducible and smooth.  This and the related Prop. I.6.4 will be strengthened and generalized in Prop. \ref{prop2.1c}, Prop. \ref{prop2.2a}, and $\S$\ref{S2.5} below.
 
\subsection*{Acknowledgement} The authors thank the IAS for providing the environment in which, some years ago, this series of papers was first conceived.  We also thank P. Brosnan, M. Green, P. Griffiths, G. Pearlstein, and C. Robles for our fruitful discussions and collaborations on period maps and moduli during the course of the NSF FRG project ``Hodge theory, moduli, and representation theory''.  The first author also thanks B. Castor and P. Gallardo for discussions and correspondence closely related to this work. The second author thanks R. Friedman and M. Mustata on several discussions related to higher Du Bois and higher rational singularities. Finally, we are grateful to M. Saito for numerous detailed comments on our series of papers, and to the referee for their helpful remarks.

An earlier version of our manuscript contained an extra section dealing with the non-isolated case. In course of revising the manuscript, this material underwent a significant expansion, and we decided it that it makes more sense to spin it off as the third installment \cite{KL3} of our study.

\section{Mixed spectra of isolated singularities\label{S2.1}}

We begin our discussion in the setting of a fixed smoothing of an isolated complete-intersection
singularity (ICIS). Let $U\subset\CC^{n+r}$ be a neighborhood of
the origin $\underline{0}$. Given $\uf=(f_{1},\ldots,f_{r}):\,(U,\uo)\to(\CC^{r},\uo)$
holomorphic with critical locus meeting $f_2=\cdots =f_r=0$ only at $\uo$, consider the 1-parameter
family
\[
\cu:=\uf^{-1}(\disc\times\{0\}^{r-1})\overset{f_{1}}{\longrightarrow}\disc
\]
over a disk about $0$ with coordinate $t$. (We shall freely conflate
$t$ with $f_{1}$.) Writing $V_{\uf}:=\tilde{H}^{n}(\mathfrak{F}_{\cu,t_{0}},\QQ)\cong\imath_{\uo}^{*}\,\pphi_{t}\QQ_{\cu}[n+1]$
for the reduced cohomology of the Milnor fiber $\mathfrak{F}_{\cu,t_{0}}=f_{1}^{-1}(t_{0})\cap B_{\epsilon}(\uo)$,
and $T=T^{\text{ss}}e^{N}$ for the monodromy (about $t=0$) on $V_{\uf}$,
we remark that these are unchanged under local analytic isomorphisms.\footnote{By the Jordan decomposition, we may write $T$ uniquely as a product of \emph{commuting} semisimple and unipotent operators (i.e., $T^{\text{ss}}$ and $e^N$, with $N$ nilpotent().}
So we may assume, arguing as in \cite{Bri} for $r=1$, that the $f_{i}$
are polynomials and (adding very general, higher-degree homogeneous
polynomials to the $f_{i}$ if necessary) that the equations $\mathsf{X}_{0}^{d_{1}}f_{1}(\tfrac{\mathsf{X}_{1}}{\mathsf{X}_{0}},\ldots,\tfrac{\mathsf{X}_{n+r}}{\mathsf{X}_{0}})-t\mathsf{X}_{0}^{d_{1}}=0$
and $\mathsf{X}_{0}^{d_{i}}f_{i}(\tfrac{\mathsf{X}_{1}}{\mathsf{X}_{0}},\ldots,\tfrac{\mathsf{X}_{n+r}}{\mathsf{X}_{0}})=0$
define a fiberwise compactification
\[
\PP^{n+r}\times\disc\supseteq\cx\overset{t}{\to}\disc
\]
of $\cu$ with unique singularity (for both\footnote{or just for $X_{0}$ if $r=1$} $\cx$ and $X_{0}$) at $[1:\uo]$. Of course, $V_{\uf}$ is just
the vanishing cohomology $H_{\mathrm{van}}^{n}(X_{t})$.

In this scenario, $\QQ_{\cx}[n+1]$ may fail to be semisimple
and \eqref{I4} need not be exact.  But we do have \eqref{I2} 
(as $\mathrm{IC}_{\cx}^{\bullet}$ is simple), and -- more consequentially at present --
\eqref{I1} in the form \begin{equation} \label{eq2.1a}
0\to H^n(X_0) \overset{\mathrm{sp}}{\to} H^n_{\mathrm{lim}}(X_t) \overset{\mathrm{can}}{\to} V_{\uf} \to H_{\mathrm{ph}}^{n+1} \to 0\; ,
\end{equation}in which $T$ acts trivially on the end terms and compatibly on the
interior ones. Since $T_{\text{ss}}$ respects the Hodge--Deligne decomposition
$V_{\uf,\CC}=\oplus_{p,q=1}^{n}V_{\uf}^{p,q}$, we may further decompose
$V_{\uf}^{p,q}=\oplus_{\lambda\in[0,1)\cap\QQ}V_{\uf,\lambda}^{p,q}$
into $e^{2\pi\sqrt{-1}\lambda}$-eigenspaces, and set
\[
m_{p+\lambda}:=\sum_{q}m_{p+\lambda,p+q}:=\sum_{q}\dim_{\CC}(V_{\uf,\lambda}^{p,q}).
\]
Moreover, the $(-1,-1)$-morphism $N\in\mathrm{End}(V_{\uf})$ commutes
with $T^{\text{ss}}$ (hence maps $V_{\uf,\lambda}^{p,q}\to V_{\uf,\lambda}^{p-1,q-1}$);
its cokernel $P_{\uf}$ is the \emph{primitive vanishing cohomology},
and we write $m_{p+\lambda,p+q}^{0}=\dim_{\CC}(P_{\uf,\lambda}^{p,q})$
etc.
\begin{defn}\label{def-spec}
The \emph{spectrum} resp. \emph{weighted spectrum}\footnote{we prefer this terminology to ``spectral pairs'' for elements of a free abelian group} of $\uf$ is the element $$\sigma_{\uf}:=\sum_{\alpha\in\QQ}m_{\alpha}[\alpha]\in\ZZ[\QQ]\;\;\;\;\text{resp.}\;\;\;\;
\tilde{\sigma}_{\uf}:=\sum_{\alpha\in\QQ}\sum_{w\in\ZZ}m_{\alpha,w}[(\alpha,w)]\in\ZZ[\QQ\times\ZZ].$$
(For \emph{primitive spectra}, replace $m_{\cdots}$ everywhere by
$m_{\cdots}^{0}$.) The sums $\sum_{\alpha\in\QQ}m_{\alpha}=\dim(V_{\uf})=:\mu_{\uf}$
and $\sum_{\alpha\in\QQ}m_{\alpha}^{0}=\dim(P_{\uf})=:\mu^0_{\uf}$
are the \emph{Milnor} resp. \emph{primitive Milnor numbers} of $\uf$. The
\emph{support} of the spectrum is the subset $|\sigma_{\uf}|\subset\QQ$
on which $m_{\alpha}\neq0$.
\end{defn}

\begin{example}
Consider the two hypersurface singularities ($r=1$) defined by $f=x_1^2+x_2^4+x_3^4$ ($\tilde{E}_7$, Ex. \ref{ex2.2a}, $n=3$) and $g=x_1^4 x_2 + x_1 x_2^4 + x_1^2 x_2^2$ (Ex. \ref{ex5.1}, $n=2$).  We have $\tilde{\sigma}_f=\tilde{\sigma}^0_f=[(1,3)]+2[(\tfrac{5}{4},2)]+3[(\tfrac{3}{2},2)]+2[(\tfrac{7}{4},2)]+[(2,3)]$ and $\tilde{\sigma}_g=[(\tfrac{1}{2},0)]+ \tilde{\sigma}_g^0=[(\tfrac{1}{2},0)]+2[(\tfrac{2}{3},1)]+2[(\tfrac{5}{6},1)]+3[(1,2)]+2[(\tfrac{7}{6},1)]+2[(\tfrac{4}{3},1)]+[(\tfrac{3}{2},2)]$; Milnor numbers are $\mu_f=\mu_f^0=9$ and $\mu_g=\mu_g^0 +1 =13$.  (The normalization here differs from Steenbrink's, cf. Remark \ref{rem5.12}.  See also Theorem \ref{t5.2a} and Remark \ref{rem5.9}(c) for the relation between Milnor, primitive Milnor and Tyurina when $r=1$.)
\end{example}

Let $H$ represent $H^n_{\text{lim}}$ or $V_{\uf}$, and $W,Z\in \mathrm{End}(H_{\CC})$ act on $H^{p,q}$ as multiplication by $p+q,p-q$ respectively. Since $T^{\text{ss}}$, $N$, and $Z$ commute, we can pick a Jordan basis $\{v_1,Nv_1,\ldots,N^{\ell_1}v_1\,;\,\ldots\,;\,v_r,Nv_r,\ldots,N^{\ell_r}v_r\}$ for $H$ such that the cyclic vectors $v_i$ each belong to some $H^{p_i,q_i}_{\lambda_i}$ (i.e.~an intersection of eigenspaces for $T^{\text{ss}}$ and $Z$); then we also have $N^kv_i\in H_{\lambda_i}^{p_i-k,q_i-k}$. We call a cyclic sub-basis $\{v_i,Nv_i,\ldots,N^{\ell_i}v_i\}$ (of cardinality $\ell_i+1$) an \emph{$N$-string of length $\ell_i$} starting at $(p_i,q_i)$.  It can be arranged for these to come in conjugate pairs (or be self-conjugate if $p_i=q_i$ and $e^{2\pi\sqrt{-1}\lambda_i}=\pm 1$).

Now define $\iota_{n}:\QQ\times\ZZ\to\QQ\times\ZZ$ be the involution
sending $(\alpha,w)\mapsto(n+1-\alpha,2n-w)$ if $\alpha\notin\ZZ$
and $(n+1-\alpha,2n+2-w)$ if $\alpha\in\ZZ$. 

\begin{prop}\label{prop1.2}
\emph{(i)} The support $|\sigma_{\uf}|$ lies in $[0,n+1)$\emph{;}
and the $N$-strings in $V_{\uf}$ have length $\leq n$, with those
in $V_{\uf,\neq0}$ centered about weight $w=n$.

\emph{(ii)} In the hypersurface singularity case $r=1$, $|\sigma_{\uf}|\subset(0,n+1)$,
$\iota_{n}(\tilde{\sigma}_{f})=\tilde{\sigma}_{f}$, and the $N$-strings
in $V_{f,0}$ are of length $\leq n-1$ \emph{(}and centered about
$w=n+1$\emph{)}.\end{prop}
\begin{proof}
To prove (i), recall that by the $\mathrm{SL}_2$-orbit theorem \cite{schmid}, the action of $Y:=W-nI$ and $N$ on $H^n_{\text{lim},\CC}$ are part of an $\mathfrak{sl}_2$-triple $(N,Y,N^+)$ commuting with $T^{\text{ss}}$ and $Z$. (This is usually \emph{not} true for $V_{\uf}$.) Hence the $N$-strings of $H^n_{\text{lim}}$, and thus of $\mathrm{im}(T^{\text{ss}}-I)\cap H^n_{\text{lim}}\overset{\can}{\underset{\cong}{\to}}V_{\uf,\neq 0}$, are centered about weight $n$. (Here we are using the fact that $T^{\text{ss}}-I$ is a morphism of MHS together with \eqref{eq2.1a} and the compatibility of Jordan decompositions.) The premise that $\dim X_0=n$ forces $0\leq p,q\leq n$ in $H^n_{\text{lim}}$ and $H^{n+1}_{\text{ph}}$, hence (by \eqref{eq2.1a}) in $V_{\uf}$.

For (ii), by \eqref{I5}
the bottom term (only) of each $N$-string in $\ker(T^{\text{ss}}-I)\cap H_{\text{lim}}^{n}$
comes from $H^{n}(X_{0})$; while $H_{\text{ph}}^{n+1}$ contributes
length-zero strings in weight $n+1$ (Thm. \ref{th-I2}(i)). Finally,
the symmetry property comes from swapping conjugate strings (which
sends $\lambda\in(0,1)\mapsto1-\lambda$ resp. $\lambda=0\mapsto0$,
and $(p,q)\mapsto(q,p)$), then reflecting about $p+q=n$ resp. $n+1$
(which sends $(q,p)\mapsto(n-p,n-q)$ resp. $(n+1-p,n+1-q)$).\end{proof}
\begin{rem}
If $r\neq1$ then (ii) is false. (The problem here is the integral
part of the spectrum. Taking $n=1$, $r=2$, $f_{1}=x_{3}$ and $f_{2}=x_{1}^{2}+x_{2}^{3}+x_{3}^{6}$
gives $\tilde{\sigma}_{\uf}=[(0,1)]+[(1,1)]$.) For this reason, a
different definition of the spectrum is preferred when $r>1$. Writing
$\cv=\uf^{-1}(\disc^{2}\times\{0\}^{r-2})\cap B_{\epsilon}(\uo)\overset{(f_{1},f_{2})}{\longrightarrow}\disc^{2}$
(and $t=f_{1}$,$s=f_{2}$), one replaces $V_{\uf}$ by $'V_{\uf}:=\imath_{\uo}^{*}\pphi_{t}\ppsi_{s}\QQ_{\cv}[n+2]\cong H^{n+1}(f_{2}^{-1}(s_{0}),(f_{1}\times f_{2})^{-1}(t_{0},s_{0}))$
(with $T$ remaining the monodromy about $t=0$), and shows that the
resulting weighted spectrum $'\tilde{\sigma}_{\uf}$ satisfies (ii) \cite{ES}.
\end{rem}
Next suppose that $X_{0}\subset\cx\overset{t}{\to}\disc$ contains
a finite set $\Xi$ of ICIS. Then the MHSs $H_{\van}^{n}(X_{t,\xi})$
defined by 
\[
\pphi_{t}\QQ_{\cx}[n+1]=:\bigoplus_{\xi\in\Xi}\imath_{*}^{\xi}H_{\van}^{n}(X_{t,\xi})
\]
are determined by the local analytic isomorphism class of $(\cx,t)$
(or just of $X_{0}$, in the hypersurface case) at each $\xi$.%
\footnote{The algorithms for computing them in the hypersurface case are touched
on in $\S\S$\ref{S2.2}-\ref{S2.3}. In the $r=1$ case, the spectral
pairs are invariant under deformations preserving the Milnor number.%
} The vanishing-cycle sequence becomes 
\begin{equation} \label{eq2.1b}
0\to H^n(X_0)\to H^n_{\lm}(X_t)\to \oplus_{\xi\in \Xi} H^n_{\van}(X_{t,\xi})\to  H^{n+1}(X_0)\to H_{\lm}^{n+1}(X_t)\to 0,
\end{equation}
with $H^{k}(X_{0})\cong H_{\lm}^{k}(X_{t})$ for $k\neq n,n+1$.
\begin{lem}
\label{lem2.1} $H^{k}(X_{0})$ is pure of weight $k$ for all $k\neq n$.\end{lem}
\begin{proof}
Since $\QQ_{X_{0}}[n]$ is perverse (for ICIS), we have
\[
0\to\oplus_{\xi\in\Xi}\imath_{*}^{\xi}W_{\xi}\to\QQ_{X_{0}}[n]\to\IC_{X_{0}}\to0
\]
with $\{W_{\xi}\}$ MHSs in degree $0$. Taking $\HH^{*}$ yields
\[
0\to H^{n-1}(X_{0})\to\IH^{n-1}(X_{0})\to\oplus_{\xi}W_{\xi}\to H^{n}(X_{0})\to\IH^{n}(X_{0})\to0
\]
and $H^{k}(X_{0})\cong\IH^{k}(X_{0})$ ($k\neq n,n-1$), with all
$\IH^{*}(X_{0})$ pure.
\end{proof}
So $H^{n+1}(X_{0})$ --- hence $H_{\lm}^{n+1}$ and $H_{\pha}^{n+1}:=\ker\{H^{n+1}(X_{0})\to H_{\lm}^{n+1}\}$
--- is pure of weight $n+1$ and level $\leq n-1$, hence contained
in $F^{1}$.
\begin{prop}
\label{prop2.1b} The spectrum of a du Bois ICIS \emph{(}resp. isolated
rational hypersurface singularity\emph{)}%
\footnote{Here we mean that $(X_{0},\xi)$ is du Bois or rational. In the hypersurface case, the converse of this proposition is contained in Prop. \ref{prop2.4a}. Note that for a rational singularity, $n>1$.%
} has support $|\sigma|\subset[1,n]$ \emph{(}resp. $(1,n)$\emph{)}.\end{prop}
\begin{proof}
For du Bois, apply Theorem \ref{th-I1}(i) (that $\gr^0_F H^n(X_0)\cong \gr^0_F H^n_{\lm}(X_t)$) and $\gr_{F}^{0}H_{\pha}^{n+1}=\{0\}$
to \eqref{eq2.1b}. 

For rational, Thm. \ref{th-I1}(ii) implies that $\gr_F^1 H^n_{\lm}(X_t) = \gr_F^1 H^n_{\lm}(X_t)^{T^{\text{un}}}$, hence (taking $T^{\text{ss}}$-invariants) $\gr_F^1 H^n_{\lm}(X_t)^{T^{\text{ss}}} = \gr_F^1 H^n_{\lm}(X_t)^{T}$.
Since $\cx$ is smooth ($r=1$), by Clemens-Schmid (cf. \eqref{I5}) $H^n(X_0)\cong H^n_{\lm}(X_t)^T$; and so $\gr_{F}^{1}H^{n}(X_{0})\cong\gr_{F}^{1}H_{\lm}^{n}(X_{t})^{T^{\text{ss}}}$.
Using \eqref{eq2.1b}, to show $1\notin|\sigma|$ it suffices to prove
that $\gr_{F}^{1}H_{\pha}^{n+1}=\{0\}$.

Let $\fx$ be the base-change of $\cx$ by $t\mapsto t^{\kappa}$,
where $(T^{\text{ss}})^{\kappa}=I$; and fix log resolutions $\cy\overset{\pi}{\to}\fx$
of $(\fx,X_{0})$ resp. $\cy'\overset{\pi'}{\to}\cx$ of $(\cx,X_{0})$,
with NCD singular fibers $\pi^{-1}(X_{0})=Y_{0}=\tilde{X}_{0}\cup\cE(\overset{\beta}{\hookleftarrow}\cE)$
resp. $(\pi')^{-1}(X_{0})=Y_{0}'=\tilde{X}_{0}'\cup\cE'$. We may
furthermore take $\cy\to\disc$ to be semistable. Write $\cE^{[0]}$
{[}resp. $\cE^{[1]}${]} for the disjoint union of components of $\cE$
{[}resp. their intersections{]}, and the same for $E$. By the weight
monodromy spectral sequence \cite[Cor. 11.23]{PS}, we have
\begin{equation} \label{eq2.1c}
\gr_F^1 \gr^W_{n+1} H^{n+1}_{\lm}(X_t)\cong  \mathrm{coker}\left\{ H^{0,n-1}(E^{[0]})\overset{\delta_0}{\to} \tfrac{H^{1,n}(\cE^{[0]})}{\delta_1 (H^{0,n-1}(\cE^{[1]}))}\oplus H^{1,n}(\tilde{X}_0)\right\}
\end{equation}where $\delta_{0}$ {[}resp. $\delta_{1}${]} is a sum of pushforward
maps along inclusions of components of $E^{[0]}$ {[}resp. $\cE^{[1]}${]}
into components of $\cE^{[0]}\amalg\tilde{X}_{0}$ {[}resp. $\cE^{[0]}${]}.

Now $X_{0}$ has isolated rational singularities $\Xi$, with preimage
$E$ under $\pi|_{\tilde{X}_{0}}=:\pi_{0}:\,\tilde{X}_{0}\to X_{0}$.
Taking stalks of $(n-1)^{\text{st}}$ cohomology sheaves of $R(\pi_{0})_{*}\mathcal{O}_{\tilde{X}_{0}}=\mathcal{O}_{X_{0}}$
at $\Xi$ gives\begin{equation} \label{eq2.1d}
H^{0,n-1}(E^{[0]}) = H^{n-1}(E)^{0,n-1} \subset \gr_F^0 H^{n-1}(E,\CC)=H^{n-1}(\mathcal{O}_E) = \{0\} \end{equation}since NCDs are du Bois (and $n>1$).

Finally, $H^{n+1}(X_{0})$ is pure hence injects into $H^{n+1}(\tilde{X}_{0})$,
so \begin{flalign*}
(H^{n+1}_{\pha})^{1,n} &= \gr^1_F \gr^W_{n+1} \ker \{ H^{n+1}(X_0)\to H^{n+1}_{\lm}\} \\
&\subseteq \gr^1_F \gr^W_{n+1} \ker \{ H^{n+1}(\tilde{X}_0)\to H^{n+1}_{\lm}\}.
\end{flalign*}This is clearly zero by \eqref{eq2.1c} and \eqref{eq2.1d}.
\end{proof}
Continuing with the notation from the second paragraph of the above
proof, we now discuss the terms of the generalized Clemens-Schmid
sequences \eqref{I3} before and after base-change. Clearly $\bar{H}^{k}(X_{0})=H^{k}(X_{0})$
for $k\neq n+1$, and $\bar{H}^{k}(\cE)=\{0\}$ for $k\neq n$. The
only interesting case is $k=n$, where $\bar{H}^{n}$ of $\cE^{(')}=\amalg_{\xi}\cE_{\xi}^{(')}$
measures the failure of the local invariant cycle property.

Define a ``coprimitive vector'' operator, defined on MHS with $(-1,-1)$-morphism
$N$, by 
\[
\mathscr{P}_{n}:=\sum_{j=0}^{n}W_{j}\cap\text{im}(N^{n-j}).
\]
The following elaborates on $\S$I.8.1, while also making
Theorem I.6.4(i) more precise:
\begin{prop}
\label{prop2.1c} Prior to base-change, we have\begin{equation} \label{eq2.1e}
\bar{H}^n(\cE ')=\frac{\oplus_{\xi} H^n(\cE_{\xi} ')}{\beta^* H^n_{\pha}(Y_0 ')} = \oplus_{\xi\in \Xi} \mathscr{P}_n \left( H^n_{\van}(X_{t,\xi})^{T^{\text{ss}}}\right) ,
\end{equation}which is zero if $r=1$ \emph{(}$\cx$ smooth\emph{)}. For the tail
after semistable base-change, we have\begin{equation} \label{eq2.1f}
\bar{H}^n(\cE) = \frac{\oplus_{\xi} H^n(\cE_{\xi})}{\beta^* H^n_{\pha}(Y_0)} = \oplus_{\xi \in \Xi} \mathscr{P}_n \left( H^n_{\van}(X_{t,\xi})\right) .
\end{equation}\end{prop}
\begin{proof}
That the invariant cycle theorem holds for $r=1$ is just the second
line of \eqref{I5}. The rest follows from (I.8.3), Lemma
\ref{lem2.1}, and the fact that $\mathscr{P}_{n}$ takes the lowest-weight
vector of each $\mathfrak{sl}_{2}$-string centered about $n$, but
not $n+1$ (the only two possibilities). This ensures that it picks
out precisely the images under $\can$ of $T^{\text{un}}$-invariant
vectors in $H_{\lm}^{n}$.\end{proof}
\begin{rem}
By purity of $H_{\pha}^{n}(Y_{0})$ (Thm. \ref{th-I2}(i)), Remark
I.8.1, and (I.7.1),
\[
H_{\pha}^{n}(Y_{0})\cong H_{n+2}(\cE)(-n-1)\oplus\IH_{\pha}^{n}(\fx)\cong\frac{\oplus_{\xi}H^{n-2}(\cE_{\xi})(-1)}{\text{im}(\delta_{1})}\oplus\IH_{\pha}^{n}(\fx),
\]
where the first term arises via ``$\imath^{*}\imath_{*}$'' (pushing
forward along $\cE^{[0]}\to\cy$, then pulling back to $Y_{0}$).
It would be interesting to know when there is a ``minimal'' choice
of log resolution for which $H_{\pha}^{n}(Y_{0})=\IH_{\pha}^{n}(\fx)$,
and whether this sharpens \eqref{eq2.1f}.
\end{rem}

\section{The quasi-homogeneous case\label{S2.2}}
We now specialize to isolated hypersurface singularities.  In this section, we discuss the general class of hypersurface singularities for which understanding the vanishing cohomology and the spectrum is easiest, namely the quasi-homogeneous singularities. We review the original computation due to Steenbrink \cite{St1}, and relate it to cohomology of the tail $\cE\backslash E$. The main point here is that for a general smoothing $\fx$ of a quasi-homogeneous hypersurface singularity, a single weighted blow-up will produce a partial resolution and an almost normal-crossings degeneration. The remaining singularities (for the central fiber and total space) will be finite quotient singularities, and Hodge theory and Griffiths's residue calculus work well in this type of situation, leading to explicit formulas in terms of the weights (e.g. see Theorem \ref{T2.2}). In the later subsections $\S\S$\ref{S-wf}-\ref{s-schoen}, we consider various examples and amplifications (deformations, period maps, nodes).

Let
$\uw\in\QQ_{>0}^{n+1}$ be a weight vector. Writing $\uz=(z_{1},\ldots,z_{n+1})$
and $\mathfrak{M}(\uw)=\{\underline{m}\in\ZZ_{\geq0}^{n+1}\mid\underline{m}\cdot\uw=1\}$,
consider \begin{equation} \label{eq2.2a}
f:=\sum_{\underline{m}\in\mathfrak{M}(\uw)}a_{\underline{m}}\uz^{\underline{m}}:\;\CC^{n+1}\to\CC,
\end{equation}where $\{a_{\um}\}\subset\CC$ are chosen to make $\{\uo\}$ an isolated (hence the unique) singularity of $\{f(\uz)=0\}$.
\footnote{Equivalently, since $f$ is quasi-homogeneous, $\hat{E}\subset\mathbf{H}$ (see below) is a quasi-smooth hypersurface, cf. \cite[3.1.5]{DoWP}}
Let $\fb\subset\ZZ_{\geq0}^{n+1}$ be such
that $\{\uz^{\fb}\}$ is a basis of 
\[
R:=\frac{\CC\{\uz\}}{J(f)}\cong\frac{\CC[\uz]}{J(f)},
\]
where $J(f)=\left(\tfrac{\partial f}{\partial z_{1}},\ldots,\tfrac{\partial f}{\partial z_{n+1}}\right).$
Corresponding to this isomorphism (and special to the quasi-homogeneous
case) is a homeomorphism between the Milnor fiber $\fF_{f,t_{0}}$
and $Z_{f}:=\{f(\uz)=1\}$.

More precisely, writing $w_{i}=\tfrac{u_{i}}{v_{i}}$ ($u_{i},v_{i}\in\mathbb{N}$
relatively prime), $d=\mathrm{lcm}\{v_{i}\}$, and $\tilde{w}_{i}=w_{i}d$,
let $\fz=\{f(\uz)=t^{d}\}\subset\CC^{n+2}$ be the base-change (of
$\CC^{n+1}$), and \[
\xymatrix{**[l] \mathrm{Bl}_{\utw}(\CC^{n+2}) \supset \fy  \ar @{->>} [r]_{\pi} \ar @/^1pc/ [rr]^{g} & \fz \ar [r]_t & \CC }
\]its weighted blow-up at $\uo$ (cf.~Remark I.8.4). The fiber $g^{-1}(0)=:\hat{Y}_{0}$
is the union of the proper transform $\hat{Z}_{0}$ of $Z_{0}\cong\{f(\uz)=0\}\subset\CC^{n+1}$
and the exceptional divisor\begin{equation} \label{eq2.2b}
\hat{\cE}:=\{f(\underline{\mathsf{Z}})=\mathsf{T}^{d}\}\subset\mathbb{WP}^{n+1}[1:\utw]=:\mathbf{P},
\end{equation}with\begin{equation} \label{eq2.2c}
\hat{\cE}\cap\hat{Z}_{0}=:\hat{E}=\{f(\underline{\mathsf{Z}})=0\}\subset\mathbb{WP}^{n}[\utw]=:\mathbf{H},
\end{equation}and obvious isomorphisms
\[
\mathbf{P}\backslash\hat{\cE}\cup\mathbf{H}\cong\CC^{n+1}\backslash Z_{f}\;\;\;\text{and}\;\;\;\hat{\cE}\backslash\hat{E}\cong Z_{f}.
\]
The hats are there to remind the reader of the possible presence of
singularities, which are however insignificant for the Hodge-theoretic
analysis. Indeed, by Prop. I.8.3 and Remark I.8.4, the
proof of Prop. I.6.3 (i.e. \eqref{I6}) extends verbatim to identify the vanishing
cohomology:
\begin{prop}
\label{prop2.2a} We have $\pphi_{f}\QQ_{\CC^{n+1}}[n+1]\cong(\imath_{\uo})_{*}V_{f}$,
where
\[
V_{f}\cong H^{n}(\hat{\cE}\backslash\hat{E})\cong H^{n}(Z_{f}).
\]

\end{prop}
For each $\ub\in\fb$, set
\[
\ell(\ub):=\frac{1}{d}\sum_{i=1}^{n+1}(\beta_{i}+1)\tilde{w}_{i}\;,\;\;\;\;\alpha(\ub):=n+1-\ell(\ub)\,,
\]
and define a rational form by
\[
\omega_{\ub}:=\frac{\uz^{\ub}dz_{1}\wedge\cdots\wedge dz_{n+1}}{(f(\uz)-1)^{\left\lceil \ell(\ub)\right\rceil }}\in\Omega^{n+1}(\CC^{n+1}\backslash Z_{f})\,,
\]
with class $[\omega_{\ub}]\in H^{n+1}(\CC^{n+1}\backslash Z_{f})$
and image
\[
\eta_{\ub}:=\mathrm{Res}_{Z_{f}}([\omega_{\ub}])\in V_{f}
\]
under the residue map $H^{n+1}(\CC^{n+1}\backslash Z_{f})\overset{\cong}{\underset{\mathrm{Res}}{\to}}H^{n}(Z_{f}).$

For what follows, it will be convenient to write $\{c\}\in [0,1)\cap\QQ$ for the
fractional part of $c\in\QQ$, and $\langle c\rangle:=1+\lfloor c\rfloor+\lfloor-c\rfloor$
($1$ for $c\in\ZZ$, $0$ otherwise).
\begin{thm}\label{T2.2}
The $\{\eta_{\ub}\}_{\ub\in\fb}$ give a basis of $V_{f}$, with $\eta_{\ub}\in V_{f,\{\alpha(\ub)\}}^{\left\lfloor \alpha(\ub)\right\rfloor ,\left\lfloor \ell(\ub)\right\rfloor }$\emph{;}
and $N$ acts trivially%
\footnote{$N$ may still act nontrivially on $H_{\lm}$ in \eqref{eq2.1a} (though
with $N^{2}=0$).%
} on $V_{f}$.\end{thm}
\begin{proof}[Sketch]
 The ``basis'' assertion is Steenbrink's weighted-projective extension
of Griffiths's residue theory \cite[Thm.~1]{St1}.  For the remaining assertions, rewrite
\begin{equation}\label{th2.2eq}
\omega_{\ub}=\frac{\mathsf{T}^{d\{-\ell(\ub)\}-1}\underline{\mathsf{Z}}^{\ub}\Omega_{\mathbf{P}}}{(f(\underline{\mathsf{Z}})-\mathsf{T}^d)^{\lceil \ell(\ub)\rceil}}
\end{equation}
in weighted projective coordinates, where $$\Omega_{\mathbf{P}}=\mathsf{T}d\mathsf{Z}_{1}\wedge\cdots\wedge d\mathsf{Z}_{n+1}-\mathsf{Z}_{1}d\mathsf{T}\wedge\cdots\wedge d\mathsf{Z}_{n+1}+\cdots+(-1)^{n+1}\mathsf{Z}_{n+1}d\mathsf{T}\wedge\cdots\wedge d\mathsf{Z}_{n}.$$
(The power of $\mathsf{T}$ is the one that makes $\omega_{\ub}$ a well-defined section of $\Omega_{\mathbf{P}}^{n+1}(*\hat{\cE})$ over $\mathbf{P}\setminus\mathbf{H}$.) This blows up at $\infty$ (i.e.~along $\mathbf{H}=\{\mathsf{T}=0\}$) iff $\ell(\ub)\in\ZZ$, in which case it has a \emph{simple} pole.
Conclude that for $\ell(\ub)\in\ZZ$, 
\begin{equation*}
\mathrm{Res}_{\mathbf{H}}\omega_{\ub}=\frac{\underline{\mathsf{Z}}^{\ub}\Omega_{\mathbf{H}}}{f(\underline{\mathsf{Z}})^{\ell(\ub)}}\;\;\;\implies\;\;\;
\mathrm{Res}_{\hat{E}}\mathrm{Res}_{\mathbf{H}}\omega_{\ub}\in\gr_{F}^{n-\ell(\ub)}H_{\text{prim}}^{n-1}(\hat{E})=H^{n}(\hat{\cE}\backslash\hat{E})^{n+1-\ell(\ub),\ell(\ub)} ,
\end{equation*} while if $\ell(\ub)\notin\ZZ$, 
\[
\mathrm{Res}_{\hat{E}}\omega_{\ub}\in\gr_{F}^{n+1-\lceil\ell(\ub)\rceil}H_{\text{prim}}^{n}(\hat{\cE})=H^{n}(\hat{\cE}\backslash\hat{E})^{n-\lfloor\ell(\ub)\rfloor,\lfloor\ell(\ub)\rfloor}.
\]
The weight filtration therefore has length $\leq1$, and so $N$ acts
trivially. On the other hand, the action of $T^{\text{ss}}$ on $V_{f}\cong H^{n}(\hat{\cE}\backslash\hat{E})$
is induced by $\mathsf{T}\mapsto\zeta_{d}\mathsf{T}$, or equivalently
\[
[\mathsf{T}:\mathsf{Z}_{1}:\cdots:\mathsf{Z}_{n+1}]\longmapsto[\mathsf{T}:\zeta_{d}^{-\tilde{w}_{1}}\mathsf{Z}_{1}:\cdots:\zeta_{d}^{-\tilde{w}_{n+1}}\mathsf{Z}_{n+1}].
\]
This affects only the numerator of \eqref{th2.2eq}, so its eigenvalue on $\omega_{\ub}$ is $\zeta_{d}^{d\{-\ell(\ub)\}}=\zeta_{d}^{d\{\alpha(\ub)\}}=e^{2\pi\sqrt{-1}\{\alpha(\ub)\}}$.
\end{proof}
\begin{rem}
Notice that by the triviality of $N$, we have $\mathscr{P}_{n}(V_{f})=W_{n}V_{f}$
($\cong\bar{H}^{n}(\cE)$ in \eqref{eq2.1d}).\end{rem}
\begin{cor}
\label{cor2.1a} We have $\mu_{f}=\mu^0_{f}=|\fb|$, and
\[
\tilde{\sigma}_{f}=\sum_{\ub\in\fb}\left[(\alpha(\ub),n+\langle\alpha(\ub)\rangle)\right].
\]

\end{cor}

\subsection{Weighted Fermat singularities and deformations}\label{S-wf}

We now turn to some examples of quasi-homogeneous singularities.  To begin, suppose that $f=\sum_{i=1}^{n+1}z_{i}^{d_{i}}$, with $d=\mathrm{lcm}\{d_{i}\}$;
that is, $f$ defines a \emph{Fermat quasi-homogeneous} (FQH) singularity.  (These are also called \emph{Brieskorn-Pham} singularities in the literature.)
In follows at once from Cor. \ref{cor2.1a} that 
\[
\fb=\left(\times_{i=1}^{n+1}[0,d_{i}-2]\right)\cap\ZZ^{n+1},
\]
so that $\mu_{f}=\prod_{i=1}^{n+1}(d_{i}-1)$ and (using the symmetry $\imath_n$) $\sigma_{f}=\sum_{\ub\in\fb}\left[\sum_{i=1}^{n+1}\tfrac{\beta_{i}+1}{d_{i}}\right]$.
In particular,
\[
\sigma_{f}^{\text{min}}:=\text{min}\{\alpha\mid\alpha\in|\sigma_{f}|\}=\sum_{i=1}^{n+1}\frac{1}{d_{i}}.
\]
As we shall see in $\S$\ref{S2.4} in more generality,
\footnote{The forward implication already follows from Prop. \ref{prop2.1b}.%
} 
\[
(Z_{f},\uo)\;\text{is }\left\{ \begin{array}{c}
\text{du Bois}\\
\text{rational}
\end{array}\right.\iff\;\sigma_{f}^{\text{min}}\;\text{is }\left\{ \begin{array}{c}
\geq1\\
>1.
\end{array}\right.
\]
To obtain further examples from this one, we need an observation 
about deformation of isolated singularities more generally.

First let $f$ be any isolated hypersurface singularity, and
$f_{s}(\uz)=F(\uz,s):\,\disc^{n+1}\times\disc\to\CC$
a 1-parameter holomorphic deformation thereof (with $f=f_{0}$, and $f_s(0)=0$ $(\forall s\in \Delta)$).  
We assume that the Milnor number $\mu_{f_{s}}=\mu_{f}$ is constant; 
equivalently (cf. \cite{Gr86}), the polar curve
$\{\partial_{z_{i}}F=0\,(\forall i)\}\cap(\disc^{n+1}\times\disc)$
is simply $\{\uo\}\times\disc\underset{\imath}{\hookrightarrow}\disc^{n+2}$,
along which $\imath^{*}\pphi_{F}\CC_{\disc^{n+2}}[n+1]\simeq\mathcal{V}_{f_{s}}$
then yields a VMHS. Evidently, the action of $T^{\text{ss}}$ --- 
and thus $\tilde{\sigma}_{f_{s}}(=\tilde{\sigma}_{f})$ --- must be constant in $s$.  
That is, \emph{the weighted spectrum is invariant under $\mu$-constant deformations} 
(as stated for spectra in \cite{Va3}). 

By a result of Varchenko \cite{Va5}, 
the $\mu$-constant deformations of quasi-homogeneous (QH) isolated singularities
are the \emph{semi-quasihomogeneous} (SQH) ones.  
More precisely, given a weight vector $\uw\in\QQ_{>0}^{n+1}$
as above and an isolated singularity of the form 
$\tilde{f}=\sum_{\underline{m}\cdot\uw\geq1}a_{\underline{m}}\uz^{\underline{m}}$,
one has equality in $\mu_{\tilde{f}}\geq\prod_{i}(w_{i}^{-1}-1)$
if and only if the weight-1 part $f:=\sum_{\um\cdot\uw=1}a_{\um}\uz^{\um}$
defines an isolated singularity (which is precisely the SQH condition) \cite{FT}. 
Informally, given a QH $f$ (of ``weight 1'') defining an isolated singularity, 
nearby QH deformations (by adding small weight-$1$ terms) are $\mu$-constant,
as are (SQH) deformations by adding arbitrary terms of weight $>1$.

The upshot is that for those isolated QH and SQH singularities which are
$\mu$-constant deformations of FQH singularities, we can easily compute
the weighted spectrum.
\begin{example}
\label{ex2.2a} If $f$ is one of the simple elliptic singularities
$\tilde{E}_{r}$ ($f_{\tilde{E}_{6}}=x^{3}+y^{3}+z^{3}+\lambda xyz$,
$f_{\tilde{E}_{7}}=x^{2}+y^{4}+z^{4}+\lambda xyz$, or $f_{\tilde{E}_{8}}=x^{2}+y^{3}+z^{6}+\lambda xyz$,
with $\lambda$ avoiding the discriminant locus), then it is a QH-deformation
of an FQH-singularity, hence has $\tilde{\sigma}_{f}=[(1,3)]+[(2,3)]+\sum_{j=1}^{r}[(1+\lambda_{j},2)]$
with $\{\lambda_{j}\}=\{\tfrac{1}{3},\tfrac{1}{3},\tfrac{1}{3},\tfrac{2}{3},\tfrac{2}{3},\tfrac{2}{3}\}$,
$\{\tfrac{1}{4},\tfrac{1}{4},\tfrac{1}{2},\tfrac{1}{2},\tfrac{1}{2},\tfrac{3}{4},\tfrac{3}{4}\}$
resp. $\{\tfrac{1}{6},\tfrac{1}{3},\tfrac{1}{3},\tfrac{1}{2},\tfrac{1}{2},\tfrac{2}{3},\tfrac{2}{3},\tfrac{5}{6}\}$.
Since this puts the weight-2 and 3 parts of $V_{f}=H^{2}(\cE_{\lambda}\backslash E_{\lambda})$
in distinct $T^{\text{ss}}$-eigenspaces, the extension class of the MHS
is trivial and its only moduli come from $H^{1}(E_{\lambda})(-1)$.
That is, all $\mathit{AJ}$-classes of divisors in the image of $\mathrm{Pic}^{0}(\cE_{\lambda})\to\mathrm{Pic}^{0}(E_{\lambda})$
are torsion. (This justifies Example I.6.2.)
\end{example}

\begin{example}
\label{ex2.2b} If $f=z^{2}+g_{M}(x,y)$, where $g_{M}$ is a product
of $M\geq2$ linear forms, the form of the MHS on $V_{f}$ bifurcates
into \[\inputencoding{latin1}{\includegraphics[scale=0.7]{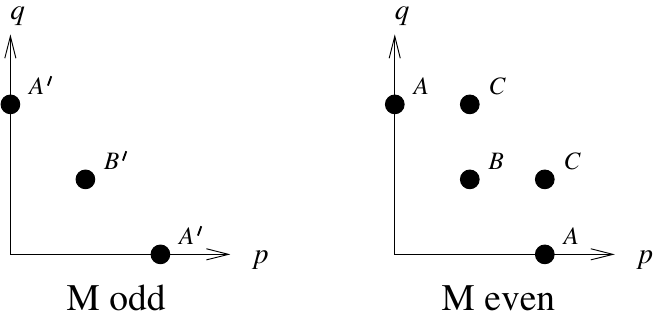}}\inputencoding{latin9}\]with
$A'=\tfrac{(M-1)(M-3)}{8}$, $B'=\tfrac{(3M-1)(M-1)}{4}$, $A=\tfrac{(M-2)(M-4)}{8}$,
$B=\tfrac{3M^{2}-6M+4}{4}$, and $C=\tfrac{M}{2}-1$. Decomposing
into $T^{\text{ss}}$-eigenspaces according to the spectrum shows that the
period map for deformations of such singularities takes values in
a product of Hermitian symmetric domains: ${\rm I}_{1,M-3}\times{\rm I}_{2,M-4}\times\cdots\times{\rm I}_{\frac{M-3}{2},\frac{M-1}{2}}$
($M$ odd) or ${\rm I}_{1,M-3}\times{\rm I}_{2,M-4}\times\cdots\times{\rm I}_{\frac{M-4}{2},\frac{M}{2}}\times{\rm III_{\frac{M}{2}-1}}$
($M$ even). Extension classes are again zero.

Note that $M=5$ is the $N_{16}$ singularity studied in \cite{La1},
whose quasi-homogeneous deformations are captured (up to analytic
isomorphism) by the 2-ball ${\rm I}_{1,2}$. On the other hand, the
period map for $V_{f}$ cannot capture strictly SQH-deformations (which
are not ``seen'' by the tail $\hat{\cE}\backslash\hat{E}$). For
$z^{2}+x^{5}+y^{5}$, such a deformation is given by adding $sx^{3}y^{3}$,
which changes the analytic equivalence class; it can be detected by
the more refined ``period map'' given by the Brieskorn lattice \cite{Sa3}.
\end{example}
The reader will notice that in Examples \ref{ex2.2a} and \ref{ex2.2b},
some or all of the $\gr_{3}^{W}$ part of $V_{f}$ could be absorbed
into $H_{\pha}^{3}$ in \eqref{eq2.1a}, rather than contributing
to $\mathrm{coim}(N)$ in $H_{\lm}^{2}(X_{t})$. In general, addressing
``where the $T^{\text{ss}}$-invariant part of $\gr_{n+1}^{W}H_{\van}^{n}$
goes'' is difficult without further assumptions (cf. Prop. \ref{prop2.4b})
in $\S$\ref{S2.4}).
\begin{example}
\label{ex2.2c} Given $\cx\overset{t}{\to}\disc$, $\cx$ smooth
(of dimension $n+1$), $p_{g}:=h^{0,n}(X_{t})$ ($t\neq0$), and $P\in X_{0}$
an ordinary $k$-tuple point%
\footnote{$X_{0}$ may have other isolated hypersurface singularities%
} ($t$ locally a SQH-deformation of $f_{k}=x_{1}^{k}+\cdots+x_{n+1}^{k}$)
we consider the relationship between $p_{g}$ and $k$. Writing $v^{p,q}:=\dim(V_{f_{k}}^{p,q})$,
we compute $v^{0,n}={k-1 \choose n+1}$ and $v^{1,n}={k-1 \choose n}$.
Since $H_{\lm}^{n}(X_{t})^{0,n}\twoheadrightarrow V_{f_{k}}^{0,n}$
we must have at least $p_{g}\geq{k-1 \choose n+1}$. If we know that
$H_{\pha}^{1,n}=\{0\}$ (e.g., if $n=1$ and $X_{0}$ is irreducible;
or in the scenario of Prop. \ref{prop2.4b}), then $H_{\lm}^{n}(X_{t})^{0,n-1}\overset{\cong}{\underset{N}{\leftarrow}}H_{\lm}^{n}(X_{t})^{1,n}\twoheadrightarrow V_{f_{k}}^{1,n}$
as well, so that $p_{g}\geq v^{0,n}+v^{1,n}={k \choose n+1}$. (If
the $\{X_{t}\}$ are hypersurfaces in $\PP^{n+1}$ of degree $d$,
this merely says that $d\geq k+1$.) Note that the case $k=3=n$ is
the $O_{16}$ singularity studied in \cite{LPZ}, with only $v^{2,1}=v^{1,2}=5$
and $v^{2,2}=6$ nonzero.
\end{example}

\begin{example}
\label{ex2.2d} One can ask more generally which isolated singularity
types (on $X_{0}$) are ruled out for $\{X_{t}\}$ of a fixed nature.
For instance, if the $\{X_{t}\}$ are cubic threefolds, and $P\in X_{0}$
has type $A_{k}$ ($f$ locally analytically equivalent to $x^{2}+y^{2}+z^{2}+w^{k+1}$),
then $\lfloor\tfrac{k}{2}\rfloor=v^{1,2}\leq h_{\lm}^{1,2}\leq5$
$\implies$ $k\leq11$. (Moreover, there would have to be nontrivial
$H_{\pha}^{2,2}$ to allow $k=11$, since $v^{2,2}=1$ for $k$ odd.)
In fact this is sharp, since $y^{3}+z^{3}+x^{2}w-yw-2xyz+w^{3}$ is
type $A_{11}$ \cite{Al} in analytically equivalent disguise.
\end{example}

\subsection{Nodes on odd-dimensional hypersurfaces}\label{s-schoen}

When the spectrum of an isolated hypersurface singularity has some integer support, this may produce an ambiguity in the effect on the limiting MHS.  More precisely, in the setting of \eqref{I1} (or \eqref{eq2.1b}) with $\cx$ smooth, the $\gr^W_{n+1}$ part of $\ker(T-I)\subseteq H^n_{\van}(X_t)$ may contribute either to $H^{n+1}_{\pha}(X_0)$ or to $H^n_{\lm}(X_t)$, or split its contributions between the two.  We shall touch further on this phenomenon in $\S$\ref{S2.4} (Prop. \ref{prop2.4b}) below, but focus here on how one may disambiguate these contributions in the first interesting case, where $\text{sing}(X_0)$ consists of finitely many nodes and $n$ is odd.  

A well-known example is the Fermat quintic 3-fold family in $\PP^4$, defined by $X_t:=\{(t+5)\sum_{i=0}^4 \mathsf{X}_i^5 = \prod_{i=0}^4 \mathsf{X}_i\}$, which acquires $125$ nodes at $t=0$ (each with spectrum $\tilde{\sigma}=[(2,4)]$).  Denoting the $125$ vanishing cycles by $\{\varphi_i\}$, the monodromy logarithm $N=\sum_{i=1}^{125}\langle \,\cdot\, ,\varphi_i\rangle\varphi_i$ only has rank $101$.  That is, some of $H_{\van}^3 = \QQ(-2)^{\oplus 125}$ contributes to $H^4_{\pha}(X_0)\cong \QQ(-2)^{\oplus 24}$, and some to $(H^3_{\lm})^{2,2}\cong \QQ(-2)^{\oplus 101}$ (which in fact is the maximum possible, since $h^{2,1}(X_t)=101$).  We now show how a simple method for quantifying this ``splitting of $H^n_{\van}$ in general can be deduced from a result of Schoen \cite{Schoen}.

For this subsection (only), $\mathbf{P}$ is any smooth, irreducible, projective variety of dimension $n+1=2m$ satisfying Bott vanishing:  that is, for every ample line bundle $\mathcal{L}$, we have $H^i(\mathbf{P},\Omega^j_{\mathbf{P}}(\mathcal{L}))=\{0\}$ for all $i>0$ and $j\geq 0$.  Let $\cx\to\Delta$ be a family of ample hypersurfaces in $\mathbf{P}$, with smooth total space, and $\text{sing}(X_0)=:\mathsf{S}=\{\mathsf{S}_i\}_{i=1}^{\rd}$ consisting of $\rd$ nodes.  Write $\tilde{X}_0$ [resp. $\tilde{\mathbf{P}}$] for the blowup of $X_0$ [resp. $\mathbf{P}$] along $\mathsf{S}$, with exceptional divisors $\mathsf{Q}=\{\mathsf{Q}_i\}_{i=1}^{\rd}$ [resp. $\mathsf{P}=\{\mathsf{P}_i\}_{i=1}^{\rd}$].  Let $r$ denote the rank of the map $$\mathrm{ev}_{\mathsf{S}}\colon H^0(\mathbf{P},K_{\mathbf{P}}(m X_0))\to \CC^{\rd}$$ evaluating sections at $\mathsf{S}$.\footnote{Warning: this map incorporates a noncanonical identification of each fiber of the line bundle $K_{\mathbf{P}}(mX_0)$ with $\CC$ (e.g., by dividing by a fixed section vanishing at none of the $\mathsf{S}_i$, if one exists).}

\begin{thm}\label{t:schoen}
The monodromy logarithm $N = T-I \in \mathrm{End}(H_{\lm}^n(X_t))$ has rank $r$.  Consequently $W_{n-1}H^n(X_0)=H^n(X_0)^{m-1,m-1}$ has dimension $r$ and $H_{\pha}^{n+1}(X_0)=H_{\pha}^{n+1}(X_0)^{m,m}$ has dimension $\rd-r$.
\end{thm}

\begin{proof}
Mayer--Vietoris and weak Lefschetz yield the exact sequence of MHS $$0\to H^{n-1}(\tilde{\mathbf{P}})\overset{\alpha}{\to} H^{n-1}(\tilde{X}_0)\overset{\beta}{\to} \oplus_{i=1}^{\rd}\tfrac{H^{n-1}(\mathsf{Q}_i)}{H^{n-1}(\mathsf{P}_i)}\overset{\gamma}{\to} W_{n-1}H^n(X_0)\to 0.$$ By Schoen's result \cite[Prop. 1.3]{Schoen}, $$\mathrm{im}\{\beta^{\vee}=\gy\colon \oplus_{i=1}^{\rd} H_{\mathrm{pr}}^{n-1}(\mathsf{Q}_i)\to H^{n+1}(\tilde{X}_0)\} \cong H^1 (\mathbf{P},\mathcal{I}_{\mathsf{S}}\otimes K_{\mathbf{P}}(mX_0))\, ,$$ the right-hand side of which is isomorphic to $\mathrm{coker}(\mathrm{ev}_{\mathsf{S}})$ by Bott vanishing.  So $\mathrm{rk}(\beta)=\dim(\mathrm{coker}(\mathrm{ev}_{\mathsf{S}}))=\rd -r$, hence $\dim W_{n-1}H^n(X_0)=r$.  Now each node has weighted spectrum $[(m,2m)]$.  Applying the resulting vanishing-cycle sequence $$0\to H^n(X_0) \overset{\sp}{\to} H^n_{\lm}(X_t) \overset{\can}{\to} \QQ(-m)^{\oplus \rd}\overset{\delta}{\to} H^{n+1}_{\pha}(X_0)\to 0$$ gives $\mathrm{rk}(\can)=\dim\{(H^n_{\lm})^{m,m}\}=\dim\{(H^n_{\lm})^{m-1,m-1}\}=\dim\{H^n(X_0)^{m-1,m-1}\}=r$ hence $\dim H^{n+1}_{\pha}(X_0)=\rd-r$.
\end{proof}

\begin{rem}
The following (somewhat) heuristic interpretation may be of use.  The image of $\mathrm{ev}_{\mathsf{S}}$ represents type-$(m,m-1)$ forms on $\tilde{X}_0$ with residues giving primitive $(m-1,m-1)$-classes on the $\{\mathsf{Q}_i\}$, i.e. \emph{killing} the image of $\gy$.  So $\mathrm{rk}(\beta^{\vee})=\rd -r$.  These same relations may also be viewed topologically as chains on $X_0$:  first, one has retraction maps $\bar{\mathfrak{F}}_{\mathsf{S}_i} \overset{\eta_i}{\twoheadrightarrow} \bar{\mathcal{N}}_{\mathsf{S}_i} \subset X_0$ from the Milnor fibers to a neighborhood of each node, which restrict to isomorphisms along the boundaries.  Representing $H^n_{\van}(X_t)=\oplus_{i=1}^{\rd} H^n(\mathfrak{F}_{\mathsf{S}_i})\cong \oplus_{i=1}^{\rd} H_n(\bar{\mathfrak{F}}_{\mathsf{S}_i},\partial\bar{\mathfrak{F}}_{\mathsf{S}_i})$ by $\rd$ $n$-cycles $\{\Gamma_i\}$, the inclusions $\eta_i(\partial\Gamma_i)=:\gamma_i$ of their boundaries into $X_0$ (as elements of $H^{n+1}(X_0)$!) compute the image of $\delta$. One then has $r$ ``relations-$n$-chains'' $\{\mathsf{R}_j\}_{j=1}^r$ whose boundaries $\partial \mathsf{R}_j = \sum_{i=1}^{\rd}c_i^{(j)}\gamma_i$ with $\{\underline{c}^{(j)}\}_{j=1}^r \subset \QQ^{\rd}$ linearly independent vectors; and the $r$ $n$-cycles $\mathsf{R}_j - \sum_{i=1}^{\rd} c_i^{(j)} \eta_i(\Gamma_i)$ are independent in $H_n(X_0)$, by pairing with $r$ of the vanishing spheres (which thereby span $H^n(X_0)$).
\end{rem}

\begin{cor}\label{c:schoen}
If $\cx$ is a family of degree $d\geq 3$ hypersurfaces in $\PP^{2m}$, and $X_0$ is nodal with at least one node, then $\mathrm{rk}(N)\geq 1$.  In particular, if $X_0$ has exactly one node, then $H^{n+1}_{\pha}(X_0)=\{0\}$.
\end{cor}

\begin{proof}
Since $K_{\PP^{2m}}(mX_0)\cong \mathcal{O}_{\PP^{2m}}(md-2m-1)$ is base-point free, $\mathrm{rk}(\mathrm{ev}_{\mathsf{S}})\geq 1$.
\end{proof}

\section{Singularities with Calabi--Yau tail\label{S2.3}}
The general theme of our study is the interplay between Hodge theory, singularities, and moduli problems. Many of our results here and in Part I are concerned with the du Bois singularities, which should be understood as 
being relatively mild for Hodge theory of degenerations \cite{steenbrink}. These are also closely related to the slc singularities \cite{KK}, the singularities relevant for the KSBA compactifications. Frequently, however, one encounters non-du Bois singularities, for instance when studying GIT quotients.  So it is of interest to understand the Hodge-theoretic behavior of such singularities, and the Hodge theory of the tail when one performs a semi-stable or KSBA replacement of the singular fiber containing them. The simplest such situation is that of a family of curves degenerating to a cuspidal curve: the semi-stable replacement of the cuspidal fiber will introduce an elliptic tail (see \cite{hassett} for more examples and discussion of the $1$-dimensional case). Here we discuss higher dimensional analogues of this example -- namely, singularities that lead to $K3$ tails, or more generally Calabi--Yau tails. This is partially motivated by the role played by triangle (i.e. Dolgachev) singularities in the study of degenerations of $K3$ surfaces (see \cite[Sect. 6]{log2}). 

\begin{rem}
A concrete application of these kinds of ideas to the study of KSBA compactifications for moduli of surfaces of general type is \cite{GPSZ} (which appeared subsequent to our manuscript), where the authors produce boundary divisors for the moduli of Horikawa surfaces by the KSBA replacements associated to singularities of the kind studied in this section. We emphasize however that the study here is local (in the spirit of the stable local reduction for curves \cite{hassett}), as opposed to the global situation of \cite{GPSZ}.
\end{rem}

Referring to the setting of \eqref{eq2.2a}-\eqref{eq2.2c} (with $\mathbf{P}$ denoting $\mathbb{WP}^{n+1}[1:\utw]$), we begin with
\begin{defn}
\label{def2.2} A quasi-homogeneous singularity is said to have a
(\emph{pure} resp. \emph{mixed}) \emph{CY tail} if $\hat{\cE}\subset\mathbf{P}$
is a quasi-smooth anticanonical hypersurface (and $H^{n}(\hat{\cE}\backslash\hat{E})$
is pure resp. mixed).
\end{defn}
An immediate consequence of the definition is that the spectrum must have
\begin{equation}\label{e3.00}
\left||\sigma_{f}|\cap(0,1)\right|=1
\end{equation}
(and $|\sigma_{f}|\cap\ZZ=\emptyset$ in the pure case); this weaker
condition defines the notion of a (pure resp. mixed) \emph{numerical
CY tail}.
\begin{rem}
\label{rem2.2.2} In Definition \ref{def2.2}, we retain the hypotheses
on $f$ in \eqref{eq2.2a}: in particular, that also $\hat{E}=\hat{\cE}\cap\mathbb{WP}[\tilde{\uw}]$
be quasi-smooth. In principle, this may exclude some cases where $W_{n+1}V_{f}\subsetneq V_{f}$,
but has no effect in the classification below for $n=1$ or $2$.
Beyond consistency within this section, there is a good reason for
this restriction: by ensuring that the weights of $V_{f}$ are only
$n$ and $n+1$, it focuses attention on the singularities which \emph{may}
have only pure weight $n$, ``Calabi--Yau HS'' contribution to $H_{\lm}^{n}$
(if the weight $n+1$ part is absorbed into $H_{\pha}^{n+1}$ in \eqref{eq2.1a}).
\end{rem}
Let $\Gamma_{\uw}$ {[}resp. $\NP_{\uw}${]} denote the convex
hull of $\mathfrak{M}(\uw)$ {[}resp. $\mathfrak{M}(\uw)$ and $\uo${]}.
The numerical condition on the weight vector that is actually equivalent
to Defn. \ref{def2.2} is twofold:\begin{enumerate}[label=\bf(\alph*),leftmargin=0.8cm]
\item triviality of $K_{\hat{\cE}}$: by adjunction, this means that%
\footnote{$\mathbb{WP}^{n+1}[1:\tilde{w}_{1}:\cdots:\tilde{w}_{n+1}]$ is automatically
well-formed if \eqref{eq2.2d} holds; i.e. any $(n+1)$-element subset
of $\{1,\tilde{w}_{1},\ldots,\tilde{w}_{n+1}\}$ has $\gcd$ 1.%
} \begin{equation} \label{eq2.2d}
\frac{1}{d} + \sum_{i=1}^{n+1} w_i = 1,\;\;\text{where}\; d=\mathrm{lcm}\{v_i\} \text{; and}
\end{equation}
\item existence of a quasi-smooth $\hat{\cE}\subset\mathbf{P}$
of degree $d$, which translates to%
\footnote{quasi-smoothness of $\hat{E}$ amounts to the additional hypothesis
that $\Gamma_{\uw}\cap\ZZ_{>0}^{n+1}\subset\mathrm{int}(\Gamma_{\uw})$.%
}\begin{equation}\label{eq2.2e}
(1,\ldots,1)\in \mathrm{int}(\NP_{\uw}).
\end{equation}
\end{enumerate}
\begin{example}
\label{ex2.2e} Let $n=1$. By the classification of simple elliptic
singularities \cite{KSa}, the $t\mapsto t^{d}$ base-change of a quasihomogeneous
curve singularity with elliptic tail must be one of $\tilde{E}_{6}$,
$\tilde{E}_{7}$, or $\tilde{E}_{8}$. Consequently the possibilities
are $f\sim x^{2}+y^{3}$ (cusp/$A_{2}$) for the pure case, and $f\sim x^{2}+y^{4}$
(tacnode/$A_{3}$) or $x^{3}+y^{3}$ (ordinary triple point/$D_{4}$)
for the mixed case. The contribution of the latter two to $H_{\lm}^{1}$
may nevertheless be pure if $X_{0}$ is reducible (cf. Remark \ref{rem2.2.2}).
\end{example}

\begin{example}
\label{ex2.2f} In the Fermat case $f=\sum_{i=1}^{n+1}z_{i}^{d_{i}}$,
we have $w_{i}=\tfrac{1}{d_{i}}$ and \eqref{eq2.2d} becomes%
\footnote{\eqref{eq2.2e} is automatic in the FQH case%
} \begin{equation} \label{eq2.2f}
\frac{1}{\mathrm{lcm}\{d_i\} } + \sum_{i=1}^{n+1} \frac{1}{d_i} = 1.
\end{equation}The resulting CY tail is pure iff \begin{equation} \label{eq2.2g}
\left\{ \sum_{i=1}^{n+1}\frac{\beta_i}{d_i} \mid \beta_i \in (0,d_i)\cap \ZZ \right\} \cap \ZZ = \emptyset.
\end{equation}For example, this yields that the threefold ($n=3$) FQH singularities
with pure CY tail are \begin{multline*}
x^4 + y^4 + z^4 + w^6,\;\; x^3 + y^4 + z^4 + w^8,\;\; x^3 + y^4 + z^5 + w^5,\\ x^3 + y^4 + z^6 + w^6,\;\; x^3 + y^5 + z^5 + w^5,
\end{multline*}and suspensions of surface FQH singularities ($x^{2}+\cdots$), which
are quite a bit more numerous.
\end{example}
For $n=2$ we now give a complete list:
\begin{thm}
\label{TCY} The isolated quasi-homogeneous hypersurface singularities
with $K3$ tail are as enumerated in Tables \ref{t2.2a} and \ref{t2.2b}.%
\footnote{The tables only list the simplest form of each singularity; the general
QH- and SQH-deformations are left to the reader.%
} In the pure tail case, these are exactly the $14$ Dolgachev singularities
\cite{Do}, the $6$ quadrilateral singularities \cite{Eb}, and $2$
trimodal singularities.\end{thm}
\begin{proof}
The $t\mapsto t^{d}$ base-change of $f$ must be in one of the $95$
classes of $K3$ hypersurfaces in weighted projective spaces identified
by Reid (unpublished) and Yonemura \cite{Yo}. Discard from Yonemura's
list of weight 4-tuples all but those of the form $w_{1,}w_{2},w_{3},\tfrac{1}{d}$
as in \eqref{eq2.2d} which also satisfy \eqref{eq2.2e}, and distinguish
pure from mixed by whether $g:=|\mathrm{int}(\Gamma_{\uw})\cap\ZZ^{3}|>0$
(which is the genus of $\hat{E}$).\end{proof}
\begin{rem}
In the tables, ``Arnol'd'' and ``Yonemura'' refer to the Arnol'd
classification (on which the subscript is the Milnor number $\mu_{f}=\dim(V_{f})$)
and the position in Yonemura's list, respectively. In Table \ref{t2.2a},
which is partitioned into the three types mentioned in the Theorem,
$m_{f}$ denotes the modality (cf. Remark \ref{rem222.1}). In Table
\ref{t2.2b}, $g$ denotes the genus of $\hat{E}$; the $H^{1}(\hat{E})(-1)$
part of $V_{f}$ may or may not be absorbed into $H_{\pha}^{3}$.
\end{rem}
\begin{table}[ht]
\caption{Pure $K3$ tail}
\centering
\begin{tabular}{ccccc}
\hline\hline
Yonemura & Arnol'd & $f$ & $d$ &  $m_f$ \\ [0.5ex]
\hline
$4$ & $U_{12}$ & $x^3 + y^3 + z^4$ & $12$ & $1$ \\
$9$ & $W_{12}$ & $x^2 + y^4 + z^5$ & $20$ & $1$ \\
$13$ & $E_{14}$ & $x^2 + y^3 +z^8$ & $24$ & $1$ \\
$14$ & $E_{12}$ & $x^2 + y^3 + z^7$ & $42$ & $1$ \\
$20$ & $Q_{10}$ & $x^2 z + y^3 + z^4$ & $24$ & $1$ \\
$22$ & $Q_{12}$ & $x^2 z + y^3 + z^5$ & $15$ & $1$ \\
$37$ & $W_{13}$ & $x^2 + y^4 + y z^4$ & $16$ & $1$ \\
$38$ & $Z_{11}$ & $x^2 + y^3 z + z^5$ & $30$ & $1$ \\
$39$ & $Z_{13}$ & $x^2 + y^3 z + z^6$ & $18$ & $1$ \\
$50$ & $E_{13}$ & $x^2 + y^3 + y z^5$ & $30$ & $1$ \\
$58$ & $S_{11}$ & $x^2 z + x y^2 + z^4$ & $16$ & $1$ \\
$60$ & $Q_{11}$ & $x^2 z + y^3 + y z^3$ & $18$ & $1$ \\
$78$ & $Z_{12}$ & $x^2 + y^3 z + y z^4$ & $22$ & $1$ \\
$87$ & $S_{12}$ & $x^2 z + x y^2 + y z^3$ & $13$ & $1$ \\
\hline
$8$ & $W_{15}$ ($W_{1,0}$) & $x^2 + y^4 + z^6$ & $12$ & $2$ \\
$12$ & $E_{16}$ ($J_{3,0}$) & $x^2 + y^3 + z^9$ & $18$ & $2$ \\
$18$ & $U_{14}$ ($U_{1,0}$) & $x^3 + y^3 + x z^3 - y z^3$ & $9$ & $2$ \\
$24$ & $Q_{14}$ ($Q_{2,0}$) & $x^2 z + y^3 + z^6$ & $12$ & $2$ \\
$40$ & $Z_{15}$ ($Z_{1,0}$) & $x^2 + y^3 z + z^7$ & $14$ & $2$ \\
$63$ & $S_{14}$ ($S_{1,0}$) & $x^2 z + x y^2 + y^2 z^2 + z^5$ & $10$ & $2$ \\
\hline
$6$ & $N_{16}$ ($\mathit{NA}_{0,0}$) & $x^2 + y^5 + z^5$ & $10$ & $3$ \\
$19$ & $V_{15}$ ($\mathit{VA}_{0,0}$) & $x^2 y + x^2 z + y^4 + z^4$ & $8$ & $3$ \\ [1 ex]
\hline
\end{tabular}
\label{t2.2a}
\end{table}\begin{table}[ht]
\caption{Mixed $K3$ tail}
\centering
\begin{tabular}{ccccc}
\hline\hline
Yonemura & $f$ & $d$ & $\mu_f$ & $g$ \\ [0.5ex]
\hline
$1$ & $x^4 + y^4 + z^4$ & $4$ & $27$ & $3$  \\
$3$ & $x^3 + y^3 + z^6$ & $6$ & $20$ & $1$  \\
$5$ & $x^2 + y^6 + z^6$ & $6$ & $25$ & $2$  \\
$7$ & $x^2 + y^4 + z^8$ & $8$ & $21$ & $1$  \\
$10$ & $x^2 + y^3 + z^{12}$ & $12$ & $22$ & $1$  \\
$25$ & $x^2 z + y^3 + z^9$ & $9$ & $20$ & $1$  \\
$42$ & $x^2 + y^3 z + z^{10}$ & $10$ & $21$ & $1$  \\
$66$ & $x^2 z + x y^2 + y^3 z + z^7$ & $7$ & $20$ & $1$  \\ [1 ex]
\hline
\end{tabular}
\label{t2.2b}
\end{table}
\begin{rem}\label{rem222.1}
We comment on the pure case. The \emph{modality} $m_{f}$ is the dimension
of the moduli space of $\mu$-constant SQH deformations of $f$ (up
to analytic isomorphism). In each case in Table \ref{t2.2a}, the
period map for the VHS on $V_{f}$ has $1$-dimensional fibers, taking
values in a $(m_{f}-1)$-dimensional ball quotient. (This is visible
in the spectrum, which reveals that the $\zeta_{d}$-eigenspace of
$T^{\text{ss}}$ in $V_{f}$ has Hodge numbers $h^{2,0}=1$, $h^{1,1}=m_{f}-1$,
$h^{0,2}=0$.) To get any sort of Torelli result, one therefore has
to refine the Hodge-theoretic period map to a Brieskorn-lattice-theoretic
one \cite{Sa3}.

The singularities with $m_{f}=1$ (resp. $2,3$) are the Dolgachev/triangular
(resp. quadrilateral, ``pentagonal'') singularities, for which $\hat{\cE}$
has $3$ (resp. $4,5$) type-$A$ singularities, on $\hat{E}\cong\PP^{1}$.
(The resolution graph for $\tilde{X}_{0}\twoheadrightarrow X_{0}$
therefore looks like $3$ resp. $4,5$ ``arms'' attached to a central
node.) The moduli of these $m_{f}+2$ points on $\PP^{1}$ is essentially
what the period map detects.
\end{rem}

\begin{rem}
The first ``non-example'' one encounters (number $2$ in Yonemura's
list) is $f=x^{4}+y^{4}+z^{3}$ ($V_{18}'$), with spectrum $[\tfrac{10}{12}]+2[\tfrac{13}{12}]+[\tfrac{14}{12}]+3[\tfrac{16}{12}]+2[\tfrac{17}{12}]+2[\tfrac{19}{12}]+3[\tfrac{20}{12}]+[\tfrac{22}{12}]+2[\tfrac{23}{12}]+[\tfrac{26}{12}].$
So this $f$ has ``numerical $K3$'' tail as defined above;\footnote{That is, it satisfies \eqref{e3.00} (not \eqref{eq2.2d}).} we briefly
explain why it does \emph{not} have an actual $K3$ tail in our sense.

Denote the base-change by $t\mapsto t^{6}$ resp. $t^{12}$ by $\mathfrak{Z}'$
resp. $\mathfrak{Z}$, and the weighted blow-ups by $\mathcal{Y}'$
resp. $\mathcal{Y}$. The exceptional divisors are the $K3$ surface
$\hat{\cE}'\subset\mathbb{WP}[2:3:3:4]$ and the ``fake $K3$''
$\hat{\cE}\subset\mathbb{WP}[1:3:3:4]$ (with $h^{2,0}=1$ but $K_{\hat{\cE}}\neq\mathcal{O}_{\hat{\cE}}$).
The $K3$ has $4$ $A_{2}$ singularities along $\hat{E}'\cong\mathbb{P}^{1}$
and $3$ $A_{1}$ singularities not on $\hat{E}'$. More importantly,
at this ``partial base-change'' stage there is still order-two monodromy.
Writing $\rho:\hat{\cE}\backslash\hat{E}\to\hat{\cE}'\backslash\hat{E}'$
for the natural $2:1$ cover, we have $\rho_{*}\QQ\cong\QQ\oplus\chi$;
and the $(-1)$-eigenspace of $T^{\text{ss}}$ on $H_{\lm}^{2}(Y_{t}')$
is $H_{\van}^{2}(Y_{t}'):=\mathbb{H}^{0}(\pphi_{t}\QQ_{\cy'}[3])\cong H^{2}(\hat{\cE}'\backslash\hat{E}',\chi)\cong\frac{H^{2}(\hat{\cE}\backslash\hat{E})}{H^{2}(\hat{\cE}'\backslash\hat{E}')}\cong\QQ(-1)^{\oplus8}.$ 
\end{rem}

\section{Isolated hypersurface singularities: birational invariants\label{S2.4}}
We now focus on du Bois singularities, rational singularities, and even better singularities, the so-called {\it $k$-log-canonical} singularities. The idea is that for du Bois (and rational) singularities, there is a tight connection between the frontier Hodge numbers of the central fiber $X_0$ and those of the LMHS associated to a smoothing $\fx$ (e.g. see \cite{KL1} and \cite{KLS}). In the lowest dimensions, this is enough for comparing geometric compactifications of moduli spaces with Hodge-theoretic compactifications. A prototype of this is the 
result of Mumford and Namikawa  saying that there exists a map $\overline{\mathcal M}_g\to \overline{\mathcal A}_g^{Vor}$ from the Deligne--Mumford compactification of $\mathcal M_g$ to a specific toroidal compactification of the associated period domain $\mathcal A_g$. In other words, for curves, the geometric limit determines the Hodge-theoretic limit (even at the level of extension data for LMHS). Similar phenomena occur for $K3$ surfaces or even hyper-K\"ahler manifolds and related objects (e.g. \cite{Friedman,FS}, \cite{klsv}, \cite{La10}). 

\medskip

In higher dimension the picture is more subtle: it still holds that $X_0$ having slc singularities implies that it is du Bois, and then du Bois implies cohomologically insignificant (see \cite{KK} and \cite{steenbrink}); but this controls at most the frontier Hodge numbers. Here, we improve on this by considering the {\it $k$-log-canonical singularities}. 
 This is a concept that emerged recently in the context of the work of Mustata and Popa on Hodge ideals (see also Remarks \ref{Rem-kDB} and \ref{Rem-krat}). Specifically, a key result of the section is  Corollary \ref{cor2.4a} which gives a tighter relationship between the LMHS vs. MHS of the central fiber under the assumption of $k$-log-canonical singularities. One further point here is that the $k$-log-canonicity can be characterized in terms of the spectrum (Prop. \ref{prop2.4a}), which in turns leads to easy examples of $k$-log-canonical singularities by iterated suspensions of lower dimensional singularities (see \S\ref{sec-ex-klog}).

\medskip

Let $\fx$ be a smooth variety of dimension $n+1$, and $X_0\subset \fx$ a hypersurface with an isolated singularity at $p$, locally cut out by $f\in\co_{\fx,p}$.  
Given a log-resolution $\tilde{\fx}\overset{\pi}{\to}\fx$ of $X_0$, and a holomorphic form $\omega \in \Omega^{n+1}(U)$ defined on a small neighborhood $U\subset \fx$ of $p$, residue theory suggests the question: \emph{for which $\kappa\in\QQ_{>0}$ does $\tilde{\Omega}:=\pi^*(\tfrac{\omega}{f^{\kappa}})$ have no worse than log poles on the NCD $\pi^{-1}(p)=\left(\bigcup_{i} \cE_i\right)\cup\tilde{X}_0$?}  
Writing (on $\tilde{U}=\pi^{-1}(U)$)\footnote{Typically the $a_i$ (resp. $b_i$) are called discrepancies (resp. ramification numbers).  The computation here is heuristic, with $\kappa \in \QQ$.} 
\begin{equation}
K_{\tilde{\fx}}-\pi^* K_{\fx} =: \Sigma_i \; a_i \cE_i \;\;\;\;\;\text{and}\;\;\;\;\;\pi^* X_0 - \tilde{X}_0 =: \Sigma_i \; b_i \cE_i \; ,
\end{equation}
taking $\pi^*$ of $\tfrac{\omega}{f^{\kappa}}\in \Gamma \left( U ,K_{\fx}(\kappa X_0)\right)$ gives 
\begin{equation} \label{eq2.4(1)}
\tilde{\Omega}\in \Gamma \{ \tilde{U}, K_{\tilde{\fx}}(\kappa \tilde{X}_0 + \Sigma_i (\kappa b_i - a_i)\cE_i )\}  .
\end{equation}
Our choice of $\kappa$ must prevent the coefficients of $\tilde{X}_0$ and $\cE_i$ in \eqref{eq2.4(1)} from exceeding $1$; equivalently, $\kappa$ must not exceed the \emph{log-canonical threshold}
\begin{equation} \label{eq2.4(2)}
\alpha_f := \min \left\{1,\tfrac{1+a_i}{b_i} \right\} .
\end{equation}
More precisely, if we identify $U$ as an open in $\CC^{n+1}$ and $f$ as the restriction of a polynomial map from $\CC^{n+1}\to \CC$, then this gives a sharp answer to our question for $\omega:=d\underline{x}:=dx_1 \wedge \cdots \wedge dx_{n+1}$. Of course, $\alpha_f =1$ $\iff$ $(\fx,X_0)$ is a log-canonical pair $\iff$ $(X_0,p)$ is log-canonical. 

On the other hand, if $\alpha_f <1$, it is the smallest of the \emph{jumping numbers} $J_f\subset \QQ_{>0}$ for the \emph{multiplier ideals}
\begin{equation}
\mathcal{A}_{\kappa}:=\left\{\phi\in\co_{\fx,p}\mid\;\tfrac{1+a_i +\nu_i(\phi)}{b_i}>\kappa \;(\forall i) \right\} ,
\end{equation}
where $\nu_i(\phi):=\mathrm{ord}_{\cE_i}(\phi\circ \pi)$.  
These comprise all ``sharp thresholds'' for all $\omega=\phi\,d\underline{x}$.  
At the same time, the jumping numbers greater than $1$ are not independent of the choice of $\pi$, so are not invariants of $f$.

An invariant with similar motivation is introduced by Mustata and Popa \cite{MP} in the context of the \emph{Hodge ideals} $\{ \mathcal{I}_r(X_0)\}_{r\in \ZZ_{\geq 0}}$, which measure the difference between the pole-order and Hodge filtrations along $X_0$:
\begin{equation}
F_r \co_{\fx}(*X_0)=\co_{\fx}((r+1)X_0)\otimes \mathcal{I}_r(X_0),
\end{equation}
where $(\co_{\fx}(*X_0),F_{\bullet})$ is thought of as the filtered $\mathcal{D}$-module underlying the mixed Hodge module $j_* \QQ_{\fx\setminus X_0}[n+1]$.  
Call $(X_0,p)$ \emph{$k$-log-canonical} if $\mathcal{I}_r(X_0)\overset{\mathrm{loc}}{=}\co_{\fx}$ for $r\leq k$, where ``loc'' means ``in a neighborhood of $p$'', and write $\lambda_f -1$ for the maximal\footnote{If $(X_0,p)$ is not ($0$-)log-canonical, we set $\lambda_f=0$.} such $k$ ($\geq 0$ $\iff$ $(X_0,p)$ is log-canonical).  
Write $\ell_f$ for the \emph{generation level} \cite{Sa5} of $\co_{\fx}(*X_0)$, i.e. the smallest integer $\ell$ for which $F_k \mathcal{D}_{\fx}\cdot F_{\ell}\co_{\fx}(*X_0)=F_{k+\ell}\co_{\fx}(*X_0)$ for all $k\geq 0$.  
Unlike $\alpha_f$, however, both $\lambda_f$ and $\ell_f$ are integers, and $\ell_f$ can change under $\mu$-constant deformations.

The asymptotic behavior of periods provides a quantity more closely tethered to the log-canonical threshold.  Let $\gamma(t)$ be a multivalued section of $\{H_n(\mathfrak{F}_t,\ZZ)\}_{t\neq 0}$, where $\mathfrak{F}_t = (f|_U)^{-1}(t)$ is the Milnor fiber (and $\mathfrak{F}_0=X_0\cap U$).  Consider the collection of periods \cite{Va2}
\begin{equation} \label{eq2.4(3)}
\mathscr{P}^{\omega,\gamma}(t) := \int_{\gamma(t)}\mathrm{Res}_{\mathfrak{F}_t}\left(\tfrac{\omega}{f-t}\right) = \frac{1}{t}\sum_{\alpha\in\QQ_{>0}}\sum_{k=0}^n a^{\omega,\gamma}_{k,\alpha}t^{\alpha}\log^k(t)
\end{equation}
taken over all $\omega$ and $\gamma$, and define the \emph{period exponent}
\begin{equation} \label{eq2.4(4)}
\beta_f := \min\{\alpha\mid a^{\omega,\gamma}_{k,\alpha}\neq 0 \text{ for some }\omega,\gamma,k\} .	
\end{equation}
Since the $\lim_{t\to 0}$ of $t^{1-\kappa}\int_{\gamma(t)}\mathrm{Res}_{\mathfrak{F}_t}\left(\tfrac{\omega}{f-t}\right)=\int_{\gamma(t)}\mathrm{Res}_{\mathfrak{F}_t}\left( \tfrac{\omega f^{1-\kappa}}{f-t}\right)$ is finite as long as $\pi^*\left(\tfrac{\omega}{f^{\kappa}}\right)$ has log poles (i.e. $\kappa\leq \alpha_f$), and infinite if $\kappa>\beta_f$ by \eqref{eq2.4(3)}-\eqref{eq2.4(4)}, we see at once that $\alpha_f \leq \beta_f$.

A connection between these invariants and the spectrum\footnote{Recall in particular from $\S$\ref{S2.2} that $\sfm$ denotes the smallest element of $|\sigma_f|$.} is already suggested by \eqref{eq2.4(3)}-\eqref{eq2.4(4)}: given $a^{\omega,\gamma}_{k,\beta_f} \neq 0$, we may assume (by applying $T^{\text{un}}-I$ to $\gamma$) that $k=0$; and replacing $\gamma$ by a suitable flat section $\hat{\gamma}\in \{ H_n(\mathfrak{F}_t,\CC)\}_{t\neq 0}$, that $\int_{\hat{\gamma}(t)}\mathrm{Res}_{\mathfrak{F}_t}(\tfrac{\omega}{f-t})=t^{\beta_f -1}h(t)$ with $h$ holomorphic and nonvanishing at $0$.  
Let $\cv_f^{\vee}$ be the canonical extension of $\{H_n(\mathfrak{F}_t,\CC)\}$, with fiber $V_f^{\vee}$ at $0$.  
Writing $\beta_f -1 =c-\lambda$ with $\lambda\in [0,1)$ and $c\in \ZZ_{\geq 0}$, $\tilde{\gamma}(t):=t^{\lambda}\hat{\gamma}(t)$ is a section of $\cv^{\vee}_f$ (not vanishing at $t=0$).  
Pairing this with $\eta:=(\nabla_{\partial_t})^c \,\mathrm{Res}_{\mathfrak{F}_t}( \tfrac{\omega}{f-t})\in \Gamma(\disc,\mathcal{F}^{n-c}_e\cv_f)$ yields $\int_{\tilde{\gamma}(t)}\eta = t^{\lambda}(\tfrac{d}{dt})^c t^{\beta_f -1}h(t) = c\cdot h(t)+\co(t)$, so that $\gr_F^{n-c} V_{f,\lambda}\neq\{0\}$ and $n-c+\lambda \in |\sigma_f|$ $\implies$ $n+1-(n+c+\lambda)=\beta_f\in |\sigma_f|$ $\implies$ $\sfm\leq \beta_f$.\footnote{The reverse inequality holds too, since (for $\sfm=c'-1+\lambda'$) \emph{given} $\eta'\in \Gamma(\disc,\mathcal{F}_e^{n-c'}\cv_f )$ and $\tilde{\gamma}'-t^{\lambda '}\hat{\gamma}'$ with $\lim_{t\to 0}\int_{\tilde{\gamma}(t)}\eta' \in \CC^*$, we can write $\eta' = \mathrm{Res}(\frac{\omega'}{(f-t)^{c+1}}) = (\nabla_{\partial_t})^c \mathrm{Res}(\frac{\omega'}{f-t})$.  We also indicate where to find a proof in the literature below.}

The following statement surveys what is known about the inter-relationships amongst all these invariants:

\begin{prop} \label{prop2.4a}
For isolated hypersurface singularities we have:
\begin{enumerate}[label=\textup{({\roman*})}]
\item $\sfm = \beta_f$.
\item $\sfm \geq \alpha_f$, with equality iff $\sfm\leq 1$.
\item $J_f\cap (0,1]=|\sigma_f|\cap (0,1]$.
\item $\lambda_f = \lfloor \sfm \rfloor$.
\item $(X_0,p)$ slc $\iff$ log-canonical $\iff$ du Bois $\iff$ $\sfm\geq 1$.
\item $(X_0,p)$ log-terminal $\iff$ canonical $\iff$ rational $\iff$ $\sfm>1$.
\item $\ell_f \leq \lfloor n-\sfm\rfloor$\textup{;} in particular, $\ell_f \leq n-2$ for $(X_0,p)$ rational and $\ell_f \leq n-1$ in general.
\end{enumerate}
\end{prop}
\begin{proof}
The \emph{Bernstein--Sato polynomial} $b_f(s)$ of $f$ \cite{Be} is the monic polynomial of smallest degree such that (in a neighborhood of $p$) we have $Pf^{s+1}=f^s b_f(s)$ for some nonzero $P\in \mathcal{D}_X[s]$. Let $\tilde{\alpha}_f$ denote the smallest root of $\hat{b}_f(s):=\tfrac{b_f(-s)}{1-s}$.  
Then by work of M. Saito \cite[(4.1.3)]{Sa3}, $\sfm=\tilde{\alpha}_f$; combining this with $\beta_f = \tilde{\alpha}_f$ (Varchenko, \cite{Va1}) gives (i), with $\alpha_f=\text{min}\{1,\tilde{\alpha}_f\}$ (Koll\'ar, \cite[Thm. 10.6]{kollarpairs}) gives (ii), and with $\lfloor \tilde{\alpha}_f \rfloor = \lambda_f$ (Saito, \cite[Cor. 2]{Sai16}; see also \cite[Cor. 9.9]{Po}) gives (iv). 
Since isolated hypersurface singularities are normal and Gorenstein, \cite[Cor. 5.24]{KM} yields $(X_0,p)$ rational $\implies$ canonical $\implies$ log-terminal $\implies$ rational; while (iv), Prop. \ref{prop2.1b}, and \cite{KK} give $(X_0,p)$ du Bois $\implies$ $\sfm\geq 1$ $\implies$ $\lambda_f\geq 1$ $\implies$ log-canonical $\implies$ slc $\implies$ du Bois.  To get the last ``$\Longleftarrow$'' of (vi), one can for instance appeal to \cite[Thm. 0.4]{Sa4}.  Finally, (iii) is \cite[$\S$4]{Va2}, and (vii) follows from  \cite[Thm. A]{MP-Inv} 
 (see also  \cite{Sa5} and \cite{MP,MOP}).
\end{proof}

Returning to the scenario where $X_0 \subset \cx \overset{f}{\to}\disc$ ($\cx$ smooth; $f$ proper, smooth over $\disc^*$) has a unique singular point (at $p$), here is one consequence of Prop. \ref{prop2.4a} which strengthens the results above (Prop. \ref{prop2.1b} and Thm. \ref{th-I1}) in the isolated hypersurface singularity case.

\begin{cor} \label{cor2.4a}
$(X_0,p)$ is $k$-log-canonical $\iff$ $\sfm\geq 1+k$, in which case
\begin{equation} \label{eq2.4C1}
\gr_F^p H^n(X_0) \cong \gr_F^p H_{\lm}^n (X_t) \;\;\;\text{for}\;\;0\leq p \leq k	
\end{equation}
and
\begin{equation} \label{eq2.4C2}	
\gr_F^p W_{n-1}H^n(X_0) = \{0\} \;\;\;\text{for} \;\; 0\leq p \leq k-1.
\end{equation}	
For singularities with $\sfm>1+k$, \eqref{eq2.4C2} also holds for $p=k$.
\end{cor}

Given the characterization of log canonical and rational singularities noted in Proposition \ref{prop2.4a}, by analogy, we define higher rational singularities\footnote{(Note added in proof) Subsequent to our work, the subject grew up tremendously. We refer to Remark \ref{Rem-krat} and references within for an overview.} as follows
\begin{defn}\label{def-krat}
Let $(X_0,p)$ be an isolated hypersurface singularity. We say that $(X_0,p)$ is {\it $k$-rational} if  $\sfm>1+k$.
\end{defn}

\begin{proof}[Proof of Cor.~\ref{cor2.4a}]
Referring to \eqref{eq2.1a}, \eqref{eq2.4C1} is obviously induced by $\sp$; and the ``$\iff$'' is just (iv).  For \eqref{eq2.4C2}, suppose $H^n_{\lm}(X_t)^{p,q}\neq \{0\}$ for some $p+q<n$; this is contained in an $N$-string centered about $p+q=n$, only the bottom of which can come from $H^n(X_0)$. So $\can(H^n_{\lm}(X_t)^{p+1,q+1})\neq\{0\}$, which forces $p+1\geq \lfloor \sfm\rfloor$ hence $p\geq k$.  Moreover, for $k+q<n$, \eqref{eq2.4C1} $\implies$ $H^n(X_0)^{k,q}\cong H^n_{\lm}(X_t)^{k,q}$ $\implies$ $T^{\text{ss}}$ acts trivially on the $N$-strings ending on $(H^n_{\lm})^{k,q}$.  So if $(H^n_{\lm})^{k,q}\neq \{0\}$ then $(H^n_{\lm})^{k+1,q+1}$ contributes $\{ k+1\}$ to $|\sigma_f|$.
\end{proof}

\begin{rem}\label{rem-cubic}
Combining Proposition \ref{prop2.4a} with the Sebastiani--Thom formula (see Section \ref{S2.6}), one sees that easy examples of $k$-log-canonical singularities can be obtained by iterated suspension of lower dimensional singularities (see \S\ref{sec-ex-klog} for precise statements). Here, we only point out that almost all of  the singularities that occur for GIT semistable cubic $4$-folds are either ADE singularities and or double suspensions of log canonical surface singularities. From the discussion of \S\ref{sec-ex-klog} (or direct computations and Proposition \ref{prop2.4a}), one sees that these singularities are $1$-log-canonical, which in turn recovers a significant part of 
 \cite{La10} (i.e. the behavior of the period map for mildly singular cubic fourfolds).
\end{rem}

We conclude this section with a discussion of the relationship between the singularities of the central fiber of a $1$-parameter degeneration and the monodromy of the family. The Picard--Lefschetz monodromy (defined below) should be understood as a relaxing of the natural notion of finite monodromy. This notion and the discussion below are inspired by the study of degenerations of hyper-K\"ahler manifolds and Calabi--Yau varieties (e.g. see \cite{klsv}, \cite{tosatti}, and \cite{wang}). The point here is that in higher dimensions (e.g. already for Calabi--Yau threefolds), the requirement of finite monodromy is too rigid, and should be replaced by Picard--Lefschetz monodromy. 

\begin{defn}
We shall say $\cx\overset{f}{\to}\disc$ has \emph{Picard--Lefschetz monodromy} if the following equivalent conditions hold:  (i) $\dim H^n_{\lm}(X_t)^{0,n}=\dim H^{0,n}(X_t)$; (ii) $\gr_F^0 W_{n-1} H^n_{\lm}(X_t)=\{0\}$; (iii) $H^n_{\lm}(X_t)$ has no $N$-strings terminating on the ``$q$-axis''.
\end{defn}

\begin{cor} \label{cor2.4b}
Assume $(X_0,p)$ is log-canonical. Then $(X_0,p)$ is canonical $\iff$ $\cx\overset{f}{\to}\disc$ has Picard--Lefschetz monodromy and $H^{1,n}_{\pha}(X_0)=\{0\}$.	
\end{cor}

\begin{proof}
Use Prop. \ref{prop2.4a}(vi) to replace the left-hand side by $\sfm>1$. This implies Picard--Lefschetz monodromy 	by Cor. \ref{cor2.4a}, and $H^{1,n}_{\pha}=\text{im}\{H^n_{\van}\to H^{n+1}(X_0)\}^{1,n}=\{0\}$ since it would otherwise contribute $\{1\}$ to $|\sigma_f|$ (in view of the trivial action of $T^{\text{ss}}$).

For the converse, we are assuming $\sfm\geq 1$ and trying to show $\sfm >1$. If $\{1\}\in |\sigma_f|$, then $((H^n_{\van})^{1,r})^{T^{\text{ss}}}\neq \{0\}$ for some $r$, and is in the image of $((H^n_{\lm})^{1,r})^{T^{\text{ss}}}$ unless $r=n$ and $H^{1,n}_{\pha}\neq \{0\}$ (ruled out by assumption). But for $\xi\in ((H^n_{\lm})^{1,r})^{T^{\text{ss}}}$ with $\can(\xi)\neq 0$, $N\xi$ is nonzero in $(H^n_{\lm})^{0,r-1}$ (also ruled out by assumption).
\end{proof}

This, of course, begs the question as to when we have $H^{1,n}_{\pha}=\{0\}$. For $\cx$ a family of curves ($n=1$), this is just irreducibility of $X_0$.  For $n>1$, we have the following:

\begin{prop} \label{prop2.4b}
Let $\cx\to \disc$ be a proper family of $p_g >0$ $n$-folds with smooth total space. Suppose the singular fiber $X_0$ has a single isolated log-canonical singularity $p$, which is not contained in the base locus of $|K_{\cx}|$.  Then $H^{1,n}_{\pha}=\{0\}$.
\end{prop}

\begin{rem}
\begin{enumerate}[label=(\roman*)]
\item The hypothesis on $|K_{\cx}|$ holds for example if the $X_{t\neq 0}$ are $K$-trivial, or if the $X_t$ are complete intersections of CY or general type in $\PP^N$.
\item The necessity of $p_g >0$ is demonstrated by taking $\cx$ a family of cubic surfaces, with $(X_0,p)$ of type $\tilde{E}_6$.
\item Corollaries \ref{cor2.4a}-\ref{cor2.4b} have obvious extensions to finitely many singularities.  Prop. \ref{prop2.4b} does not: for instance, if $\cx$ is a family of $K3$ surfaces, with two $\tilde{E}_6$ singularities on $X_0$.  (In this case, $\tilde{X}_0$ is an elliptic ruled surface and the semistable reduction has singular fiber of Kulikov type II.)
\end{enumerate}
\end{rem}

\begin{proof}[Proof of Prop. \ref{prop2.4b}]
This proceeds in three steps.  Let $\tilde{X}_0 \overset{\pi}{\to}X_0$ be a log-resolution, with exceptional divisor $E=\cup_i E_i$, restricting to an isomorphism from $\tilde{X}_0\setminus E$ to $X_0\setminus \{p\}$. Since $n\geq 2$, $p$ is a normal point, whence (via Zariski's Main Theorem) $E$ is connected.

$\underline{\text{Step 1}}$: \emph{Reduction from $H^{1,n}_{\pha}=\{0\}$ to surjectivity of $\mathrm{Res:}\,H^0(K_{\tilde{X}_0}(E))\to H^0(K_E)$}.  We may assume that $\pi$ arises as the restriction to $X_0$ of a semistable reduction $\cy\to\cx$ with exceptional divisor $\cE$.  The specialization map $\sp:H^{n+1}(X_0)^{1,n}\to H^{n+1}_{\lm}(X_t)^{1,n}$ factors as
\begin{equation}
H^{n+1}(X_0)^{1,n}\overset{\cong}{\to}H^{1,n}(\tilde{X}_0)\hookrightarrow \underset{H^{n+1}(Y_0)^{1,n}}{\underbrace{H^{1,n}(\tilde{X}_0)\oplus H^{1,n}(\tilde{\cE})}} 
\to \underset{H^{n+1}_{\lm}(X_t)^{1,n}}{\underbrace{\frac{H^{1,n}(\tilde{X}_0)\oplus [H^{1,n}(\tilde{\cE})]}{\text{im}(H^{0,n-1}(\tilde{E}))}}}
\end{equation}	
by purity of $H^{n+1}(X_0)$ and the weight-monodromy spectral sequence for the SSD (semistable degeneration) $\cy\to \disc$. (Here $[H^{1,n}(\tilde{\cE})]$ denotes the quotient by the image of $H^{0,n-1}(\cE^{(1)})$, with $\cE^{(1)}=\amalg_{i<j}\cE_i \cap \cE_j$.) So $\sp$ is injective provided $H^{0,n-1}(\tilde{E})\overset{\gy}{\to}H^{1,n}(\tilde{X}_0)$, or conjugate-dually $H^{0,n-1}(\tilde{X}_0)\to H^{0,n-1}(\tilde{E})$, is zero. Since the NCD $E$ is du Bois, this translates to the vanishing of the restriction map $H^{n-1}(\co_{\tilde{X}_0})\to H^{n-1}(\co_E)$, hence to injectivity of $H^{n-1}(\co_E)\to H^n(\co_{\tilde{X}_0}(-E))$, and Serre-dually to surjectivity of $H^0(K_{\tilde{X}_0}(E))\to H^0(K_E)$.

$\underline{\text{Step 2}}$: \emph{Adjunction}.  We may also assume that $\pi$ is the restriction of a log-resolution $\cy' \overset{\beta}{\to}\cx$ of $(\cx,X_0)$ with exceptional divisor $\cE'$.  Then writing $\tilde{X}_0=(\beta^* t)-\sum_i a_i \cE_i '$ and $K_{\cy'}=\beta^* K_{\cx} + \sum_i b_i \cE_i '$ (with $a_i,b_i\in \mathbb{N}$), we have $K_{\cy'} + \tilde{X}_0 \equiv \beta^* K_{\cx} + \sum_i (b_i -a_i) \cE_i '$, with $\beta^* K_{\cx}$ an effective divisor with support not meeting $\tilde{X}_0 \cup \cE'$.  Writing $E_i=\tilde{X}_0 \cap \cE_i '$, adjunction yields
\begin{equation}
K_{\tilde{X}_0}(E) = (K_{\cy'} + \tilde{X}_0 + \cE')|_{\tilde{X}_0} = \sum_i (b_i - a_i +1)E_i =:\sum_i \alpha_i E_i. 
\end{equation}
Clearly the $\alpha_i \in \ZZ$, and since $\tilde{X}_0$ is log-canonical, they are $\geq 0$. Writing $D:=\sum_i \alpha_i E_i (\geq 0)$, adjunction again gives $K_E \equiv \co_E (E\cdot D)$.  So $\mathrm{Res}$ identifies with the restriction map 
\begin{equation} \label{eq*prop2.4b}
H^0(\tilde{X}_0,\co(D))\to H^0(E,\co(E\cdot D)).
\end{equation}

$\underline{\text{Step 3}}$: \emph{Hodge Index Theorem}. Let $\tilde{S}\subset \tilde{X}_0$ be a general $(n-2)$-fold hypersurface section (i.e.~of dimension 2).  Clearly the surface $S:=\pi(\tilde{S})$ contains $p$, and we write $C_j:=\tilde{S}\cap E_j$.  Consider an ample divisor $A$ on $S$, and note that $\tilde{A}:=\pi^* A$ has $\tilde{A}^2 = A^2 >0$. For any effective divisor on $\tilde{S}$ of the form $\mathscr{D}=\sum_j \gamma_j C_j$ ($\gamma_j \geq 0$), we have $\mathscr{D}\cdot \tilde{A}=0$; so by the Hodge Index Theorem, $\mathscr{D}\cdot \mathscr{D}<0$ (unless all $\gamma_j=0$). It follows that $C_i\cdot\mathscr{D}<0$ for some $i$ (necessarily with $\gamma_i \neq 0$); therefore any $g\in H^0(C,\co(C\cdot \mathscr{D}))$ has $g|_{C_i}\equiv 0$, hence actually belongs to $H^0(C,\co(C\cdot \mathscr{D}'))$ where $\mathscr{D}'=\mathscr{D}-C_i$.  Continuing in this vein we find that $g\in H^0(C,\co_C)\cong \CC$.

If $G\in H^0(E,\co(E\cdot D))$, then applying this to $g:=G|_C$ shows that $G\in H^0(E,\co_E)$ (since $\tilde{S}$ was general). In other words, $\mathrm{RHS}\eqref{eq*prop2.4b}\subseteq H^0(E,\co_E)$ is at most $1$-dimensional (recall $E$ is connected), and $(H^0(\tilde{X}_0,\co_{\tilde{X}_0})\subseteq)\,H^0(\tilde{X}_0,\co_{\tilde{X}_0}(D))$ surjects onto it.
\end{proof}

\section{Isolated hypersurface singularities: spectral combinatorics\label{S2.5}}

In this section we explore two ways of generalizing the constructions of $\S$\ref{S2.2}: a geometric approach based on toric geometry which replaces Proposition \ref{prop2.2a}; and an algebraic one, related to the $V$-filtration on a Brieskorn lattice, which replaces Theorem \ref{T2.2}. We shall assume throughout that the singular fiber $X_0$ of $f:\cx\to\disc$ has a unique\footnote{The results that follow have, once again, obvious extensions to the case of multiple isolated singularities.} singularity at $p$, which (by abuse of notation) locally identifies with a (i) \emph{convenient} and (ii) \emph{nondegenerate} polynomial map $f=\sum_{\um\in \ZZ^{n+1}_{\geq 0}} a_{\um} \uz^{\um}:\, \CC^{n+1}\to \CC$.  That is, defining $\mathfrak{M}(f):=\{\um\in\ZZ_{\geq 0}^{n+1}\mid a_{\um}\neq 0\}$, we require that:
\begin{enumerate}[label=\bf(\roman*),leftmargin=0.8cm]
\item the \emph{Newton polytope} $\NP$ of $f$, given by the convex hull of $\bigcup_{\um\in\mathfrak{M}(f)} (\um + \RR_{\geq 0}^{n+1})\subseteq \RR^{n+1}_{\geq 0}$, have no noncompact faces not contained in the coordinate hyperplanes; and
\item for each compact $r$-face $\sigma$ of $\partial\NP$, the compactification of $\{f=0\}$ in $\PP_{\NP}$ have smooth (and reduced) intersection\footnote{described by the facet polynomial $(0=)\sum_{\um\in\sigma\cap\ZZ^{n+1}}a_{\um}\uz^{\um}$.} with the open torus-orbit $(\CC^*)^r$ in $\PP_{\sigma}$.
\end{enumerate}

\begin{rem}\label{rem-con}
We note that {\bf (i)} is an essentially vacuous constraint, since given an $f$ \emph{which already defines an isolated singularity}, we may arrange {\bf (i)} (without affecting the local analytic isomorphism class of $f$) by adding a Fermat polynomial $\sum z_i^{M_i}$ with $M_i \gg 0$.
\end{rem}

\subsubsection*{Notation}	

Writing $\NP^c := \overline{\RR_{\geq 0}^{n+1}\setminus \NP}^{\text{an}}$ and $\Gamma:=\NP\cap \NP^c$, let $\h:\,\RR^{n+1}_{\geq 0}\to \RR_{\geq 0}$ be the function defined by $\h(\Gamma)\equiv 1$ and $\h(r\underline{x})=r \h(\underline{x})$.  If $\Gamma=\cup \Gamma_i$ is the decomposition into common facets of $\NP$ and $\NP^c$, $\NP_i$ is the convex hull of $\{\uo\}$ and $\Gamma_i$, and $C_i$ the cone on $\{\uo\}$ through $\Gamma_i$, then $\h$ is linear on each $C_i$ and $\NP^c =\cup_i \NP_i$. For each $I=\{i_0,\ldots,i_{\ell}\}$ we write $\Gamma_I = \cap_{i\in I}\Gamma_i$ [resp. $\NP_I$, $C_I$]; and set $\Gamma^{[\ell]}:= \amalg_{|I|=\ell+1} \Gamma_I$ [resp. $\NP^{[\ell]}$, etc.]. We shall have frequent use for counting integer points $\Lambda^*(\sigma):=|\text{int}(\sigma)\cap \ZZ^{n+1}|$ in the interior of a polytope, and for refining these counts by $\Lambda_{\lambda}^*(\sigma):=|\text{int}(\sigma)\cap\ZZ^{n+1}\cap \h^{-1}(\lambda+\ZZ)|$, where $\lambda\in [0,1)\cap \QQ$. 

Finally, for any face $\tau$ of $\Gamma$, we write $\NP_{\tau}$ [resp. $C_{\tau}$] for the convex hull of $\{\uo\}$ and $\tau$ [resp. cone on $\{\uo\}$ through $\tau$], $d_{\tau}$ for the dimension of $\tau$, and $k_{\tau}$ for the dimension of the minimal \emph{coordinate} plane containing $\tau$.

\subsection{Toric geometry approach} \label{S5.1}

First assume that $f$ is \emph{simple}, i.e.
\begin{enumerate}[label=\bf(\roman*),leftmargin=0.9cm, resume]
\item all cones on $\NP$ at vertices of $\NP$ are simplicial.
\end{enumerate}
In this case we construct, after Steenbrink \cite{steenbrinkvan}, a diagram
\[ 
\xymatrix{\cy \ar [r] \ar [rd]_g & \fx \ar [r] \ar [d] & \cx \ar [d]^{f} & \tilde{\cx} \ar [l] \ar [ld]^{\tilde{f}} \\ & \disc \ar [r]^{(\cdot)^d} & \disc} 
\]
in which $\fx$ is the base-change (singular at $p$), while $\tilde{\cx}$, $\cy$, and irreducible components of $\tilde{f}^{-1}(0)$ and $g^{-1}(0)$ are quasi-smooth with quasi-normal crossings.\footnote{This simply means that the intersections are locally those of the coordinate hypersurfaces in the toric variety of a simplicial cone.} 

First, let $\tilde{\cx}$ be the ``blowup'' of $\cx$ which locally replaces $\CC^{n+1}=\PP_{\RR_{\geq 0}^{n+1}}$ by $\PP_{\NP}$, and write $(\tilde{f})=\tilde{X}_0+\sum d_i \mathscr{D}_i$. Here $\mathscr{D}_i \cong \PP_{\Gamma_i}$ and $\mathrm{lcm}(d_i)=:d$ is the lowest common denominator of $\h(\ZZ_{\geq 0}^{n+1})$. The quasi-SSD $(\cy,g)$ is obtained by normalizing the base-change of $\tilde{\cx}$, or (equivalently) by locally compactifying $f(\uz)=t^d$ in $\PP_{\NNP}$, where $\NNP\subset \RR^{n+2}_{\geq 0}$ denotes the convex hull of $\{0\}\times \NP$ and the ray $\theta:=(d+\RR_{\geq 0})\times \{\uo\}$.\footnote{We are regarding $(t,z_1,\ldots,z_{n+1})$ as local toric coordinates.} The compact faces of $\NNP$ are (together with integral lattices) isomorphic to the $\NP_i$, and so the components $\cE_i$ of $(g)=\tilde{X}_0 + \sum \cE_i$ are hypersurfaces in $\PP_{\NP_i}$. An action of $\mu_d$ on $\cy$ is given by $t\mapsto \zeta_d t$, with quotient $\tilde{\cx}$; equivalently, we can map $\cy\twoheadrightarrow \tilde{\cx}$ by using the natural projection $(\PP_{\NNP}\setminus \PP_{\theta}) \twoheadrightarrow \PP_{\NP}$. Its restrictions to $\PP_{\NP_i} \twoheadrightarrow \PP_{\Gamma_i}$ present each $\cE_i$ as a cyclic $d_i$-cover of $\PP_{\Gamma_i}$, branched along the hypersurface $E_i :=\tilde{X}_0 \cap \cE_i \subset \PP_{\Gamma_i}$.  Write as above $\cE_I := \cap_{i\in I}\cE_i$ and $\cE^{[\ell]}:= \amalg_{|I|=\ell+1}\cE_I$ (and similarly for $E$).

Next, define a cohomological motive ${\mathbf{M}}^{\bullet}_{\van}$ (say, in Levine's category $D^b_{\text{mot}}(\text{Sm}_{\CC})$ \cite{Lev}) by
\begin{multline} \label{eq2.5.1}
\mathbf{M}^i_{\van} := \left( \coprod_{k\geq \max\{1,-i\}} E^{[2k+i-1]}(-k)[-2k]\right) \amalg \left(\coprod_{k\geq \max\{0,-i\}}\cE^{[2k+i]}(-k)[-2k]\right)	\\ \left( \amalg \text{ pt.}\text{ , if }i=-1\right) ,\mspace{200mu}
\end{multline}
with alternating sums of Gysin and restriction morphisms between them.\footnote{We shall have no need to specify signs, which can be extracted from \cite{steenbrinkvan}. Note that \eqref{eq2.5.1} is closely related to the \emph{motivic Milnor fiber} of Denef and Loeser \cite{DL}.} Since the weight monodromy spectral sequence holds for a quasi-SSD,\footnote{This is shown in \cite{steenbrinkvan}; alternatively, one may argue as in \cite[$\S$8.3]{KL1}.} the spectral sequence $\mathbb{E}_1^{i,j}:= H^j(\mathbf{M}_{\van}^i)$ (with $d_1$ induced by the above morphisms) converges at $\mathbb{E}_2$.  This yields $\mathbb{E}_2^{i,j}\cong\gr^W_j H^{i+j}_{\van}$, so that $\mathbf{M}_{\van}^{\bullet}$ replaces the ``tail'' from $\S$\ref{S2.2}. Since $H^{i+j}_{\van}=\{0\}$ for $i+j\neq n$, the Hodge-Deligne numbers satisfy \[ h^{p,q}_{\lambda}(H^n_{\van}) = h^{p,q}_{\lambda}(\mathbb{E}_2^{n-p-q,p+q}) = \sum_{i=-n}^n (-1)^{i-(n-p-q)} h_{\lambda}^{p,q}(\mathbb{E}_1^{i,p+q})\] which yields the formula
\begin{multline} \label{eq2.5.2}
h_{\lambda}^{p,q}(H^n_{\van}(X_t))=\sum_{i=-n}^n (-1)^{n+p+q+i}\left\{\sum_{k\geq{1,-i}} h_{\lambda}^{p-k,q-k}(E^{[i+2k-1]})+\sum_{k\geq 0,-i} h_{\lambda}^{p-k,q-k}(\cE^{[i+2k]})\right\}\\ \left(\; +\;(-1)^{n-1}\text{ if }p=q=\lambda=0\;\right).\mspace{100mu}
\end{multline}
Here $h^{a,b}_0 (E^{[\ell]}): =h^{a,b}(E^{[\ell]})$ and $h^{a,b}_{\lambda}(E^{[\ell]}):=0$ for $\lambda\neq 0$, while $h^{a,b}_{\lambda}(\cE^{[\ell]})$ is the $e^{2\pi\sqrt{-1}\lambda}$-eigenspace of $t\mapsto \zeta_d t$. Danilov \cite{Da} supplemented \eqref{eq2.5.2} with combinatorial formulas for the Hodge-Deligne numbers of the toric hypersurfaces $E_I\subset \PP_{\Gamma_I}$ and $\cE_I \subset \PP_{\NP_I}$, which we may write as
\begin{equation}\label{eq2.5.3}
h^{a,b}(E_I)=\left\{ \begin{array}{cl} 0,&a\neq b,\,n-|I|-b \\ \sum_{\tau\subseteq \Gamma _I} {(-1)}^{d_{\tau}+b}\binom{d_{\tau}}{a}, &a=b<\tfrac{1}{2}(n-|I|) \\  \sum_{\tau\subseteq \Gamma _I}{(-1)}^{d_{\tau}+b+1}\binom{d_{\tau}}{a+1}, &a=b>\tfrac{1}{2}(n-|I|) \\ \begin{matrix} \sum_{\tau\subseteq \Gamma _I} {(-1)}^{b+1}\left\{ {(-1)}^{d_{\tau}} \binom{d_{\tau}}{a+1} + \right. \mspace{30mu}{} \\ \mspace{60mu}\left.\sum_{\ell\geq 0}{(-1)}^{\ell} \binom{d_{\tau}+1}{a+\ell+1}\Lambda^*(\ell \tau)\right\} , \end{matrix} &a=b=\tfrac{1}{2}(n-|I|) \\ \sum_{\tau\subseteq \Gamma _I} \sum_{\ell\geq 0}{(-1)}^{b+\ell+1} \binom{d_{\tau}+1}{a+\ell+1}\Lambda^*(\ell\tau), &a=n-|I|-b>b\end{array} \right.
\end{equation}
\begin{equation}\label{eq2.5.5}
h^{a,b}_0 (\cE_I)=h^{a,b}(\PP_{\Gamma _I})= \left\{ \begin{array}{cl} 0, & a\neq b \\ \sum_{\tau\subseteq \Gamma _I} {(-1)}^{d_{\tau}+b}\binom{d_{\tau}}{a},& a=b \end{array} \right.
\end{equation}
and
\begin{equation}\label{eq2.5.6}
h^{a,b}_{\lambda\neq 0}(\cE_I)= \left\{ \begin{matrix} 0\;\;\mbox{ unless }a+b = n-|I|+1,\mbox{ in which case:} \\ \sum_{\tau \subseteq \Gamma _I} \sum_{\ell \geq 0} {(-1)}^{\ell+b} \binom{d_{\tau}+1}{a+\ell+1} \left\{ \Lambda_{1-\lambda}^*((\ell+1)\NP_{\tau})-\Lambda_{1-\lambda}^*(\ell \NP_{\tau}) \right\} . \end{matrix} \right. 
\end{equation}
Before proceeding to some general consequences of \eqref{eq2.5.2}-\eqref{eq2.5.6}, we illustrate their power when $n=1$.
\begin{example}\label{ex5.1}
After analytic change of coordinates via $z_1= x + y^4$, $z_2=y + x^4$, the Newton polytope of $f_0 = z_1^4 z_2+ z_1 z_2^4 + z_1^2 z_2^2$ is
\[ \includegraphics[scale=0.7]{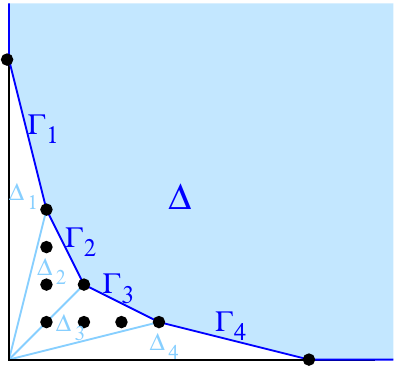} \]
so we get the same weighted spectrum by starting with
\[ f= z_1^8 + z_2^8 + z_1^4 z_2 + z_1 z_2^4 + z_1^2 z_2^2 . \]
The formulas yield $h_0^{1,1}(H^1_{\van})=h^{0,0}(E^{[0]})+h_0^{0,0}(\cE^{[1]})-h_0^{1,1}(\cE^{[0]})=4+3-4=3$, $h^{1,1}_{\frac{1}{2}}(H^1_{\van})=h^{0,0}_{\frac{1}{2}}(H^1_{\van})=\sum_{|I|=2}\Lambda^*_{\frac{1}{2}}(\NP_I) =1$, and $h^{0,1}_{\lambda}(H^1_{\van})=h^{1,0}_{1-\lambda}(H^1_{\van})=h^{1,0}_{1-\lambda}(\cE^{[0]})=\sum_{j=1}^4 \Lambda_{\lambda}^*(\NP_j)$ ($=2$ for $\lambda=\tfrac{2}{3},\tfrac{5}{6}$, and $0$ otherwise) producing a weighted spectrum
\[ \tilde{\sigma}_f = \underset{V_f^{0,0}}{\underbrace{[(\tfrac{1}{2},0)]}}+\underset{V_f^{0,1}}{\underbrace{2[(\tfrac{2}{3},1)]+2[(\tfrac{5}{6},1)]}}+\underset{V_f^{1,1}}{\underbrace{3[(1,2)]+[(\tfrac{3}{2},2)]}}+\underset{V_f^{1,0}}{\underbrace{2[(\tfrac{7}{6},1)]+2[(\tfrac{4}{3},1)]}} \]
with contributions to the $(p,q)$ pieces of $V_f=H^1_{\van}(X_t)$ as shown.
\end{example}

The first main result provides sufficient conditions for $k$-log-canonicity or $k$-rationality\footnote{take $c=k+1$ in Theorem \ref{thm2.5A}} of isolated hypersurface singularities in terms of the Newton polytope:

\begin{thm}\label{thm2.5A}
Let $c\in\ZZ_{>0}$, and write $\underline{1}=(1,\ldots,1)$. Under assumptions {\bf (i)}-{\bf (iii)}, we have $\sigma_f^{\text{min}}\geq c$ if $\underline{1}\in c\NP$ and $\sigma_f^{\text{min}}>c$ if $\underline{1}\in \textit{int}(c\NP)$.
\end{thm}
\begin{proof}
Rather than using \eqref{eq2.5.2}, observe that $\mathrm{Gr}^W_{p+q}H^n_{\text{van}}$ is the cohomology of the middle term of $\mathbb{E}_1^{n-p-q-1,p+q}\to \mathbb{E}_1^{n-p-q,p+q}\to \mathbb{E}_1^{n-p-q+1,p+q}$, which is a subquotient of
\begin{align*}
&\bigoplus_{k\geq \max\{0,p+q-n\}}\mathrm{coker}\left\{\mathrm{Gy}\colon H^{p+q-2k-2}(E^{[2k+n-p-q]})(-1)\to H^{p+q-2k}(\mathcal{E}^{[2k+n-p-q]})\right\}(-k)\\
\oplus &\bigoplus_{k\geq \max\{1,p+q-n\}}\ker\left\{\mathrm{Gy}\colon H^{p+q-2k}(E^{[2k+n-p-q-1]})\to H^{p+q-2k+2}(\mathcal{E}^{[2k+n-p-q-1]})(1)\right\}(-k).
\end{align*} 
Taking $h^{p,q}_{\lambda}$ of this for $\lambda\neq 0$ gives
\begin{equation}\label{e5*new}
\textstyle\sum_{k\geq 0,p+q-n}\sum_{|I|=2k+n-p-q+1}h_{\lambda}^{p-k,q-k}(\mathcal{E}_I)
\end{equation}
while $h^{p,q}_0$ yields
\begin{equation}\label{e5**new}
\textstyle\sum_{k\geq 1,p+q-n}\sum_{|I|=2k+n-p-q}\left\{\Lambda^*\text{-terms in }h^{p-k,q-k}(E_I)\right\}
\end{equation} 
as the terms of the type in \eqref{eq2.5.5} cancel with those in \eqref{eq2.5.3} via the Gysin maps.  For $\sigma_f^{\text{min}}>c$, we want \eqref{e5*new} and \eqref{e5**new} to vanish for $p\geq n-c+1$. For $\sigma_f^{\text{min}}\geq c$, we want \eqref{e5*new} to vanish for $p\geq n-c+1$ and \eqref{e5**new} to vanish for $p\geq n-c+2$.

If $\underline{1}\in \textit{int}(c\NP)$, then $\ZZ_{>0}^{n+1}\subset \text{int}(c\NP)$ and $\Lambda^*(\ell\tau)=0=\Lambda^*(\ell\NP_{\tau})$ for all $\ell\leq c$. 
When $p\geq n-c+1$, this forces \eqref{e5*new} and \eqref{e5**new} to vanish. 
For instance, the $\Lambda^*(\ell\tau)$ terms in \eqref{eq2.5.3} have a factor of $\binom{d_{\tau}+1}{p-k+\ell+1}$, with $\tau\subseteq \Gamma_I \implies d_{\tau}\leq n+1-|I|=p+q-2k+1$.  
For that factor to be nonzero, we must have $p-k+\ell\leq d_{\tau}$; but then 
$$\ell\leq d_{\tau}-p+k\leq q+1-k\leq q+1-(p+q-n)=n+1-p\leq n+1-(n-c+1)=c,$$
forcing $\eqref{e5**new}=0$. 
The $\Lambda^*((\ell+1)\NP_{\tau})$ terms in \eqref{eq2.5.6} have the same factor but with $d_{\tau}\leq p+q-2k$, forcing $\ell+1\leq c$ and $\eqref{e5*new}=0$.

If $\underline{1}\in c\NP$, then $\Lambda^*(\ell\NP_{\tau})=0$ for $\ell\leq c$ and $\eqref{e5*new}=0$; but we know $\Lambda^*(\ell\tau)=0$ only for $\ell\leq c-1$. Forcing $\ell\leq c-1$ (and $\eqref{e5**new}=0$) requires taking $p\geq n-c+2$.
\end{proof}

Now the condition that $\underline{1}\in c\NP$  or $\textit{int}(c\NP)$ is rather strong, even for $c=1$; indeed, many well-known singularities arising in moduli problems have $\sigma_f^{\text{min}}\leq 1$. In this case, one may expect $\gr_F^0 H^n_{\lm}$ to be affected by the degeneration. (We shall write in what follows $h^{a,b}_{\lm}:=h^{a,b}(H^n_{\lm}(X_t))$.)

Recalling that $h^{0,n}(X_{t\neq 0})=\sum_{j=0}^n h^{0,j}_{\lm}$, where $h^{0,j}_{\lm}=h^{0,j}_{\lm,0}+h^{0,j}_{\lm,\neq 0}$ and $h^{0,j}(X_0)=h^{0,j}_{\lm,0}$ ($0\leq j \leq n$), we have
\begin{equation}\label{eq2.5.9}
\left\{
\begin{matrix}
h^{0,j}_{\lm,\neq 0}\;\;\;\;\;(\;=h^{n-j,n}_{\lm,\neq 0}\;)&=\;h^{n-j,n}_{\van,\neq 0}&(0\leq j\leq n)
\\
h^{0,j}_{\lm, 0}\;\;\;\;\; (\;=h^{n-j,n}_{\lm, 0}\;)&=\;h^{n-j,n}_{\van, 0}&(0\leq j< n-1)
\\	
h^{0,n-1}_{\lm, 0}\;\;\;\;\;(\;=h^{1,n}_{\lm, 0}\;)&=\;h^{1,n}_{\van, 0} - h^{1,n}_{\pha}& {}
\end{matrix}
\right.
\end{equation}
since $N$-strings in $H^n_{\lm}$ starting at $(*,n)$ survive to $(0,n-*)$. The next result concerns these \emph{extremal $N$-strings}:

\begin{thm} \label{thm2.5B}
Assuming {\bf (i)}-{\bf (iii)}, we have
\begin{equation}\label{eq2.5.10}
\left\{
\begin{matrix}
h^{n-j,n}_{\van,\lambda(\neq 0)}&=&\sum_{|I|=n-j+1}\Lambda_{\lambda}^*(\NP_I)&(0\leq j\leq n)
\\
h^{n-j,n}_{\van, 0}&=&\sum_{|I|=n-j}\Lambda^*(\Gamma_I)&(1\leq j\leq n-1)
\\	
h^{n,n}_{\van, 0}&=&\sum_{|I|=n}\Lambda^*(\Gamma_I) + \sum_{|I|=n+1}1.
\end{matrix}
\right.
\end{equation}
If {\bf (iii)} fails, these formulas still hold if we replace $\{\Gamma_I\}_{|I|=a}$ \textup{[}resp. $\{\NP_I\}_{|I|=a}$\textup{]} by the set of $(n+1-a)$-dimensional faces $\gamma\subseteq \Gamma$ intersecting $\RR^{n+1}_{>0}$ \textup{[}resp. the convex hulls of $\{\uo\}$ and these $\gamma$\textup{]}.
\end{thm}

\begin{proof}
Apply \eqref{eq2.5.2}-\eqref{eq2.5.6} directly to $h^{n-j,n}_{\lambda}(H^n_{\van})=h^{n,n-j}_{\{1-\lambda\}}(H^n_{\van})$ to obtain the result assuming {\bf (iii)}.  To drop {\bf (iii)}, apply instead a result of Stapledon \cite[Thm. 6.20]{Sta} (again a straightforward computation, which is omitted to avoid introducing heaps of notation).	
\end{proof}

\begin{rem}
Theorem \ref{thm2.5B} includes results of Matsui and Takeuchi \cite{MT} in its $j=0$ and $1$ cases. Note that the formula for $h^{n,n}_{\van,0}$ is counting points in the intersection of $\ZZ_{>0}^{n+1}$ with the 1-stratum of $\Gamma$. 	
\end{rem}

\noindent An interesting scenario here is when the $X_{t\neq 0}$ are Calabi--Yau $n$-folds, so that one $h^{0,j_0}_{\lm} = 1$ and the other $\{h^{0,j}_{\lm}\}_{j\neq j_0}$ vanish:

\begin{defn}\label{def-Kgen}
We shall say in this case that the CY-degeneration has \emph{type} $n+1-j_0$ (sometimes expressed in Roman numerals).	
\end{defn}

\noindent This definition generalizes Kulikov's notation (I), (II), (III) for the semistable reduction in the $K3$ setting. Of course, (I) is equivalent to Picard--Lefschetz monodromy.

\begin{cor}\label{cor2.5B1}
For a CY-degeneration with $h^{1,n}_{\pha}=0$, smooth total space, and unique isolated singularity \textup{(}with Newton polytope $\NP$\textup{)}, $\NP^c \cap \ZZ^{n+1}_{>0}$ is either empty or $\{\underline{1}\}$, where $\underline{1}=(1,\ldots,1)$. Accordingly, this degeneration is:\footnote{with modifications as in Thm. \ref{thm2.5B} if {\bf (iii)} does not hold.}
\begin{itemize}[leftmargin=0.5cm]
\item of type $1$ if $\underline{1}\in \text{int}(\NP)$;
\item of type $|I|$ if $\underline{1}\in \text{int}(\NP_I)$;
\item of type $|I|+1$ if $\underline{1}\in \text{int}(\Gamma_I)$ \textup{(}for $|I|\leq n$\textup{)}; and
\item of type $n+1$ if $\underline{1}$ is a vertex of $\Gamma$. 
\end{itemize}
\end{cor}

\begin{proof}
If any point of $\ZZ_{>0}^{n+1}$ meets $\NP^c$, then $\underline{1}$ does (by convexity of $\text{int}(\NP)$). Now apply Theorem \ref{thm2.5B}.
\end{proof}

\noindent The corollary is closely related to work of Watanabe \cite{Wa}. Note that by Theorem \ref{thm2.5A}, the CY degeneration is log-canonical if $\underline{1}\in \NP$ and rational if $\underline{1}\in \text{int}(\NP)$. In fact, the converse statements are true, and much more generally we have:

\begin{cor}\label{cor2.5B2}
Under assumptions {\bf (i)}-{\bf (ii)}, the singularity $(X_0,p)$ is \textup{(a)} log-canonical iff $\textup{int}(\NP^c)\cap \ZZ^{n+1}_{>0}=\emptyset$, and \textup{(b)} rational iff $\NP^c\cap \ZZ^{n+1}_{>0}=\emptyset$.	
\end{cor}

\begin{proof}
(a):  $(X_0,p)$ log-canonical $\iff$ $\sigma_f^{\text{min}}\geq 1$ $\iff$ $h^{0,j}_{\van}=0$ ($\forall j$) $\iff$ $h^{n-j,n}_{\van,\lambda}=0$ ($\forall j,\forall \lambda\neq 0$)	$\underset{\text{Thm. }\ref{thm2.5B}}{\iff}$ $\text{int}(\NP_I)\cap \ZZ^{n+1}_{>0}=\emptyset$ ($\forall I$) $\iff$ $\text{int}(\NP^c)\cap \ZZ^{n+1}_{>0}=\emptyset$.

\noindent (b): $(X_0,p)$ rational $\iff$ $\sigma_f^{\text{min}}>1$ $\iff$ $\text{int}(\NP^c)\cap \ZZ^{n+1}_{>0}=\emptyset$ and $1\notin |\sigma_f|$. But $1\notin |\sigma_f|$ $\iff$ $h^{1,j}_{\van,0}$ ($\forall j$) $\iff$ $h^{n-j+1,n}_{\van,0}$ ($\forall j$) $\underset{\text{Thm. }\ref{thm2.5B}}{\iff}$ $\Gamma\cap\ZZ^{n+1}_{>0}=\emptyset.$
\end{proof}

In fact, we can strengthen the entire statement of Theorem \ref{thm2.5A}.  It follows from Theorem \ref{t5.2b} below (or \cite{Sa6}) that, assuming only {\bf (i)}-{\bf (ii)}, we have $\sigma_f^{\text{min}}=c \iff \underline{1}\in \partial(c\NP)$.

\subsection{Brieskorn lattice approach} \label{S5.2}

Begin with the local system $\HH :=\{ H^n(\mathfrak{F}_t,\ZZ)\}_{t\neq 0}$ and connection $(\ch:=\HH\otimes \co_{\disc^*},\nabla)$ determined by the Milnor fiber.\footnote{In what follows, we shall write $\partial_t$ [resp. $\delta_t$] for $\nabla_{\partial_t}$ [resp. $\nabla_{t\partial_t}$].  (Note that on $(\jmath_* \ch)_0$, $\delta_t$ computes $\mathrm{Res}(\nabla)$.)} Writing $\mathsf{e}(\cdot)=e^{2\pi\sqrt{-1}(\cdot)}$, we decompose $\HH_{\CC}=\oplus_{\lambda\in[0,1)}\HH^{e(\lambda)}$ according to eigenvalues of $T^{\text{ss}}$.  Regarding $\mathsf{e}$ as a map $\mathfrak{H} \overset{\mathsf{e}}{\to} \disc^* \overset{\jmath}{\hookrightarrow} \disc$, define for each $\alpha\in\QQ$ a $\CC$-subspace $C^{\alpha}\subset (\jmath_*\ch)_0$ by the image of\footnote{The point is that $e^{\alpha\log(t)-\frac{\log(t)}{2\pi\sqrt{-1}}N}\nu$ is $T=T^{\text{ss}}e^N$-invariant hence decends to $\disc^*$.}
\begin{equation}\label{eq5.2.0}
\begin{split}
\psi_{\alpha}:\, \Gamma(\mathsf{e}^*\HH^{\mathsf{e}(-\alpha)}&)\,\to\, (\jmath_* \ch)_0 \\
&\nu \mapsto t^{\alpha-\frac{N}{2\pi\sqrt{-1}}}\nu. 
\end{split}
\end{equation}
Since one clearly has
\begin{equation}\label{eq5.2.1}
(\delta_t - \alpha)\psi_{\alpha}\nu \;=\; -\tfrac{N}{2\pi\sqrt{-1}} \psi_{\alpha} \nu \, ,
\end{equation}
$C^{\alpha}$ is the generalized eigenspace $\tilde{E}_{\alpha}(\delta_t)$.

Now $(\jmath_* \ch)_0$ is too large:  since we work in the analytic topology, it will include germs of non-meromorphic sections.  Writing $\co:=\co_{\disc,0}=\CC\{t\}$ for the convergent power series ring and $\mathcal{K}:=\CC\{t\}[t^{-1}]$ for its fraction field, we shall work inside the $\mu_f$-dimensional $\mathcal{K}$-vector space $\mathcal{M}:=\oplus_{\lambda\in[0,1)}\mathcal{K} C^{\lambda}$; this may also be viewed as a $\CC[t,\partial_t]$-module, with $t$ [resp. $\partial_t$] restricting to isomorphisms $C^{\alpha}\overset{\cong}{\to}C^{\alpha+1}$ [resp. $C^{\alpha+1}\overset{\cong}{\to}C^{\alpha}$, if $\alpha\neq -1$]. An $\co$-submodule of rank $\mu_f$ in $\mathcal{M}$ is called a \emph{lattice}.  The obvious examples are the terms of the $\cv$\emph{-filtration}
$$\cv^{\beta}:=\oplus_{\beta\leq \alpha<\beta+1}\co C^{\alpha}\supseteq \oplus_{\beta<\alpha\leq\beta+1}\co C^{\alpha} =:\cv^{>\beta}.$$
In particular, $\ch_e := \cv^{>-1}$ is the \emph{canonical lattice}. Its quotient by $\tilde{\ch}_e := \cv^{>0} = t\ch_e = \partial_t^{-1}\ch_e$ identifies naturally with $V_f := H_{\van}^n (X_t)$.

Two less obvious lattices are given by the images of 
\begin{equation}\label{eq5.2.2}
\begin{split}
\cb :=\frac{\Omega_{\cx,p}^{n+1}}{d\Omega_{\cx,p}^{n-1}\wedge df}& \overset{\imath}{\longhookrightarrow} \mathcal{M}\\
&\mspace{-20mu} [\Omega] \longmapsto \int_{(\cdot)} \left. \frac{\Omega}{df}\right|_{X_t} = \int_{(\cdot)} \mathrm{Res}_{X_t}\left( \frac{\Omega}{f-t}\right)
\end{split}
\end{equation}
and
\begin{equation}\label{eq5.2.3}
\begin{split}
\tcb :=\frac{\Omega_{\cx,p}^{n}}{\Omega_{\cx,p}^{n-1}\wedge df + d\Omega_{\cx,p}^{n-1}}\cong \frac{\Omega^{n}_{\cx/\disc,\,p}}{d\Omega^{n-1}_{\cx/\disc,\,p}}& \overset{\tilde{\imath}}{\longhookrightarrow} \mathcal{M}\\
&\mspace{-20mu} [\omega] \longmapsto \int_{(\cdot)} \left. \omega\right|_{X_t}.
\end{split}
\end{equation}
(Here the integrals are to be interpreted as over $\{\gamma_t\}\in \{H_n(\mathfrak{F}_t,\ZZ)\}_{t\neq 0}$.) The inclusion $\wedge df:\,\tcb\hookrightarrow \cb$ clearly satisfies $\imath([\omega\wedge df])=\tilde{\imath}(\omega)$, so that
\begin{equation}\label{eq5.2.4}
\Omega_f :=\frac{\co_{n+1}}{J_f}:=\frac{\co_{\CC^{n+1},\uo}}{(\partial_{z_1}f,\ldots,\partial_{z_{n+1}}f)} \overset{\cong}{\underset{\wedge d\uz}{\to}} \frac{\Omega^{n+1}_{\CC^{n+1},\uo}}{\Omega^n_{\CC^{n+1},\uo}\wedge df}
\cong \frac{\Omega^{n+1}_{\cx,p}}{\Omega^n_{\cx,p}\wedge df}\cong \frac{\cb}{\tcb\wedge df}\cong \frac{\imath(\cb)}{\tilde{\imath}(\tcb)}.
\end{equation}
Every $\Omega$ is a $d\omega$, and by the Cartan formula $\imath(d\omega)=\partial_t \tilde{\imath}(\omega)$, so that $\imath(\cb)=\partial_t\tilde{\imath}(\tcb)$. Since $\lim_{t\to 0}\int_{\gamma_t}\left.\omega\right|_{X_t} =0$ on any vanishing cycle, $\tilde{\imath}(\tcb)\subset \cv^{>0}$ and $\imath(\cb)\subset \cv^{>-1}$, and $\partial_t,\,\partial_t^{-1}$ actually induce an isomorphism $\imath(\cb) \cong \tilde{\imath}(\tcb)$. But unlike $\ch_e$, the Brieskorn lattice $\imath(\cb)$ is (in general) not a sum of $C^{\alpha}$'s (or even a sum of subspaces of $C^{\alpha}$'s), so that we may have $\delta_t \imath(\cb) \not\subset \imath(\cb)$ and $t\imath(\cb)\not\subset \tilde{\imath}(\cb)$.  Since multiplication by $t$ on $\imath(\cb)$ and $\tilde{\imath}(\tcb)$ corresponds to multiplication by $f$ on $\Omega_f$, the latter may therefore be nontrivial.

To relate $V_f$ to $\Omega_f$, we cannot just consider the natural map $\frac{\imath(\cb)}{\tilde{\imath}(\tcb)} \to \frac{\ch_e}{\tilde{\ch}_e}$, which is almost never an isomorphism and frequently zero. Instead, we introduce a filtration by sublattices
\begin{equation} \label{eq5.2.5}
\cf^k \ch_e := (\partial_t^{n-k}\imath(\cb))\cap \ch_e
\end{equation}
on $\ch_e$ (and $\tilde{\ch}_e$), with the aim of relating its graded pieces on $V_f\cong \oplus_{\lambda\in [0,1)}\gr_{\cv}^{-\lambda}\ch_e$ to those of\footnote{Note that $\gr_{\cv}^{\alpha}:=\frac{\cv^{\alpha}}{\cv^{>\alpha}}$.} $\cv^{\bullet}$ (pulled back to $\cb,\,\tcb$ via $\imath,\,\tilde{\imath}$) on $\Omega_f$. Writing $\cb^{\alpha}:=\gr_{\cv}^{\alpha}\cb \overset{\bar{\imath}}{\hookrightarrow}\gr_{\cv}^{\alpha}\ch_e$ [resp. $\tcb^{\alpha}:=\gr_{\cv}^{\alpha}\tcb$] and $\alpha=p+\lambda$, we have isomorphisms \footnote{Warning: $C^{\alpha}$ is a subspace of $\ch_e$. We write $\gr_{\cv}^{\alpha}\ch_e$ etc. to emphasize that $\bar{\imath}(\cb^{\alpha})$ is the projection to $\gr_{\cv}^{\alpha}$ of $\imath(\cb)\cap \cv^{\alpha}$; under the identification $C^{\alpha}\underset{\cong}{\to}\gr_{\cv}^{\alpha}$, this is \emph{not} the same as $\imath(\cb)\cap C^{\alpha}$ (which may be a proper subspace of $\bar{\imath}(\cb^{\alpha})$).} 
\begin{equation}\label{eq5.2.6}
\cf^p\gr^{-\lambda}_{\cv}\ch_e \underset{\partial_t^{n-p}\circ \bar{\imath}}{\overset{\cong}{\longleftarrow}} \cb^{n-\alpha}
\end{equation}
hence
\begin{equation}\label{eq5.2.7}
\gr{\cf}^p\gr^{-\lambda}_{\cv}\ch_e \underset{\partial_t^{n-p}\circ \bar{\imath}}{\overset{\cong}{\longleftarrow}} \frac{\cb^{n-\alpha}}{\tcb^{n-\alpha}\wedge df} =\gr_{\cv}^{n-\alpha} \Omega_f.
\end{equation}
Writing $\bar{\Omega}_f :=\co_{n+1}/(\partial_{z_1}f,\ldots ,\partial_{z_{n+1}}f,f),$ the basic theoretical results can then be summarized as follows:

\begin{thm}\label{t5.2a}
\begin{enumerate}[label=(\alph*)]
\item The Milnor and spectral numbers are given by $\mu_f = \dim_{\CC}(\Omega_f)$ and $m_{\alpha}=\dim_{\CC}(\gr_{\cv}^{n-\alpha}\Omega_f)$.
\item The primitive Milnor and primitive spectral numbers satisfy $\mu_f^0\geq \dim_{\CC}(\bar{\Omega}_f)$ and $m_{\alpha}^0 \geq \dim_{\CC}(\gr_{\cv}^{n-\alpha}\bar{\Omega}_f)$.
\item \cite{Va4} The Jordan block structures of $f\cdot$ on\footnote{By this we mean the map $\gr_{\cv}\Omega_f \to \gr_{\cv[1]}\Omega_f$ of associated gradeds induced by multiplication by $f$ -- that is, the direct sum (over $\alpha\in \QQ$) of maps $\gr_{\cv}^{\alpha}\to \gr_{\cv}^{\alpha+1}$.} $\gr_{\cv}\Omega_f$ and $N$ on $V_f$ agree.
\end{enumerate}
\end{thm}

\begin{proof}
By \cite{ScSt}, the above-mentioned identification of $\ch_e/\tilde{\ch}_e$ with $V_f$ maps $\cf^{\bullet}$ from \eqref{eq5.2.5} to the Hodge filtration $F^{\bullet}$.  Since it also identifies $\gr_{\cv}^{-\lambda}\ch_e$ with $V_{f,\lambda}$, and $m_{p+\lambda}=\dim(\gr_F^p V_{f,\lambda})$, (a) follows.  Assuming (c), we have (b) since we defined $\mu^0_f :=\dim(\mathrm{coker}\{N:V_f\to V_f (-1)\})$. Finally, for (c), let $\Omega\in \cb^{n-p-\lambda}$ and $\eta:=\partial_t^{n-p}\bar{\imath}(\Omega) \in \cf^p \gr_{\cv}^{-\lambda}$; then $f \Omega \in \cb^{n-\alpha+1}$ and 
\begin{equation*}
\begin{split}
\partial_t^{n-p+1} \imath(f\Omega)&=\partial_t^{n-p+1} t \imath(\Omega) \\
&= [(n-p+1)+\delta_t]\partial_t^{n-p}\imath(\Omega) \\
&=[(n-(p+\lambda)+1) -\tfrac{N}{2\pi\sqrt{-1}}]\eta \in \cf^{p-1}\gr_{\cv}^{-\lambda}\ch_e
\end{split}
\end{equation*}
by \eqref{eq5.2.1}. Hence 
\begin{equation}\label{eq5.2.8}
\xymatrix{\gr_F^p V_{f,\lambda} \ar [d]_{\frac{-N}{2\pi\sqrt{-1}}} & \gr_{\cv}^{n-p-\lambda}\Omega_f \ar [d]^{f\cdot} \ar [l]^{\partial_t^{n-p}\circ \imath}_{\cong} \\
\gr^{p-1}_F V_{f,\lambda} & \gr_{\cv}^{n-p-\lambda+1}\Omega_f \ar [l]_{\cong}^{\partial_t^{n-p+1}\circ \imath}}
\end{equation}
commutes.
\end{proof}

\begin{rem}\label{rem5.9}
\begin{enumerate}[label=(\alph*)]
\item In light of the Theorem, \eqref{eq5.2.6} expresses the Hodge flag $F^{\bullet}$ on each $V_{f,\lambda}$ (hence on all of $V_f$) in terms of $(\imath(\cb),\gr_{\cv}^{\bullet})$. As mentioned in previous sections, $(\imath(\cb),\cv^{\bullet})$ contains strictly more information.  See \cite[$\S$III.3]{Kul}, \cite[$\S\S$4-6]{Her}, and \cite[$\S$3]{Sa3} for examples.
\item We really do need to use $\CC\{\uz\}$ (or $\CC[\uz]_{(\uz)}$) in defining $\Omega_f$.  In contrast to the quasi-homogeneous case, the affine hypersurfaces $Z_t :=\{f(\uz)=t\}\subset \CC^{n+1}$ need not be homeomorphic to the Milnor fibers $\mathfrak{F}_t$ they contain. In particular, the (compatible) homomorphisms $R_f:=\frac{\CC[\uz]}{(\{\partial_{z_i}f\})} \to \Omega_f$ and $H^n(Z_t)\to H^n(\mathfrak{F}_t)$ need not be isomorphisms. Already for SQH deformations of QH singularities (e.g. $E_{12},N_{16}$) we find that $\dim_{\CC}R_f \neq \dim_{\CC}\Omega_f =\mu_f$. 
\item We must be equally careful with the meaning of $f\cdot$ on $\gr_{\cv}\Omega_f$.  Its rank can easily be strictly smaller than that of $f\cdot$  on $\Omega_f$, which is the reason for the \emph{in}equalities in Thm. \ref{t5.2a}(b).  In particular, in a $\mu$-constant deformation of the singularity (which we recall preserves the weighted spectrum), the \emph{Tjurina number} $\tau_f :=\dim_{\CC}(\bar{\Omega}_f)$ need not remain constant; indeed, this is precisely what happens in SQH deformations of QH singularities.  For example, if $f_s := z^2 + x^5 + y^5 + s x^3 y^3$ ($|s|$ small) is the SQH-deformation of an $N_{16}$ singularity, then $f_s \cdot$ is zero on $\gr_{\cv}\Omega_{f_s}$ for all $s$, but has rank $1$ on $\Omega_{f_s}$ for $s\neq 0$ (with the result that $\tau_{f_s}$ is $16$ for $s=0$ and $15$ for $s\neq 0$).\footnote{The tricky point here is that, for $s\neq 0$, the nontrivial bit of $f_s\cdot$ maps from $\gr_{\cv}^{\tfrac{-1}{10}}$ to $\gr_{\cv}^{\tfrac{11}{10}}$ (rather than to $\gr_{\cv}^{\tfrac{9}{10}}$, which is all ``the map induced on $\gr_{\cv}\Omega_{f_s}$'' is allowed to detect).}
\end{enumerate}
\end{rem}

\noindent Given $g=\sum_{\um\in\ZZ_{\geq 0}^{n+1}}a_{\um}\uz^{\um}\in \CC\{\uz\}$, write $\nu_f(g):=\min\{\mathsf{h}(\um)\mid a_{\um}\neq 0\}$. Using the identification
\begin{equation*}
\cb\cong \frac{\Omega^{n+1}_{\CC^{n+1},\uo}}{d\Omega^n_{\CC^{n+1},\uo}\wedge df} \overset{\cong}{\underset{\cdot d\uz}{\longleftarrow}}\frac{\co}{(\{\partial_{z_i}G\partial_{z_j}f - \partial_{z_j}G\partial_{z_i}f\})}
\end{equation*}
we define the \emph{Newton filtration} on $\cb$ (hence on $\Omega_f$) by
\begin{equation*}
\cn^{\alpha}\cb :=\left\{ [g\,d\uz]\mid \nu_f(z_1 \cdots z_{n+1}g)-1\geq \alpha\right\};
\end{equation*}
replacing the $\geq$ by $>$ gives $\cn^{>\alpha}\cb$.

Now we define several types of ``Poincar\'e polynomial'':
\begin{itemize}
\item for any rational filtration $\mathcal{M}^{\bullet}$ on $\Omega_f$, $$P_{\mathcal{M}^{\bullet}}(u):=\sum_{\alpha\in \QQ}\dim(\gr_{\mathcal{M}}^{\alpha}\Omega_f)u^\alpha\,;$$
\item for any face $\tau$ of $\Gamma$, $$P_{C_{\tau}}(u):=\sum_{\alpha\in \QQ} \left| \mathsf{h}^{-1}(\alpha)\cap C_{\tau}\cap\ZZ_{\geq 0}^{n+1}\right| u^{\alpha}\,; \text{ and}$$
\item writing $\sigma_f = \sum_{\alpha\in\QQ}m_{\alpha}[\alpha]$, $$P_{\sigma_f}(u):=\sum_{\alpha\in\QQ} m_{\alpha} u^{\alpha}.$$
\end{itemize}
By a result of Saito \cite{Sa6}, $\cn^{\bullet}=\cv^{\bullet}$ on $\cb$.  It follows from Theorem \ref{t5.2a} and Prop. \ref{prop1.2}(ii) that
\begin{equation*}
P_{\sigma_f}(u) = u^{n+1}P_{\sigma_f}(u^{-1})=u P_{\cv^{\bullet}}(u)=u P_{\cn^{\bullet}}(u).
\end{equation*}
Combining this with Steenbrink's calculation of $P_{\cn^{\bullet}}$ \cite{steenbrinkvan} yields

\begin{thm}\label{t5.2b}
Assuming {\bf (i)}-{\bf (ii)}, we have:
\begin{equation*}
P_{\sigma_f}(u)=(-1)^{n+1}+\sum_{\tau\subseteq \Gamma}(-1)^{n-d_{\tau}}(1-u)^{k_{\sigma}}P_{C_{\tau}}(u).
\end{equation*}
Assuming in addition that \\ {\bf (iii$'$)} \textup{cones }$C_{\tau}$\textup{ on faces of }$\Gamma$\textup{ are regular simplicial,}\\
then writing $\underline{v}^{(0)},\ldots,\underline{v}^{(d_{\tau})}$ for generators of the monoid $C_{\tau}\cap \ZZ_{\geq 0}^{n+1}$, each $$P_{C_{\tau}}(u)=\prod_{i=0}^{d_{\tau}}(1-u^{\mathsf{h}(\underline{v}^{(i)})})^{-1}.$$
\end{thm}
\noindent Notice that neither {\bf (iii)} nor {\bf (iii$'$)} implies the other; so the approaches here and in $\S$\ref{S5.1} cover different territory.

\begin{example}\label{e5.11}
For $n\geq 2$ and $\sum_{i=1}^{n+1}\frac{1}{r_i}<1$, the singularity type$$T_{\underline{r}}\;:\;\;\;\;\;f=\prod_{i=1}^{n+1}z_i + \sum_{i=1}^{n+1}z_i^{r_i}$$generates nodes\footnote{Unlike $\prod_{i=1}^{n+1}z_i$ for $n>1$, $z_1 z_2$ already defines an isolated singularity, so is analytically equivalent to $z_1 z_2 + z_1^{r_1} + z_2^{r_2}$ (if $r_1^{-1}+r_2^{-1}<1$).} ($n=1$) and $T_{p,q,r}$ singularities ($n=2$); it satisfies {\bf (i)}, {\bf (ii)}, and {\bf (iii$'$)}, but not {\bf (iii)}.\footnote{For instance, the cone on $\NP$ at $(r_1,0,0)$ in the picture below is not simplicial.}  Enumerate the $n+1$ simplices of $\Gamma$ so that $\Gamma_i$ does not meet the $x_i$-axis, and write $\gamma_i$ for its intersection with $x_i =0$:
\[ \includegraphics[scale=0.5]{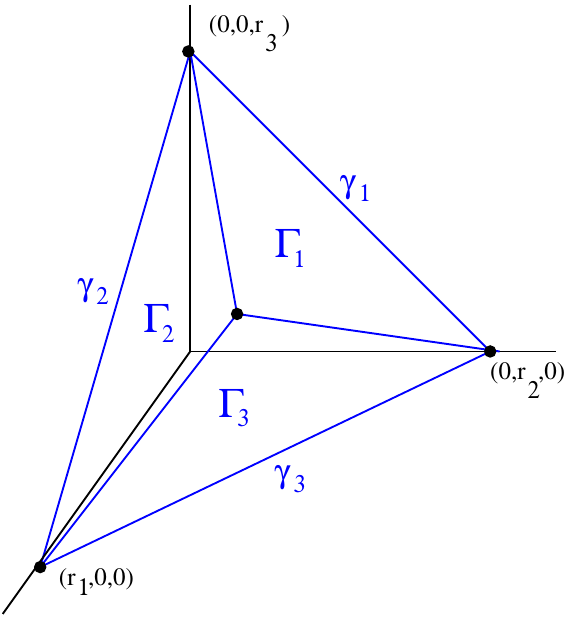} \]
We have $k_{\Gamma_I}=n+1$, $k_{\gamma_I}=n-|I|+1=d_{\Gamma_I}$, and $d_{\gamma_I}=n-|I|$, so that writing\footnote{If $J=\{1,\ldots ,n+1\}$ then $F_{\hat{J}}:=1$.} $G_r:=F_r+1:=\sum_{i=0}^{r-1}u^{\frac{i}{r}}$, $F_{\hat{J}}:=\prod_{j\notin J}F_{r_j}$, $H_k:=(-1)^k + (u-1)^{k-1}$, and $P_k:=\sum_{i=1}^{k-1} u^i$, Theorem \ref{t5.2b} yields
\begin{equation*}
\begin{split}
P_{\sigma_f}(u)&= \sum_{k=1}^{n+1}\sum_{|I|=k}\frac{\{(-1)^{k-1}(1-u)^n + (-1)^{n-k+1}(1-u)^{n-k+1}\}}{\prod_{j\notin I}(1-u^{\frac{1}{r_i}})} \\
&=\sum_{k=2}^{n+1}H_k \sum_{|I|=k}\prod_{j\notin I}G_{r_j} = \sum_{k=2}^{n+1}H_k \sum_{|J|\geq k}\binom{|J|}{k}F_{\hat{J}} \\
&= \sum_{2\leq |J|\leq n+1}P_{|J|}F_{\hat{J}}.
\end{split}
\end{equation*}
Now suppose the $\{r_i\}$ are pairwise coprime; then we can determine $\tilde{\sigma}_f$ (equiv., the form of the MHS on $V_f$) as follows. One easily shows that no two exponents occurring in the $\{F_{\hat{J}}\}$ differ by an integer. Hence the rank $\mu^0_f$ of the primitive vanishing cohomology $P_f$ (i.e. $\text{coker}(N)$) is bounded below by $c:=\sum_{\ell=2}^{n+1}\sum_{|J|=\ell}\prod_{j\notin J}(r_j - 1)$, with equality if and only if
\begin{equation} \label{eq5.2*}
V_f \otimes \RR \cong W_{n-1}(-1) \oplus \bigoplus_{\ell=2}^{n} \mathscr{P}_{n-\ell}\otimes W_{\ell-2}(-1)
\end{equation}
($W_k = \oplus_{j=0}^k \RR(-j)$ the $N$-string of length $k$, and $\mathscr{P}_k$ of weight and level $k$). On the other hand, $(f,\{\partial_{z_i} f\}_i)$ contains $\{z_i^{r_i}\}_i$, $\prod_{i=1}^{n+1} z_i$, and $\{\prod_{j\neq i} z_j \}_i$, so that $\mu^0_f \leq c$, yielding \eqref{eq5.2*}.
\end{example}

\begin{rem}\label{rem5.12}
The computer algebra system SINGULAR may be used to calculate weighted spectra of \emph{individual} polynomials defining isolated singularities.  The ``sppairs'' procedure implements an algorithm of Schulze \cite{Schulze} computing the action of $t$ on $\Omega_f$.  One should be aware of the different normalization when using this system:  an ``sppairs'' output of $(\hat{\alpha},\hat{w})$ (i.e., the Steenbrink spectrum) means $\left( (\alpha,w)=\right) (n-\hat{\alpha},\hat{w}+\langle\hat{\alpha}\rangle)$ in our notation.
\end{rem}

\section{Sebastiani--Thom formula\label{S2.6}}
The purpose of this section is to review the so-called Sebastiani--Thom formulas from the perspective of our paper. Namely, once one understands the vanishing cohomology and spectrum of lower dimensional hypersurface singularities $V(f)$ (say by using $\S$\ref{S2.2} and $\S$\ref{S2.5}), one can produce interesting new higher dimensional singularities by considering the ``join'' $V(f(\underline{x})+g(\underline{y}))$ where $f$ and $g$ are in distinct sets of variables. The topology of the resulting hypersurface was first studied by  Sebastiani--Thom \cite{ST}. This was then enhanced at the level of MHS and spectrum by Scherk--Steenbrink \cite{ScSt}. In the interest of completeness and coherence, we give a brief proof of the  Sebastiani--Thom (ST) formula, and  correct in passing some inaccuracies in \cite{ScSt} (see Remark \ref{rem-SSerr} below). Returning to the theme of our paper, we then note that ST formula can be used to produce examples of $k$-log canonical singularities via double-suspensions (cf. \S\ref{sec-ex-klog} below). 

As in $\S$\ref{S2.5}, we shall take $f\colon \CC^{m+1}\to \CC$ and $g\colon \CC^{n+1}\to \CC$ to be polynomials with isolated singularities at $\uo$.  Let $\VV_f := \{\tilde{H}^m(\mathfrak{F}_{f,s},\ZZ)\}_{s\neq 0}$ and $\VV_g :=\{ \tilde{H}^n(\mathfrak{F}_{g,t},\ZZ)\}_{t\neq 0}$ be the corresponding local systems on $\disc^*$, and $\tilde{\sigma}_f ,\tilde{\sigma}_g \in \ZZ[\QQ\times \ZZ]$ be the weighted spectra. We are interested in the weighted spectrum of the \emph{join} $f\oplus g\colon\CC^{m+n+2}\to \CC^2$ of $f$ and $g$, given by $(f\oplus g)(z_0,\ldots,z_{m+n+1}):= f(z_0,\ldots,z_m)+g(z_{m+1},\ldots ,z_{m+n+1})$.

Define a binary operation on $\ZZ[\QQ\times\ZZ]$ by bilinear extension of
\begin{equation}\label{eq6.1}
(\alpha,w)*(\beta,\omega):=(\alpha+\beta,w+\omega+\langle\alpha\mid \beta \rangle )\;,
\end{equation}
where
\begin{equation}\label{eq6.2}
\langle\alpha\mid\beta\rangle := 1+\langle \alpha+\beta\rangle -\langle\alpha\rangle - \langle \beta \rangle = \left\{ \begin{array}{cc} 0 & \alpha\text{ or }\beta\in \ZZ \\ 1 & \alpha,\beta,\alpha+\beta\notin \ZZ \\ 2 & \alpha,\beta\notin\ZZ,\, \alpha+\beta\in\ZZ. \end{array} \right.
\end{equation}
The main result of this section will be the
\begin{thm}[ST formula] \label{t6}
$\tilde{\sigma}_{f\oplus g}= \tilde{\sigma}_f * \tilde{\sigma}_g .$
\end{thm}
\begin{rem}\label{rem-SSerr}
A similar result of Scherk and Steenbrink (cf. \cite[Thm. 8.11]{ScSt} or \cite[Thm. 8.7.8]{Kul}) is incorrect in the case when $\alpha\in\ZZ$ and $\beta \notin\ZZ$ or vice versa (where they would have $\langle \alpha \mid \beta \rangle =1$). On the other hand, the more general result of Steenbrink and Nemethi \cite[Thm. 11.2]{StN} specializes to the correct statement in this situation. \end{rem}

In the original (topological) version due to Sebastiani and Thom \cite{ST}, the Theorem simply says that 
\begin{equation}\label{eq6.3}
H^{n+m+1}(\mathfrak{F}_{f\oplus g})\cong \tilde{H}^m(\mathfrak{F}_f)\otimes \tilde{H}^n(\mathfrak{F}_g)\, ,
\end{equation}
with monodromy given by $T_f \otimes T_g$.  Indeed, $\mathfrak{F}_{f\oplus g,\epsilon}\overset{\pi_{\mathsf{D}}}{\longrightarrow}\mathsf{D}$ is the restriction of $f\boxtimes g\colon \CC^{m+n+2}\to \CC^2$ to a disk $\mathsf{D}\subset \{s+t=\epsilon\}\subset \CC^2$ intersecting the axes in $p_1,p_2$.  Writing $j^{\cdots}\colon \mathsf{D}^{\setminus \cdots} \hookrightarrow \mathsf{D}$ and $\mathsf{V}_{f,g} := \left. \VV_f \otimes \VV_g \right|_{\mathsf{D}^{\setminus \{p_1,p_2\}}}$, since $\tilde{H}^m(\mathfrak{F}_{f,s})$ [resp. $\tilde{H}^n(\mathfrak{F}_{g,t})$] actually \emph{vanishes} at $s=0$ [resp. $t=0$], we have
\begin{equation}\label{eq6.4}
R\pi^{\mathsf{D}}_* \ZZ \simeq \jmath_!^{p_1,p_2} \mathsf{V}_{f,g}[-n-m] \oplus \jmath^{p_1}_! \VV_f [-m] \oplus \jmath_!^{p_2} \VV_g [-n] \oplus \ZZ.
\end{equation}
Applying $\HH^{n+m+1}(\mathsf{D},\text{---})$ to \eqref{eq6.4} gives $H^{n+m+1}(\mathfrak{F}_{f\oplus g,\epsilon},\ZZ)\cong H^1(\mathsf{D},\jmath_!^{p_1,p_2}\mathsf{V}_{f,g})$, which is isomorphic to a fiber $v$ of $\mathsf{V}_{f,g}$ by Mayer-Vietoris.

Owing to the lack of a canonical MHS on $H^1(\mathsf{D},\mathsf{V}_{f,g})$, in order to recover $\tilde{\sigma}_{f\oplus g}$ we must both \textbf{(a)} replace $\mathsf{D}$ by the exceptional divisor $\bar{D}\cong \PP^1$ of the blow-up $\rho \colon \widetilde{\CC^2}\twoheadrightarrow \CC^2$ at $\uo$, and \textbf{(b)} promote the local system $\mathsf{V}_{f,g}$ to a mixed Hodge module $\jmath_!^{0,\infty} R j_*^1 \VV$ on $\bar{D}$.  Its underlying variation of MHS $\VV$ over $D = {\PP^1 }^{\setminus \{0,1,\infty\} }$ has fibers isomorphic to $V=V_f\otimes V_g$; see the proof below.  Although it may look more opaque, the heuristic idea of how a weight-graded piece $U=\gr^W_{w+\omega} V$ contributes to weights of $V_{f\oplus g}\cong \HH^1(\bar{D},\jmath_!^{0,\infty} R j_*^1 \VV) \cong H^1(\PP^1\setminus \{1\},\{0,\infty\};\VV)$ is uncomplicated.  The monodromies about $0$ and $\infty$ are given by the (finite-order) semisimple parts $\mathsf{T}_0 = T^{\text{ss}}_g$ and $\mathsf{T}_{\infty} = T^{\text{ss}}_f$.  These produce a filtration
\begin{equation}\label{eq6.5}
U^{T_0} + U^{T_{\infty}} \subset (\mathsf{T}_1 - I)U + U^{\mathsf{T}_0} + U^{\mathsf{T}_{\infty}} \subset U
\end{equation}
whose graded pieces contribute to $\gr^W_{w+\omega}$, $\gr^W_{w+\omega+1}$, resp. $\gr^W_{w+\omega+2}$ of $V_{f\oplus g}$.  The reader will readily observe that the $\gr_{0,1,2}$ of \eqref{eq6.5} match the three cases of \eqref{eq6.2}.

\begin{proof}[Proof of Theorem \ref{t6}]
The following notion of \emph{mixed (quasi-)nilpotent orbit} will be used in what follows.  Let $(H,F^{\bullet},W_{\bullet}, T_1,T_2)$ be a MHS with two automorphisms.  Then we define a VMHS $H^{\text{nilp}}$ on $\disc^*_s \times \disc^*_t$ (with local system monodromies $T_i$) by $F_{s,t}^{\bullet}:= T_1^{-\ell(s)}T_2^{-\ell(t)}F^{\bullet}$ together with \emph{the same} $W_{\bullet}$.  Conversely, given an admissible VMHS $\mathscr{H}$ on $\disc^* \times \disc^*$, we write $\mathscr{H}_{\lm} = \psi_s \psi_t \mathscr{H}$; note that $(H^{\text{nilp}})_{\lm}\cong H$.

$\underline{\text{Step 1}}$: \emph{Reduction to the mixed nilpotent orbit of $V_f \otimes V_g$}.
Let $\pi_f \colon \cx_f \to \disc_s$ and $\pi_g \colon \cx_g \to \disc_t$ be projective fibrations with unique singularity at $p_f$ resp. $p_g$, locally analytically isomorphic to $f$ resp. $g$.  By \cite{Scherk} we may assume the canonical maps $H^m_{\lm}(X_{f,s})\to V_f$ and $H^n_{\lm}(X_{g,t})\to V_g$ are surjective. Write $p=(p_f,p_g)\underset{\imath_p}{\hookrightarrow} \cx = \cx_f \times \cx_g$, $\pi = \pi_f \boxtimes \pi_g \colon \cx \to \disc_{s,t}^2$, $\varepsilon = s+t \colon \disc^2_{s,t}\to \disc_z$, $\mathbf{H}:= R\pi_*\QQ_{\cx}[n+m+2]$, and $\mathbf{H}^{i}:=R^{i}\pi_*\mathbb{Q}_{\cx}[-i+m+n+2]$. We aim to compute
\begin{equation} \label{eq6.6}
V_{f\oplus g} :=\imath_p^*\, \pphi_{\varepsilon \circ \pi}\QQ_{\cx}[m+n+2] \cong \imath_{\uo}^* \,\pphi_{s+t}\mathbf{H}\cong \imath_{\uo}^* \,\pphi_{s+t}\mathbf{H}^{m+n},
\end{equation}
where the last isomorphism follows essentially from \eqref{eq6.4}.\footnote{The underlying (shifted) perverse sheaves for $\{\mathbf{H}^i\}_{i\neq m+n}$ are locally constant at $\uo$ in $s$ or $t$. By a change of coordinates we can thus view them as locally constant in $s+t$, whence applying $\pphi_{s+t}$ yields zero.}  Here $\mathbf{H}^{m+n}\in\mathrm{MHM}(\Delta^2)$, and we write $\mathbb{H}^{m+n}$ for the underlying perverse sheaf.

Let $\jmath \colon (\disc^*)^2 \hookrightarrow \disc^2$ denote the inclusion, $\mathscr{H}^{m+n}:=\jmath^*\mathbf{H}^{m+n}[-2]$ (a VHS on $(\Delta^*)^2$), and $T_1:=T_f\otimes I$, $T_2:=I\otimes T_g$ the monodromy operators on $\mathbf{H}^{m+n}$ (and $\mathscr{H}^{m+n}_{\text{lim}}$). Writing $(-)^T$ to denote $\ker(T_1-I)+\ker(T_2-I)$, we have short-exact sequences of MHS
\begin{equation}\label{eq6.6i}
0\to (\mathscr{H}^{m+n}_{\text{lim}})^T \to \mathscr{H}^{m+n}_{\text{lim}}\to V_f\otimes V_g\to 0
\end{equation}
and of perverse sheaves on $\Delta^2$
\begin{equation}\label{eq6.6ii}
0\to (\mathbb{H}^{m+n})^T \to \mathbb{H}^{m+n}\to \jmath_!(\VV_f\otimes \VV_g)[2]\to 0.
\end{equation}

Now by an argument of M. Saito \cite[p.~713]{Sa-steen} we may replace $\mathbf{H}^{m+n}$ by its \emph{Verdier specialization} $\hat{\mathbf{H}}$ at $\uo$ [op.~cit., p.~708].  This is an object of $\mathrm{MHM}(\CC^2)$ with\footnote{We use $\jmath$ also for the inclusion $(\CC^*)^2\hookrightarrow \CC^2$, as well as $s,t$ for coordinates.}
$\jmath^*\hat{\mathbf{H}}\cong (\mathscr{H}_{\text{lim}}^{m+n})^{\text{nilp}}[2]$, constructed via deformation to the normal cone $C_{\uo}\Delta^2\cong \CC^2$.  Writing $\hat{\mathbf{V}}:=\hat{\mathbf{H}}/(\hat{\mathbf{H}})^T$, we have $\jmath^*\hat{\mathbf{V}}\cong (V_f\otimes V_g)^{\text{nilp}}[2]$ by [op.~cit., (3.8)] and \eqref{eq6.6i}.  Moreover, since the Verdier specialization at $\uo$ commutes with restriction to the coordinate axes, \eqref{eq6.6ii} implies
\begin{equation}\label{eq6.6iii}
\hat{\mathbf{V}}\cong\jmath_!(\mathbb{V}_f\otimes \mathbb{V}_g)[2]
\end{equation}
as a perverse sheaf.\footnote{We may view $\VV_f\otimes \VV_g$ as a local system on $(\CC^*)^2$ in the obvious way.} It follows that in $\mathrm{MHM}(\CC^2)$ we have
\begin{equation}\label{eq6.6iv}
\hat{\mathbf{V}}\cong \jmath_!(V_f\otimes V_g)^{\text{nilp}}[2].
\end{equation}
Finally, $(\hat{\mathbf{H}})^T$ is sum of objects constant in $s$ and $t$, which (via change of coordinates) we may view as constant in $s+t$.  It follows that $\pphi_{s+t}(\hat{\mathbf{H}})^T=\{0\}$, and so
\begin{equation}\label{eq6.7}
V_{f\oplus g} \cong \imath_{\uo}^* \,\pphi_{s+t} \hat{\mathbf{H}} \cong \imath^*_{\uo} \,\pphi_{s+t} \hat{\mathbf{V}}.
\end{equation}

$\underline{\text{Step 2}}$:  \emph{Reduction to $H^1$ of a MHM on $\PP^1$}.
By \eqref{eq6.6iv}, $\hat{\mathbf{V}}$ has trivial ``stalk'' at $\uo$, and so its pullback $\tilde{\mathbf{V}}:=\rho^* \hat{\mathbf{V}}$ under the blow-up satisfies $R\rho_* \tilde{\mathbf{V}} \cong \hat{\mathbf{V}}$ and $\imath_{\bar{D}}^* \tilde{\mathbf{V}}=0$.  Writing $u:=\rho^*(s+t)$, and $C_1$ [resp. $C_0$, $C_{\infty}$] for the proper transform of $s+t=0$ [resp. $t=0$, $s=0$], we have $(u)=\bar{D}+C_1$.  Hence $\pphi_u \tilde{\mathbf{V}}$ has support in $\bar{D}\cup C_1$ (but is in fact $0$ off $\bar{D}$), and 
\begin{equation}\label{eq6.8}
V_{f\oplus g} \cong \imath_{\uo}^* R \rho_* \pphi_u \tilde{\mathbf{V}} \cong R \Gamma_{\bar{D}}\imath_{\bar{D}}^* \pphi_u \tilde{\mathbf{V}} \cong R\Gamma_{\bar{D}}\imath^*_{\bar{D}}\ppsi_u\tilde{\mathbf{V}}
\end{equation}
with the last isomorphism because $\imath_{\bar{D}}^* \tilde{\mathbf{V}}=0$.

Let $w$ be a coordinate on $\bar{D}\cong \PP^1$ with $w|_{\bar{D}\cap C_{\alpha}}=\alpha$, and write (by abuse of notation) $\jmath^S \colon \bar{D}^{\setminus S\cup S'} \hookrightarrow \bar{D}^{\setminus S'}$ for any $S,S'\subset \bar{D}$ finite.  By \cite[Lemma 4.7]{Sa-steen}, $\imath_{\bar{D}}^* \ppsi_u \tilde{\mathbf{V}} \cong R \jmath_*^1 \imath_{\bar{D}^{\setminus \{1\} }}^* \ppsi_u \tilde{\mathbf{V}}$; and evidently $\imath^*_{\bar{D}^{\setminus\{1\}}} \ppsi_u \tilde{\mathbf{V}}\cong \jmath^{0,\infty}_! \imath_D^* \ppsi_u \tilde{\mathbf{V}}$, where $\mathscr{V} := (\imath^*_D \ppsi_u \tilde{\mathbf{V}})[-1]$ is a VMHS with fiber $V_f \otimes V_g$ and monodromies $I\otimes T_g$, $T_f^{-1}\otimes T_g^{-1}$, $T_f\otimes I$ about $w=0,1,\infty$ respectively. Conclude that 
\begin{equation}\label{eq6.9}
V_{f\oplus g}\cong R\Gamma_{\bar{D}}\jmath_!^{0,\infty} R \jmath_*^1 \mathscr{V}[1] \cong H^1(\bar{D},\jmath_!^{0,\infty} R\jmath_*^1 \mathscr{V}).
\end{equation}

$\underline{\text{Step 3}}$:  \emph{Calculation of the $(p,q)$-types}.
For this purpose we may replace $\mathscr{V}$ in \eqref{eq6.9} by $\gr^W \mathscr{V}$, and thus $T_f ,T_g$ by $T_f^{\text{ss}},T_g^{\text{ss}}$.  This yields a decomposition into rank-1 subobjects
\begin{equation}\label{eq6.10}
\gr^W \mathscr{V}_{\CC} \cong \oplus (\underset{\chi^{\alpha,\beta}_{w,\omega}}{\underbrace{\chi_w^{\alpha}\otimes \chi^{\beta}_{\omega}}})^{\oplus m_{\alpha,w}m_{\beta,\omega}}
\end{equation}
with Hodge type $(p_0,q_0)=\left(\lfloor\alpha\rfloor + \lfloor \beta\rfloor,w+\omega-(\lfloor \alpha\rfloor + \lfloor\beta \rfloor)\right)$ and monodromies $\mathbf{e}(\beta),\mathbf{e}(-\alpha-\beta),\mathbf{e}(\alpha)$ about $0,1,\infty$.

Write $\lambda_{\alpha}=\{\alpha\},\, \lambda_{\beta}=\{\beta\}\in [0,1)$, and $\Sigma = \{0,1,\infty\}$.  We must compute the $(p,q)$-types of the $H^1(\bar{D},\jmath^{0,\infty}_! R\jmath_*^1 \chi^{\alpha,\beta}_{w,\omega})$, or equivalently the ``supplement'' $(p-p_0,q-q_0)$ given by the type of 
\begin{equation}\label{eq6.11}
V_{\lambda_{\alpha},\lambda_{\beta}} := H^1(\bar{D},\jmath_!^{0,\infty} R \jmath_*^1 \chi^{\lambda_{\alpha},\lambda_{\beta}}_{0,0})\cong \HH^1(\bar{D},\Omega^{\bullet}\langle\Sigma\rangle\otimes \mathcal{L}).
\end{equation}
Here $(\mathcal{L},\nabla)$ is the extension of the line bundle underlying $\chi^{\lambda_{\alpha},\lambda_{\beta}}_{0,0}$ to $\bar{D}$ in such a way that the eigenvalues of $\mathrm{Res}(\nabla)$ are in $[0,1)$ at $w=1$ and $(0,1]$ at $w=0,\infty$, while being $\equiv -\lambda_{\alpha},\lambda_{\alpha}+\lambda_{\beta},-\lambda_{\beta}$ mod $\ZZ$ at $0,1,\infty$.  Heuristically we may write this as\footnote{as a $\mathcal{D}$-module, i.e. keeping track of the (fractional) residues of $\nabla$} $\mathcal{L}\cong \co_{\bar{D}}(-(1-\lambda_{\beta})[0]-\{\lambda_{\alpha}+\lambda_{\beta}\}[1]-(1-\lambda_{\alpha})[\infty])$; since $\{\lambda_{\alpha}+\lambda_{\beta}\}=\lambda_{\alpha}+\lambda_{\beta}-\epsilon$ with $\epsilon = 0$ or $1$, we obtain\footnote{as a line bundle} $\mathcal{L}\cong \co_{\bar{D}}(\epsilon -2)=\co_{\bar{D}}(-2)$ resp. $\co_{\bar{D}}(-1)$.  Accordingly, \eqref{eq6.11} becomes
\begin{equation}\label{eq6.12}
\left\{ 
\begin{array}{cc}
\mathbf{(i)} & H^1(\co_{\bar{D}}(\mathcal{L}))\cong H^1\left(\co_{\bar{D}}(-(1-\lambda_{\beta})[0]-(\lambda_{\alpha}+\lambda_{\beta})[1]-(1-\lambda_{\alpha})[\infty])\right) \\
\;\;\text{resp.} & \\
\mathbf{(ii)} & H^0(\Omega^1_{\bar{D}}\langle\Sigma\rangle(\mathcal{L}))\cong H^0\left( \Omega^1_{\bar{D}}(\lambda_{\beta}[0]+(1-\lambda_{\alpha}-\lambda_{\beta})[1]+\lambda_{\alpha}[\infty]) \right).
\end{array}
\right.
\end{equation}
In case \textbf{(i)}, we have $\lambda_{\alpha}+\lambda_{\beta}<1$, and the space has Hodge type $(0,1)$ unless $\lambda_{\alpha}$ or $\lambda_{\beta} =0$, in which case the type is $(0,0)$ (because \textbf{(i)} will then be the image of $H^0(\co_p)$ for $p=0$ or $\infty$).  In case \textbf{(ii)}, we have $\lambda_{\alpha}+\lambda_{\beta}\geq 1$, and $\lambda_{\alpha},\lambda_{\beta}>0$.  If $\lambda_{\alpha}+\lambda_{\beta}=1$ then \textbf{(ii)} is isomorphic to $H^0(\co_{\{1\}})(-1)$ by a residue map, and the type is $(1,1)$; otherwise, it is $(1,0)$.  This matches up with the rules \eqref{eq6.1}-\eqref{eq6.2} and completes the proof.
\end{proof}

\subsection{Application to $k$-log-canonical singularities}\label{sec-ex-klog} Turning to a few simple applications, let $f\colon \CC^{m+1}\to \CC$ be as above, with $\tilde{\sigma}_f = \sum_{\alpha,w} m_{\alpha,w} [(\alpha,w)]$, and take $n=0$ and $g(z):=z^r$.  Writing $\Sigma f := f\oplus z^2$ [resp. $\Sigma_{_{(r)}} f := f \oplus z^r$] for the [\emph{generalized}] \emph{suspension} of $f$, and noting that $\tilde{\sigma}_{z^r} = \sum_{i=1}^{r-1} [(\tfrac{i}{r},0)]$, Theorem \ref{t6} gives
\begin{equation}\label{eq6.13}
\tilde{\sigma}_{\Sigma_{_{(r)}}f} = \sum_{\alpha,w}\sum_{i=1}^{r-1} m_{\alpha,w}[(\alpha+\tfrac{i}{r},w+\langle \alpha \mid \tfrac{i}{r}\rangle )]
\end{equation}
where $$\langle \alpha \mid \tfrac{i}{r}\rangle = \left\{ \begin{array}{cc} 0 & \alpha \in \ZZ \\ 2 & \alpha \in \ZZ - \tfrac{i}{r} \\ 1 & \text{otherwise} . \end{array} \right.$$
In particular, this yields $\sigma_{\Sigma_{_{(r)}}f}^{\text{min}} = \sigma_f^{\text{min}} + \tfrac{1}{r},$ so that (for isolated hypersurface singularities) we have:
\begin{itemize}
\item $\Sigma_{_{(r)}}$ (any $r$) sends log-canonical singularities to rational ones;
\item any double suspension $\Sigma^2 f$ gives a rational singularity; and
\item $\Sigma^2 := \Sigma \circ \Sigma$ sends $k$-log-canonical singularities to $(k+1)$-log-canonical ones.
\end{itemize}
There are also easy implications for cyclic base-change, related to some of the discussion in $\S\S$7-9 of Part I.  If $f\colon \cx \to \disc$ has an isolated singularity at a smooth point $p$ of $\cx$, and $\fx$ is its base-change by $t\mapsto t^r$, then:
\begin{itemize}
\item if $(X_0,p)$ is \emph{not} log-canonical, then $(\fx,p)$ is rational [resp. log-canonical] if and only if $r$ is less than [or equal to] $\tfrac{1}{1-\sigma_f^{\text{min}}}$; while
\item if $(X_0,p)$ \emph{is} ($k$-)log-canonical, then $(\fx,p)$ is ($k$-)rational (cf. Cor. \ref{cor2.4a}) for any $r$.
\end{itemize}
Finally, for double suspensions, there is the very simplest consequence of all:  since $\tilde{\sigma}_{x^2+y^2}= [(1,2)]$, Theorem \ref{t6} yields $$\tilde{\sigma}_{\Sigma^2 f} =\sum_{\alpha,w}m_{\alpha,w} [(\alpha+1,w+2)],$$ which is to say that
\begin{equation} \label{eq6.14}
V_{\Sigma^2 f} \cong V_f (-1)
\end{equation}
as MHS.
\begin{example}\label{ex6.1}
Together with Example \ref{ex5.1}, this gives for $F=z_1^4 z_2 + z_1 z_2^4 + z_1^2 z_2^2 + z_3 ^2 + z_4^2$
that
\[ \tilde{\sigma}_f = \underset{V_f^{1,1}}{\underbrace{[(\tfrac{3}{2},2)]}}+\underset{V_f^{1,2}}{\underbrace{2[(\tfrac{5}{3},3)]+2[(\tfrac{11}{6},3)]}}+\underset{V_f^{2,2}}{\underbrace{3[(2,4)]+[(\tfrac{5}{2},4)]}}+\underset{V_f^{2,1}}{\underbrace{2[(\tfrac{13}{6},3)]+2[(\tfrac{7}{3},3)]}}, \]
which justifies the assertion in Example I.9.10.
\end{example}

\begin{rem}
In the event that $|\sigma_f| \cap \tfrac{1}{2}\ZZ = \emptyset$, the vanishing cohomology $V_{\Sigma f}$ of the suspension may be regarded as a \emph{half-twist} of $V_f$ in the sense of van Geemen \cite{vG}.  (This is consistent with the full Tate-twist in \eqref{eq6.14}.)  Without loss of generality we may assume that all $\alpha \in |\sigma_f|$ have denominator $r$ when written in lowest form; then $V_f$ is (via the action of $T^{\text{ss}}$) a $\QQ(\zeta_{\ell})$-vector space.  On each $V_{f,\frac{k}{\ell}}$ in the decomposition $V_{f,\CC}=\oplus_{\lambda \in (0,1)}V_{f,\lambda}$, $\QQ(\zeta_{\ell})$ acts through one of its complex embeddings, and we choose the CM-type $\Phi$ on $\QQ(\zeta_{\ell})$ that makes $V_{f,\Phi}=\oplus_{\lambda\in (\frac{1}{2},1)} V_{f,\lambda}$ and $V_{f,\bar{\Phi}}=\oplus_{\lambda\in(0,\frac{1}{2})}V_{f,\lambda}$.  The corresponding half-twist is defined by $(V_f)^{p,q}_{\Phi,-\frac{1}{2}} = V_{f,\Phi}^{p-1,q}\oplus V_{f,\bar{\Phi}}^{p,q-1}$, and one readily checks that $V_{\Sigma f}\cong (V_f)_{\Phi,-\frac{1}{2}}$.
\end{rem}

\bibliography{csseq2}

\providecommand{\bysame}{\leavevmode\hbox to3em{\hrulefill}\thinspace}
\providecommand{\MR}{\relax\ifhmode\unskip\space\fi MR }
\providecommand{\MRhref}[2]{%
  \href{http://www.ams.org/mathscinet-getitem?mr=#1}{#2}
}
\providecommand{\href}[2]{#2}
\begin{thebibliography}{MOPW23}

\bibitem[All03]{Al}
D.~Allcock, \emph{The moduli space of cubic threefolds}, J. Algebraic Geom.
  \textbf{12} (2003), no.~2, 201--223.

\bibitem[Ber71]{Be}
I.~N. Bern\v{s}te\u{\i}n, \emph{Modules over a ring of differential operators.
  {A}n investigation of the fundamental solutions of equations with constant
  coefficients}, Funkcional. Anal. i Prilo\v{z}en. \textbf{5} (1971), no.~2,
  1--16.

\bibitem[Bri70]{Bri}
E.~Brieskorn, \emph{Die {M}onodromie der isolierten {S}ingularit\"{a}ten von
  {H}yperfl\"{a}chen}, Manuscripta Math. \textbf{2} (1970), 103--161.

\bibitem[Cas21]{Ca}
B.~Castor, \emph{Bounding projective hypersurface singularities},
  arXiv:2110.12574, 2021.

\bibitem[CDKP22]{CDKP}
B.~Castor, H.~Deng, M.~Kerr, and G.~Pearlstein, \emph{Remarks on eigenspectra
  of isolated singularities}, to appear in Pacific Math J. (arXiv:2211.04648),
  2022.

\bibitem[Dan79]{Da}
V.~Danilov, \emph{Newton polyhedra and vanishing cohomology}, Funktsional.
  Anal. i Prilozhen. \textbf{13} (1979), 32--47.

\bibitem[DL98]{DL}
J.~Denef and F.~Loeser, \emph{Motivic {I}gusa zeta functions}, J. Algebraic
  Geom. \textbf{7} (1998), no.~3, 505--537.

\bibitem[Dol74]{Do}
I.~V. Dolgachev, \emph{Conic quotient singularities of complex surfaces},
  Funkcional. Anal. i Prilo\v{z}en. \textbf{8} (1974), no.~2, 75--76.

\bibitem[Dol82]{DoWP}
I.~Dolgachev, \emph{Weighted projective varieties}, Group actions and fields
  (Vancouver, B.C., 1981), Lecture Notes in Math, vol. 956, Springer, Berlin,
  1982, pp.~34--71.

\bibitem[Ebe99]{Eb}
W.~Ebeling, \emph{Strange duality, mirror symmetry, and the {L}eech lattice},
  Singularity theory ({L}iverpool, 1996), London Math. Soc. Lecture Note Ser.,
  vol. 263, Cambridge Univ. Press, Cambridge, 1999, pp.~xv--xvi, 55--77.

\bibitem[ES98]{ES}
W.~Ebeling and J.~H.~M. Steenbrink, \emph{Spectral pairs for isolated complete
  intersection singularities}, J. Algebraic Geom. \textbf{7} (1998), no.~1,
  55--76.

\bibitem[FL22a]{FL-Def}
R.~Friedman and R.~Laza, \emph{Deformations of singular {F}ano and
  {C}alabi-{Y}au varieties}, arXiv:2203.04823, 2022.

\bibitem[FL22b]{FL-isolated}
\bysame, \emph{The higher {D}u {B}ois and higher rational properties for
  isolated singularities}, to appear in J. Algebraic Geom. (arXiv:2207.07566),
  2022.

\bibitem[FL22c]{FL-DuBois}
\bysame, \emph{Higher {D}u {B}ois and higher rational singularities}, with an
  Appendix by M. Saito, to appear in Duke Math. J. (arXiv:2205.04729), 2022.

\bibitem[Fri84]{Friedman}
R.~Friedman, \emph{A new proof of the global {T}orelli theorem for {$K3$}
  surfaces}, Ann. of Math. (2) \textbf{120} (1984), no.~2, 237--269.

\bibitem[FS86]{FS}
R.~Friedman and F.~Scattone, \emph{Type {${\rm III}$} degenerations of {$K3$}
  surfaces}, Invent. Math. \textbf{83} (1986), no.~1, 1--39.

\bibitem[FT04]{FT}
M.~Furuya and M.~Tomari, \emph{A characterization of semi-quasihomogeneous
  functions in terms of the {M}ilnor number}, Proc. Amer. Math. Soc.
  \textbf{132} (2004), no.~7, 1885--1890.

\bibitem[Gal19]{gallardo}
P.~Gallardo, \emph{On the {GIT} quotient space of quintic surfaces}, Trans.
  Amer. Math. Soc. \textbf{371} (2019), no.~6, 4251--4276.

\bibitem[GKS21]{GKS}
P.~Gallardo, M.~Kerr, and L.~Schaffler, \emph{Geometric interpretation of
  toroidal compactifications of moduli of points in the line and cubic
  surfaces}, Adv. Math. \textbf{381} (2021), Paper No. 107632, 48.

\bibitem[GPSZ23]{GPSZ}
P.~Gallardo, G.~Pearlstein, L.~Schaffler, and Z.~Zhang, \emph{Unimodal
  singularities and boundary divisors in the {KSBA} moduli of a class of
  horikawa surfaces}, Math. Nachr. (2023), no.~00, 1--34.

\bibitem[Gre86]{Gr86}
G.-M. Greuel, \emph{Constant milnor number implies constant multiplicity for
  quasihomogeneous singularities}, Manuscripta Math. \textbf{56} (1986),
  159--166.

\bibitem[Has00]{hassett}
B.~Hassett, \emph{Local stable reduction of plane curve singularities}, J.
  Reine Angew. Math. \textbf{520} (2000), 169--194.

\bibitem[Her95]{Her}
C.~Hertling, \emph{Ein {T}orelli-{S}atz f\"{u}r die unimodalen und bimodularen
  {H}yperfl\"{a}chensingularit\"{a}ten}, Math. Ann. \textbf{302} (1995), no.~2,
  359--394.

\bibitem[JKSY22a]{JKSY-duBois}
S.-J. Jung, I.-K. Kim, M.~Saito, and Y.~Yoon, \emph{Higher {D}u {B}ois
  singularities of hypersurfaces}, Proc. Lond. Math. Soc. (3) \textbf{125}
  (2022), no.~3, 543--567.

\bibitem[JKSY22b]{JKSY}
\bysame, \emph{Hodge ideals and spectrum of isolated hypersurface
  singularities}, Ann. Inst. Fourier (Grenoble) \textbf{72} (2022), no.~2,
  465--510.

\bibitem[KK10]{KK}
J.~Koll{\'a}r and S.~J. Kov{\'a}cs, \emph{Log canonical singularities are {D}u
  {B}ois}, J. Amer. Math. Soc. \textbf{23} (2010), no.~3, 791--813.

\bibitem[KL21]{KL1}
M.~Kerr and R.~Laza, \emph{Hodge theory of degenerations, ({I}): consequences
  of the decomposition theorem, with an appendix by {M.} {S}aito}, Selecta
  Math. (N.S.) (2021), no.~4, Paper No. 71, 48.

\bibitem[KL23]{KL3}
\bysame, \emph{Hodge theory of degenerations, ({III}): vanishing-cycle calculus
  for non-isolated singularities}, preprint, 2023.

\bibitem[KLS22]{KLS}
M.~Kerr, R.~Laza, and M.~Saito, \emph{Deformation of rational singularities and
  {H}odge structure}, Algebr. Geom. \textbf{9} (2022), no.~4, 476--501.

\bibitem[KLSV18]{klsv}
J.~Koll\'{a}r, R.~Laza, G.~Sacc\`a, and C.~Voisin, \emph{Remarks on
  degenerations of hyper-{K}\"{a}hler manifolds}, Ann. Inst. Fourier (Grenoble)
  \textbf{68} (2018), no.~7, 2837--2882.

\bibitem[KM98]{KM}
J.~Koll\'ar and S.~Mori, \emph{Birational geometry of algebraic varieties},
  Cambridge Tracts in Mathematics, vol. 134, Cambridge University Press,
  Cambridge, 1998, With the collaboration of C. H. Clemens and A. Corti.

\bibitem[Kol97]{kollarpairs}
J.~Koll{\'a}r, \emph{Singularities of pairs}, Algebraic geometry---{S}anta
  {C}ruz 1995, Proc. Sympos. Pure Math., vol.~62, Amer. Math. Soc., Providence,
  RI, 1997, pp.~221--287.

\bibitem[Kol23]{Kollar-Book}
J.~Koll\'{a}r, \emph{Families of varieties of general type}, Cambridge Tracts
  in Mathematics, vol. 231, Cambridge University Press, 2023.

\bibitem[Kul98]{Kul}
V.~S. Kulikov, \emph{Mixed {H}odge structures and singularities}, Cambridge
  Tracts in Mathematics, vol. 132, Cambridge University Press, Cambridge, 1998.

\bibitem[Laz09]{La1}
R.~Laza, \emph{Deformations of singularities and variation of {GIT} quotients},
  Trans. Amer. Math. Soc. \textbf{361} (2009), no.~4, 2109--2161.

\bibitem[Laz10]{La10}
\bysame, \emph{The moduli space of cubic fourfolds via the period map}, Ann. of
  Math. (2) \textbf{172} (2010), no.~1, 673--711.

\bibitem[Lev98]{Lev}
M.~Levine, \emph{Mixed motives}, Mathematical Surveys and Monographs, vol.~57,
  American Mathematical Society, Providence, RI, 1998.

\bibitem[LO18]{log2}
R.~Laza and K.~G. O'Grady, \emph{G{IT} versus {B}aily-{B}orel compactification
  for quartic {$K3$} surfaces}, Geometry of moduli, Abel Symp., vol.~14,
  Springer, Cham, 2018, pp.~217--283.

\bibitem[LPZ18]{LPZ}
R.~Laza, G.~Pearlstein, and Z.~Zhang, \emph{On the moduli space of pairs
  consisting of a cubic threefold and a hyperplane}, Adv. Math. \textbf{340}
  (2018), 684--722.

\bibitem[MOP20]{MOP}
M.~Musta\c{t}\u{a}, S.~Olano, and M.~Popa, \emph{Local vanishing and {H}odge
  filtration for rational singularities}, J. Inst. Math. Jussieu \textbf{19}
  (2020), no.~3, 801--819.

\bibitem[MOPW23]{MOPW}
M.~Musta\c{t}\u{a}, S.~Olano, M.~Popa, and J.~Witaszek, \emph{The {D}u {B}ois
  complex of a hypersurface and the minimal exponent}, Duke Math. J.
  \textbf{172} (2023), no.~7, 1411--1436.

\bibitem[MP18]{MP18i}
M.~Musta\c{t}\u{a} and M.~Popa, \emph{Restriction, subadditivity, and
  semicontinuity theorems for {H}odge ideals}, Int. Math. Res. Not. IMRN
  (2018), no.~11, 3587--3605.

\bibitem[MP19]{MP}
\bysame, \emph{Hodge ideals}, Mem. Amer. Math. Soc. \textbf{262} (2019),
  no.~1268, v+80.

\bibitem[MP20a]{MP-Inv}
\bysame, \emph{Hodge filtration, minimal exponent, and local vanishing},
  Invent. Math. \textbf{220} (2020), no.~2, 453--478.

\bibitem[MP20b]{MP18ii}
\bysame, \emph{Hodge ideals for {$\mathbb{Q}$}-divisors, {$V$}-filtration, and
  minimal exponent}, Forum Math. Sigma \textbf{8} (2020), e19.

\bibitem[MP22]{MP-Rat}
\bysame, \emph{On $k$-rational and $k$-{D}u {B}ois local complete
  intersections}, arXiv:2207.08743, 2022.

\bibitem[MT14]{MT}
Y.~Matsui and K.~Takeuchi, \emph{Motivic {M}ilnor fibers and {J}ordan normal
  forms for {M}ilnor monodromies}, Publ. Res. Inst. Math. Sci. \textbf{50}
  (2014), 207--226.

\bibitem[NS95]{StN}
A.~N\'{e}methi and J.~H.~M. Steenbrink, \emph{Spectral pairs, mixed {H}odge
  modules, and series of plane curve singularities}, New York J. Math.
  \textbf{1} (1994/95), 149--177, electronic.

\bibitem[Pop18]{Po}
M.~Popa, \emph{{$\mathcal D$}-modules in birational geometry}, Proceedings of
  the {I}nternational {C}ongress of {M}athematicians---{R}io de {J}aneiro 2018.
  {V}ol. {II}. {I}nvited lectures, World Sci. Publ., Hackensack, NJ, 2018,
  pp.~781--806.

\bibitem[PS08]{PS}
C.~A.~M. Peters and J.~H.~M. Steenbrink, \emph{Mixed {H}odge structures},
  Ergebnisse der Mathematik und ihrer Grenzgebiete. 3. Folge. A Series of
  Modern Surveys in Mathematics., vol.~52, Springer-Verlag, Berlin, 2008.

\bibitem[Sai74]{KSa}
K.~Saito, \emph{Einfach-elliptische singularit\"aten}, Invent. Math.
  \textbf{23} (1974), 289--326.

\bibitem[Sai88]{Sa6}
M.~Saito, \emph{Exponents and {N}ewton polyhedra of isolated hypersurface
  singularities}, Math. Ann. \textbf{281} (1988), no.~3, 411--417.

\bibitem[Sai91a]{Sa-steen}
\bysame, \emph{On {S}teenbrink's conjecture}, Math. Ann. \textbf{289} (1991),
  no.~4, 703--716.

\bibitem[Sai91b]{Sa3}
\bysame, \emph{Period mapping via {B}rieskorn modules}, Bull. Soc. Math. France
  \textbf{119} (1991), no.~2, 141--171.

\bibitem[Sai93]{Sa4}
\bysame, \emph{On {$b$}-function, spectrum and rational singularity}, Math.
  Ann. \textbf{295} (1993), no.~1, 51--74.

\bibitem[Sai09]{Sa5}
\bysame, \emph{On the {H}odge filtration of {H}odge modules}, Mosc. Math. J.
  \textbf{9} (2009), no.~1, 161--191, back matter.

\bibitem[Sai16]{Sai16}
\bysame, \emph{Hodge ideals and microlocal {$V$}-filtration}, arXiv:1612.08667,
  2016.

\bibitem[Sch73]{schmid}
W.~Schmid, \emph{Variation of {H}odge structure: the singularities of the
  period mapping}, Invent. Math. \textbf{22} (1973), 211--319.

\bibitem[Sch80]{Scherk}
J.~Scherk, \emph{On the monodromy theorem for isolated hypersurface
  singularities}, Invent. Math. \textbf{58} (1980), no.~3, 289--301.

\bibitem[Sch85]{Schoen}
C.~Schoen, \emph{Algebraic cycles on certain desingularized nodal
  hypersurfaces}, Math. Ann. \textbf{270} (1985), no.~1, 17--27.

\bibitem[Sch02]{Schulze}
M.~Schulze, \emph{The differential structure of the {B}rieskorn lattice},
  Mathematical software ({B}eijing, 2002), World Sci. Publ., River Edge, NJ,
  2002, pp.~136--146.

\bibitem[SS85]{ScSt}
J.~Scherk and J.~H.~M. Steenbrink, \emph{On the mixed {H}odge structure on the
  cohomology of the {M}ilnor fibre}, Math. Ann. \textbf{271} (1985), no.~4,
  641--665.

\bibitem[ST71]{ST}
M.~Sebastiani and R.~Thom, \emph{Un r\'{e}sultat sur la monodromie}, Invent.
  Math. \textbf{13} (1971), 90--96.

\bibitem[Sta17]{Sta}
A.~Stapledon, \emph{Formulas for monodromy}, Res. Math. Sci. \textbf{4} (2017),
  42.

\bibitem[Ste77a]{St1}
J.~H.~M. Steenbrink, \emph{Intersection form for quasi-homogeneous
  singularities}, Compositio Math. \textbf{34} (1977), no.~2, 211--223.

\bibitem[Ste77b]{steenbrinkvan}
\bysame, \emph{Mixed {H}odge structure on the vanishing cohomology}, Real and
  complex singularities ({P}roc. {N}inth {N}ordic {S}ummer {S}chool/{NAVF}
  {S}ympos. {M}ath., {O}slo, 1976), Sijthoff and Noordhoff, Alphen aan den
  Rijn, 1977, pp.~525--563.

\bibitem[Ste81]{steenbrink}
\bysame, \emph{Cohomologically insignificant degenerations}, Compositio Math.
  \textbf{42} (1980/81), no.~3, 315--320.

\bibitem[Tos15]{tosatti}
V.~Tosatti, \emph{Families of {C}alabi-{Y}au manifolds and canonical
  singularities}, Int. Math. Res. Not. IMRN (2015), no.~20, 10586--10594.

\bibitem[Var80]{Va1}
A.~N. Var\v{c}enko, \emph{Gauss-{M}anin connection of isolated singular point
  and {B}ernstein polynomial}, Bull. Sci. Math. (2) \textbf{104} (1980), no.~2,
  205--223.

\bibitem[Var81a]{Va2}
\bysame, \emph{Asymptotic {H}odge structure on vanishing cohomology}, Izv.
  Akad. Nauk SSSR Ser. Mat. \textbf{45} (1981), no.~3, 540--591, 688.

\bibitem[Var81b]{Va4}
\bysame, \emph{On the monodromy operator in vanishing cohomology and the
  multiplication operator by {$f$} in the local ring}, Dokl. Akad. Nauk SSSR
  \textbf{260} (1981), no.~2, 272--276.

\bibitem[Var82a]{Va3}
\bysame, \emph{The complex exponent of a singularity does not change along
  strata $\mu =$ const}, Funkcional. Anal. i Prilo\v{z}en. \textbf{16} (1982),
  1--12.

\bibitem[Var82b]{Va5}
\bysame, \emph{A lower bound for the codimension of the stratum $\mu=$ constant
  in terms of the mixed hodge structure}, Moscow Univ. Math. Bull. \textbf{37}
  (1982), 30--33.

\bibitem[vG01]{vG}
B.~van Geemen, \emph{Half twists of {H}odge structures of {CM}-type}, J. Math.
  Soc. Japan \textbf{53} (2001), no.~4, 813--833.

\bibitem[vS20]{vS}
D.~van Straten, \emph{The spectrum of hypersurface singularities},
  arXiv:2003.00519, 2020.

\bibitem[Wan03]{wang}
C.~L. Wang, \emph{Quasi-{H}odge metrics and canonical singularities}, Math.
  Res. Lett. \textbf{10} (2003), no.~1, 57--70.

\bibitem[Wat86]{Wa}
K.~Watanabe, \emph{On plurigenera of normal isolated singularities {II}},
  Complex Analytic Singularities, North-Holland, Amsterdam, 1986, pp.~671--685.

\bibitem[Yon90]{Yo}
T.~Yonemura, \emph{Hypersurface simple {$K3$} singularities}, Tohoku Math. J.
  (2) \textbf{42} (1990), no.~3, 351--380.

\end{thebibliography}
 
\end{document}